\documentclass{amsart}

\usepackage{amsmath}
\usepackage{amsfonts}
\usepackage{amsthm}
\usepackage{mathrsfs}
\usepackage[utf8]{inputenc}
\usepackage{enumitem}
\usepackage{dsfont}
\usepackage[margin=1.5in]{geometry}
\usepackage{amssymb}
\usepackage{pifont}
\usepackage[T1]{fontenc}
\usepackage{mathtools}
\usepackage{color}
\usepackage[dvipsnames, hyperref]{xcolor}
\usepackage{hyperref, cleveref}
\usepackage{tikz}
\usepackage{float}

\hypersetup{
    colorlinks=true,
    linkcolor=Blue,
    citecolor=Blue
  }

\usepackage{todonotes}
\usepackage{comment}

\newtheorem{thm}{Theorem}[section]
\newtheorem{lem}[thm]{Lemma}
\newtheorem{prop}[thm]{Proposition}

\newtheorem{cor}[thm]{Corollary}

\theoremstyle{definition}
\newtheorem*{claim}{Claim}

\newtheorem{definition}[thm]{Definition}

\newtheorem{remark}[thm]{Remark}
\newtheorem{question}[thm]{Question}
\newtheorem{notation}[thm]{Notation}

\setlist[description]{leftmargin=0cm}

\newcommand{\eps}{\varepsilon}
\newcommand{\dee}{\ensuremath{\partial}}
\newcommand{\mc}[1]{\ensuremath{\mathcal{#1}}}
\newcommand{\ten}{\otimes}

\newcommand{\Z}{\ensuremath{\mathbb{Z}}}

\newcommand{\R}{\ensuremath{\mathbb{R}}}

\renewcommand{\H}{\ensuremath{\mathbb{H}}}
\renewcommand{\P}{\ensuremath{\mathbb{P}}}

\DeclareMathOperator{\Aut}{Aut}

\DeclareMathOperator{\Hom}{Hom}

\DeclareMathOperator{\SL}{SL}
\DeclareMathOperator{\GL}{GL}
\DeclareMathOperator{\PGL}{PGL}

\DeclareMathOperator{\Cay}{Cay}
\DeclareMathOperator{\spn}{span}

\DeclareMathOperator{\diam}{diam}

\newcommand{\identity}{\mathrm{id}}
\newcommand{\minus}{-}

\newcommand{\vindomain}{\Omega_{\mathrm{Vin}}}

\newcommand{\mindomain}{\Omega_{\min}}
\newcommand{\uvindomain}{\tilde{\Omega}_{\mathrm{Vin}}}

\newcommand{\toproot}{\mu_{1,2}}

\newcommand{\iroot}{\mu_{1,2}}

\newcommand{\udist}{d_{\mathcal{U}}}
\newcommand{\utruncl}[1]{\trunc{L}{\udist(#1)}}
\newcommand{\ulongl}[1]{\{\udist(#1)\}^L}

\newcommand{\trunc}[2]{\{#2\}_{#1}}
\newcommand{\untrunc}[2]{\{#2\}^{#1}}

\newcommand{\elt}{\gamma}

\newcommand{\udvindomain}{\tilde{\mathscr{O}}_{\mathrm{Vin}}}
\newcommand{\dvindomain}{\mathscr{O}_{\mathrm{Vin}}}

\newcommand{\daviscx}{\mathrm{D}(C,S)}
\newcommand{\walls}{\mathbf{W}}
\newcommand{\minwalls}{\mathbf{W}_{\min}}

\newcommand{\type}[1]{s(#1)}

\newcommand{\BP}{\mathrm{BP}}
\newcommand{\SLpm}{\SL^{\pm}}

\newcommand{\halfspace}{\mathbf{Hs}}
\newcommand{\halfcone}{\mathbf{Hc}}
\newcommand{\uhalfspace}{\widetilde{\halfspace}}
\newcommand{\uhalfcone}{\widetilde{\halfcone}}

\newcommand{\stable}[1]{{E^-_{#1}}}
\newcommand{\unstable}[1]{{E^+_{#1}}}
\newcommand{\stableunstable}[1]{{E^\pm_{#1}}}

\newcommand{\dproj}{d_{\mathbb{P}}}
\newcommand{\ddproj}{d^*_{\mathbb{P}}}

\DeclarePairedDelimiter\norm{\lVert}{\rVert}

\title{Cubulated hyperbolic groups admit Anosov representations}
\author{Sami Douba}
\address{Institut des Hautes Études Scientifiques, 35 route de Chartres, 91440 Bures-sur-Yvette, France}
\email{douba@ihes.fr}

\author{Balthazar Fl\'echelles}
\address{Institut des Hautes Études Scientifiques, 35 route de Chartres, 91440 Bures-sur-Yvette, France}
\email{flechelles@ihes.fr}

\author{Theodore Weisman}
\address{Department of Mathematics, University of Michigan, Ann Arbor,
  MI 48109, USA
}
\email{tjwei@umich.edu}

\author{Feng Zhu}
\address{Department of Mathematics, University of Wisconsin -- Madison, WI 53706, USA}
\email{fzhu52@wisc.edu}

\begin{document}

\maketitle

\begin{abstract}
  We prove that any hyperbolic group acting properly discontinuously and cocompactly on a
  $\mathrm{CAT}(0)$ cube complex admits a projective Anosov
  representation into $\mathrm{SL}(d, \R)$ for some $d$. More
  specifically, we show that if $\Gamma$ is a hyperbolic
  quasiconvex subgroup of a right-angled Coxeter group $C$, then a generic representation of $C$ by reflections
  restricts to a projective Anosov representation of $\Gamma$.
\end{abstract}

\tableofcontents

\section{Introduction}

The prototypical examples of Gromov-hyperbolic groups are the convex
cocompact subgroups of
$\mathrm{PO}(n,1) = \mathrm{Isom}(\mathbb{H}^n)$, where $\mathbb{H}^n$
denotes real hyperbolic $n$-space, and more generally, of
$\mathrm{Isom}(X)$, where $X$ is a rank-one symmetric space of
noncompact type. It is known that not all hyperbolic groups can arise
in this fashion: there are hyperbolic groups which do not admit
faithful representations into \emph{any} matrix group
\cite{kapovichpolygons, canary2019new, tholozan2021linearity,
  tholozan2022residually}, and even among linear hyperbolic groups
there are examples which fail to admit discrete faithful
representations in any rank-one Lie group \cite{tholozan2021linearity,
  dt2022anosov}.

Apart from the convex cocompact subgroups of linear rank-one Lie
groups, a significant source of linear hyperbolic groups is the class
of hyperbolic groups that are {\em virtually special} in the sense of
Haglund and Wise \cite{HW08}. A theorem of Agol \cite{agolVHC}
established that the virtually compact
special hyperbolic groups are precisely those hyperbolic groups that
are {\em cubulated}, that is, those that act properly discontinuously and cocompactly
on $\mathrm{CAT}(0)$ cube complexes. Many diverse examples of
cubulated hyperbolic groups can be found ``in nature'' and in the
literature: see
e.g.\ \cite{nibloreevescoxeter,wise2004cubulating,ollivier2011cubulating, bergeron2012boundary, wise2021structure, giralt}.
In fact, there are currently no known examples of noncubulated convex cocompact
subgroups of $\mathrm{PO}(n,1)$, and indeed, Wise conjectured that no
such subgroups exist \cite[Conj.~13.52]{wise2014cubical}. The
following question is also due to Wise
\cite[Prob.~13.53]{wise2014cubical}.

\begin{question}\label{pon1}
  Does every cubulated hyperbolic group embed as a convex cocompact
  subgroup of $\mathrm{PO}(n,1)$ for some $n \geq 2$?
\end{question}

There are positive results towards \Cref{pon1} for some
low-dimensional cubulated hyperbolic groups; see
e.g.\ \cite{kapovichpolygons,
  markovic2013criterion,haissinsky2015hyperbolic}.
However, the question remains open even for some of the most standard
examples of cubulated hyperbolic groups. For example, it is not known
if every hyperbolic right-angled Coxeter group embeds convex
cocompactly in $\mathrm{PO}(n,1)$ for some $n$, or even if the latter
holds for such Coxeter groups of arbitrarily large virtual
cohomological dimension. 
On the other hand, Danciger--Gu\'eritaud--Kassel \cite{DGKgeomded} and Lee--Marquis \cite{leemarquis2019} (see also \cite{DGKLM}) have shown that every hyperbolic Coxeter group admits an {\em Anosov} representation.

\subsection{Anosov representations}
Anosov representations were introduced by Labourie
\cite{labourie2006anosov}, and their theory subsequently developed by
Guichard--Wienhard \cite{gw2012anosov}, Kapovich--Leeb--Porti
\cite{klp2017characterizations, klp2018domains, KLP2018},
Gu\'eritaud--Guichard--Kassel--Wienhard \cite{ggkw2017anosov},
Bochi--Potrie--Sambarino \cite{BPS}, and many others. They have
emerged as a successful generalization of convex cocompact
representations in arbitrary rank. An Anosov representation $\rho\colon \Gamma \to
  \GL(d,\R)$ gives a quasi-isometric embedding of $\Gamma$ into the
  target group, and the class of Anosov representations is stable
  under small perturbations:  the set of Anosov representations from
  $\Gamma$ into $\GL(d,\R)$ is an open subset of
  $\Hom(\Gamma,\GL(d,\R))$ \cite{labourie2006anosov}. Anosov
  representations may also be characterized in terms of the existence
  of limit maps relating the dynamics of  $\Gamma$ acting on its
  boundary to the dynamics of the image subgroup acting on a flag
  variety \cite{klp2017characterizations}, and are closely related to
  convex cocompact actions in real projective space
  \cite{dgk2017convex}.

The class of groups admitting an Anosov representation is known to be
strictly larger than that of groups which admit a convex
cocompact embedding in rank one \cite{dt2022anosov}. However, little
is currently understood about the possible isomorphism types of groups
admitting Anosov representations. Any such group is necessarily finite-by-linear,
and is
Gromov-hyperbolic \cite{klp2017characterizations,BPS}, but as observed
by Canary \cite[Q. 50.2]{canary2021anosov}, there are no other known
restrictions:
\begin{question}
  \label{quest:linear_anosov}
  Does every linear hyperbolic group admit an Anosov representation?
\end{question}

Even though \Cref{quest:linear_anosov} is open, there are still
surprisingly few constructions available for producing concrete
examples of groups admitting Anosov representations, beyond the groups
that are already known to admit convex cocompact representations in
rank one. Apart from hyperbolic Coxeter groups, Kapovich \cite{kapovich2007convex} has shown that
fundamental groups of certain Gromov--Thurston manifolds admit Anosov
representations (note that in high dimensions, groups of the latter form are not commensurable to Coxeter groups \cite{js2003}). 
In another direction, some combination theorems for Anosov subgroups
have been established in recent years \cite{DGKexamples, dkl2019,
  dk2022, dk2023}, which can be used to prove that the class of groups
admitting Anosov embeddings into $\SL(d, \R)$ for some $d$ is closed under free
products \cite{dgk2017convex, dt2022anosov}.

In this paper, we prove:
\begin{thm}\label{thm:cubulated_anosov}
  Every cubulated hyperbolic group admits an Anosov representation.
\end{thm}

\Cref{thm:cubulated_anosov} significantly enlarges the family of
groups known to admit Anosov representations; see \Cref{sec:examples}
below for some example applications. However, our proof also provides
new examples of Anosov representations even for groups already known
to admit them, since our main theorem (\Cref{thm:mainthm} below) can
provide different Anosov representations for different cubulations of
the same group.

Before stating the result, we briefly recall the definition of an
Anosov representation. For a matrix $g \in \GL(d, \R)$ and
$1 \le k \le d$, we let $\sigma_k(g)$ denote the $k$\textsuperscript{th} largest
\emph{singular value} of $g$, i.e.\ the $k$\textsuperscript{th} largest eigenvalue of the
matrix $\sqrt{gg^T}$, counted with multiplicity.
\begin{definition}
  \label{defn:anosov}
  Let $\Gamma$ be a finitely generated group, equipped with the word
  metric $|\cdot|$ induced by a finite generating set. A
  representation $\rho\colon \Gamma \to \GL(d, \R)$ is \emph{$k$-Anosov} for
  some $1 \le k < d$ if there are constants $A, B > 0$ so that for all
  $\gamma \in \Gamma$,
  \[
    \log\left(\frac{\sigma_k(\rho(\gamma))}{\sigma_{k+1}(\rho(\gamma))}\right) \ge
    A|\gamma| - B.
  \]
\end{definition}
A $1$-Anosov representation is also called a \emph{projective} Anosov representation. We will discuss singular values and Anosov representations in further
detail in \Cref{sec:singular_values}. For now, we remark that
\Cref{defn:anosov} bears little resemblance to Labourie's original
definiton of an Anosov representation, which was stated in terms of a certain flow on a bundle associated to the representation; the equivalence of the definitions is a theorem due to Kapovich--Leeb--Porti
\cite{klp2017characterizations} and Bochi--Potrie--Sambarino \cite{BPS}. 

Our main theorem is as follows:
\begin{thm}
  \label{thm:mainthm}
  Let $(C, S)$ be a right-angled Coxeter system, let
  $\Gamma \hookrightarrow C$ be a quasiconvex embedding of a
  hyperbolic group $\Gamma$ into $C$, and let
  $\rho\colon C \to \SLpm(|S|, \R)$ be a simplicial representation of $C$
  whose Cartan matrix is \emph{fully nondegenerate}. Then the
  restriction $\rho|_\Gamma$ is $1$-Anosov.
\end{thm}

In this paper, a \emph{simplicial representation} of a right-angled
Coxeter group $C$ is a deformation of the well-known \emph{geometric
  representation} $C \to \SLpm(|S|, \Z)$ studied by Tits
\cite{bourbaki68}; see \Cref{sec:cartan_simplicial} for more
detail. The geometric representation itself essentially never
satisfies the technical condition demanded by \Cref{thm:mainthm}, but
representations with this property do exist for any right-angled
Coxeter group. In fact, these representations form an open dense
subset of the space of simplicial representations, and can even be
arranged to have image lying in $\mathrm{O}(p, q)$ or $\SLpm(|S|, \Z)$
(see \Cref{rem:fully_nondegenerate_exists}).

Thus, our proof of \Cref{thm:cubulated_anosov} does not reprove the
result of Haglund--Wise and Agol that cubulated hyperbolic groups are
linear. Rather, we use their characterization of cubulated hyperbolic
groups as hyperbolic virtual quasiconvex subgroups of right-angled
Coxeter groups, and apply
\Cref{thm:mainthm}. \Cref{thm:cubulated_anosov} then follows once we
know that admitting an Anosov representation is a commensurability
invariant (see \cite[Lemma 2.1]{dt2022anosov}).

\subsection{Actions of reflection groups on projective space}\label{sec:refproj}

In the special case where the ambient right-angled Coxeter group $C$
in \Cref{thm:mainthm} is Gromov-hyperbolic, then work of
Danciger--Gu\'eritaud--Kassel--Lee--Marquis \cite{DGKLM} (see also
\cite{DGKgeomded}) implies that any representation
$\rho\colon C \to \SLpm(|S|, \R)$ as in the theorem is already
$1$-Anosov; in that case it follows easily that the restriction of
$\rho$ to any quasiconvex subgroup of $C$ is 1-Anosov also.

However, the Haglund--Wise construction typically yields a quasiconvex
embedding of a compact special hyperbolic group $\Gamma$ into a
\emph{non-hyperbolic} right-angled Coxeter group $C$, and when this
occurs we cannot invoke the results in \cite{DGKgeomded} or
\cite{DGKLM}. Further, we know of no procedure that replaces $C$ with
a hyperbolic Coxeter group (although we do not know of any reason such
a procedure cannot exist).

Our proof of \Cref{thm:mainthm} ultimately differs significantly from
the approach in \cite{DGKgeomded} and \cite{DGKLM}, and provides a new
proof of the fact that hyperbolic right-angled Coxeter groups admit
Anosov representations. Indeed, our methods rely on a completely
different characterization of Anosov representations. However, the
starting point for both proofs is the same: we consider
representations $C \to \SLpm(|S|, \R)$ which are \emph{generated by
  reflections}. The general theory of such representations was
developed by Vinberg \cite{Vinberg1971}, generalizing Tits' study of
the geometric representation, and centers around the action of $C$ on
a \emph{convex domain} in real projective space $\P(\R^{|S|})$.

As an illustrative example, the geometric representation of the free
product $\Gamma = \Z/2 * \Z/2 * \Z/2$ realizes $\Gamma$ as a
reflection lattice in $\mathrm{O}(2,1)$, acting on the
\emph{projective} or \emph{Klein} model of the hyperbolic plane
embedded in $\P(\R^3)$ as a convex ball. Although $\Gamma$ is a
hyperbolic group in this case, the geometric representation fails to
be Anosov, as the product of any pair of distinct generators is a
nontrivial unipotent element in $\mathrm{O}(2,1)$. However, $\Gamma$
also admits \emph{convex cocompact} representations into
$\mathrm{O}(2,1)$, where the product of any pair of distinct
generators is instead loxodromic. Such representations also preserve
another convex domain $\Omega \subset \P(\R^3)$, called the \emph{Vinberg
  domain}, which is \emph{not} projectively equivalent to the Klein
model for $\H^2$.

\subsection{Proof idea}

Our proof of \Cref{thm:mainthm} heavily exploits the relationship
between the projective geometry of the Vinberg domain $\Omega$ for a
right-angled Coxeter system $(C, S)$ acting by reflections, and the
combinatorial geometry of the \emph{Davis complex} $\daviscx$, a
natural $\mathrm{CAT}(0)$ cube complex with a properly discontinuous and cocompact $C$-action. Specifically, we relate \emph{half-spaces} in
$\daviscx$ to certain convex subsets of projective space, which we
call \emph{half-cones}. By examining the nesting properties of
half-cones, we are able to prove that if a geodesic $\gamma_n$ in an
irreducible right-angled Coxeter group $C$ does not get ``stuck''
inside of a proper \emph{standard} subgroup in $C$ (i.e.\ a subgroup
generated by a subset of the generating set $S$) then the singular
value gaps of the sequence $\rho(\gamma_n)$ grow at a uniform
rate. This directly verifies the condition in \Cref{defn:anosov}.

Our proof also needs to handle the case where the geodesic $\gamma_n$
spends arbitrarily long amounts of time inside of standard subgroups
of $C$, and this is where the majority of the work takes place. The
strategy is to induct on the size of the generating set $S$, and
assume that the geodesic $\gamma_n$ experiences uniform singular value
gap growth on each sub-geodesic lying in a proper standard subgroup of
$C$. The challenge is then to ``glue together'' the singular value gap
growth on each of these sub-geodesics.

This ``gluing'' process is somewhat involved, but there are
essentially only two different techniques at play. One of them is a
\emph{uniform transversality} argument for stable and unstable
subspaces of elements in $\SLpm(|S|, \R)$, and relies on an
understanding of the convex projective geometry of the Vinberg domain
for $\rho$. The other technique is to apply the \emph{higher-rank
  Morse lemma} and \emph{local-to-global principle} of
Kapovich--Leeb--Porti \cite{KLP2018,KLP2023} (see also
\cite{Riestenberg}), a pair of deep theorems about the geometry of
certain quasi-geodesic sequences in higher-rank symmetric spaces.

\begin{remark}
  An interesting feature of our proof is that it uses the
  hyperbolicity of the quasiconvex subgroup $\Gamma$ only indirectly,
  in the form of a condition on the walls in $\daviscx$ crossed by an
  arbitrary geodesic in $\Gamma$ (see \Cref{defn:bpp}). As a consequence, our proof actually shows that this condition implies hyperbolicity of $\Gamma$, since any group
  admitting an Anosov representation is necessarily hyperbolic; see
  \Cref{rem:bpp_hyperbolicity}. In the special case where
  $\Gamma = C$, this gives an alternative (albeit inefficient) proof of Gromov's ``no empty square''
  characterization of hyperbolic right-angled Coxeter groups relying on the
  theory of Anosov representations (this was also accomplished in
  \cite{DGKgeomded} and \cite{leemarquis2019}).
\end{remark}

\subsection{Examples and applications}
\label{sec:examples}

\Cref{thm:cubulated_anosov} provides evidence that groups admitting
Anosov representations are in some sense ``abundant.'' One way to make
this precise is Gromov's density model for random finitely presented
groups: at density
$<\frac{1}{12}$, a random group satisfies the $C'(\frac{1}{6})$ small
cancellation condition, and such groups are hyperbolic and
cubulated \cite{wise2004cubulating}; more generally, at density $<\frac{1}{6}$, random groups are hyperbolic and
cubulated (see \cite[9.B]{gromov1993asymptotic},
\cite{ollivier2011cubulating}). Hence, by \Cref{thm:cubulated_anosov}, random finitely presented groups at density $< \frac16$ admit Anosov representations. 

These results tell us that we cannot control the dimension
of the Anosov representations provided by \Cref{thm:cubulated_anosov} even for torsion-free cubulated hyperbolic groups of bounded cohomological dimension: for fixed $n$, a random group at any positive density has no
$n$-dimensional linear representations with finite
kernel \cite{kozmalubotzky}, while random groups at density $<\frac{1}{2}$ have cohomological dimension~$2$.

  We mention, however, that not all groups that admit $1$-Anosov representations into $\mathrm{SL}(d,\R)$---indeed, not all convex cocompact subgroups of rank-one Lie groups---are cubulated. For example, any action of a discrete group with Kazhdan's property (T) on a finite-dimensional $\mathrm{CAT}(0)$ cube complex has a global fixed point \cite{nibloreeves97}, ruling out cubulability for uniform lattices in $\mathrm{Sp}(n,1)$,
$n \geq 2$, and $F_{4{(-20)}}$. For subtler reasons, it is also true that
uniform lattices in $\mathrm{PU}(n,1)$, $n \geq 2$, fail to be
cubulated \cite{delzantgromov, py2013coxeter}.

\subsubsection{Strict hyperbolization}
Recent work of Lafont and Ruffoni \cite{lafont2022special} has
established that the Charney--Davis \emph{strict hyperbolization}
process \cite{charney1995strict} also yields cubulated hyperbolic
groups, which allows us to produce examples of Anosov subgroups with
various ``exotic'' properties. For instance, Ontaneda \cite{ontaneda}
has used strict hyperbolization to construct new examples of closed
negatively curved Riemannian manifolds in any dimension $n \ge 4$
which are not homeomorphic to any locally symmetric space of rank
one. Work of Januszkiewicz--\'Swi\c{a}tkowski \cite{js2003} implies
that the fundamental groups of these manifolds are not commensurable
to any Coxeter group when $n > 61$, meaning that
\Cref{thm:cubulated_anosov} gives the first proof that these groups
admit Anosov representations.

For another sample application of \Cref{thm:cubulated_anosov} and
strict hyperbolization, recall that if $M$ is a closed negatively
curved Riemannian manifold, then the Gromov boundary of $\pi_1(M)$ is
a topological sphere. However, Davis--Januszkiewicz \cite{dj91} showed
that the latter may fail if $M$ is merely a closed aspherical manifold
with hyperbolic fundamental group. The Davis--Januszkiewicz examples
are constructed via strict hyperbolization, so their fundamental groups are cubulated by the work of Lafont--Ruffoni. Combining these
results with a theorem of Bestvina \cite[Thm.~2.8]{bestvina96} and
Theorem \ref{thm:cubulated_anosov} yields the following:
\begin{thm}\label{exotic}
  For every $n \geq 4$, there is some $d \in \mathbb{N}$ and an Anosov subgroup of
  $\mathrm{SL}(d, \mathbb{R})$ whose Gromov boundary is not homeomorphic to an $n$-sphere, but is
  nevertheless a homology $n$-manifold with the homology of an $n$-sphere.
\end{thm}

This theorem may be viewed as a positive answer for each $n \geq 4$ to
a variant of a question of Kapovich
\cite[Q.~9.4]{kapovich2008kleinian}, within the broader realm of
Anosov groups. It in fact follows from known results that Theorem
\ref{exotic} also holds for $n=3$. Indeed, for an example with $n=3$,
it suffices to take one of the Anosov representations guaranteed by
Danciger--Gu\'eritaud--Kassel \cite{DGKgeomded} of a right-angled
Coxeter group given by a flag no-square triangulation of a nontrivial
homology $3$-sphere (see \cite{davis1983}); for the existence of such
triangulations, see \cite{ps2009flag}. That the latter approach fails
as soon as $n \geq 4$ follows from aforementioned work of
Januszkiewicz--\'Swi\c{a}tkowski \cite[Sect.~2.2]{js2003}.

\begin{remark}
  The situation for groups admitting $1$-Anosov representations into
  $\SL(d, \R)$ appears to be strikingly different from the situation
  for groups admitting \emph{Borel} Anosov representations into
  $\SL(d, \R)$ (a representation $\rho\colon \Gamma \to \SL(d, \R)$ is Borel
  Anosov if it is $k$-Anosov for every $1 \le k < d$). Indeed,
  Sambarino conjectured that every group admitting a Borel Anosov
  representation into $\SL(d, \R)$ is virtually either a free group or
  the fundamental group of a closed surface, and this conjecture has
  been verified for infinitely many $d$ (see \cite{ct2020},
  \cite{tsouvalas2020}, \cite{dey2022}).
\end{remark}

\subsection{Further questions}

We have already observed that \Cref{thm:mainthm} would no longer hold
if we removed the hypothesis regarding fully nondegenerate Cartan
matrices (see the end of Section~\ref{sec:refproj}). However, it seems plausible that a version of
\Cref{thm:mainthm} could hold with a considerably relaxed version of
this hypothesis, as long as we impose some assumptions on the
quasiconvex embedding $\Gamma \hookrightarrow C$. For instance, it
might be sufficient to ask for $\Gamma$ to have finite intersection
with every standard virtually unipotent subgroup of $C$. For
appropriate quasiconvex embeddings, this could allow the conclusion of
\Cref{thm:mainthm} to hold for an \emph{arbitrary} representation of
$C$ by reflections (in particular, for the geometric representation).

Even in its current form, \Cref{thm:mainthm} still gives us a great
deal of freedom to pick the simplicial representation
$\rho\colon C \to \SLpm(d, \R)$. In particular, for a fixed infinite
right-angled Coxeter group $C$, we can pick $\rho$ from a
positive-dimensional submanifold $M$ of the representation variety
$\Hom(C, \SLpm(d, \R))$. We might want to consider properties of the
restriction map $R\colon M \to \Hom(\Gamma, \SLpm(d, \R))$---for instance,
it would be interesting to know the dimension of $R(M)$, or whether
$R(M)$ contains any representations with Zariski-dense image.

\subsection{Acknowledgments}

The authors would like to thank Gye-Seon Lee, Jason Manning, Lorenzo
Ruffoni, and Abdul Zalloum for productive discussion. We are also
grateful to the Mathematisches Forschungsinstitut Oberwolfach and the
organizers of the Arbeitsgemeinschaft: Higher Rank Teichm\"uller
Theory (2241), where this work was initiated. S.D. was supported by
the Huawei Young Talents Program.  B.F. was partially supported by the
grant NRF-2022R1I1A1A01072169 and received funding from the European
Research Council (ERC) under the European Union's Horizon 2020
research and innovation programme (ERC starting grant DiGGeS, grant
agreement No 715982, and ERC consolidator grant GeometricStructures,
grant agreement No 614733).  T.W. was partially supported by NSF grant
DMS-2202770.  F.Z. was partially supported by ISF grant 737/20 and an
AMS-Simons Travel Grant.

\section{Cube complexes and right-angled Coxeter groups}
\label{sec:cube_complexes_RACGs}

In this section we briefly review some essential background on the
topic of nonpositively curved cube complexes, right-angled Coxeter
groups, and the Davis complex. We refer to \cite{HW08}, \cite{Davis}
for further detail. Afterwards, we introduce a useful combinatorial
framework for working with geodesics in the Davis complex, in the form
of \emph{itineraries}.

\subsection{CAT(0) cube complexes}

For our purposes, a {\em cube complex} $X$ is a finite-dimensional
cell complex in which each cell is a cube and attaching maps are
combinatorial isomorphisms onto their images. A $d$-cell of $X$ is a
{\em $d$-cube}. A $0$-cube is a {\em vertex}, a $1$-cube is an {\em edge}, and a $2$-cube is a {\em square}. Under the identification
of a $d$-cube $c$ of $X$ with $[-1,1]^d$, a {\em midcube} of $c$ is an
intersection of $c$ with a coordinate hyperplane of $\mathbb{R}^d$. By
gluing midcubes of adjacent cubes of $X$ whenever they meet, one
obtains immersed subspaces of $X$, called {\em hyperplanes}, each of
which also carries a natural cube complex structure. Note that a
compact cube complex possesses only finitely many hyperplanes.

Two edges of a cube complex $X$ are {\em elementary parallel} if they
appear as opposite edges of a square in $X$. A {\em wall} of $X$ is a
class of the equivalence relation on the edge set of $X$ generated by
elementary parallelisms. Two edges of $X$ are {\em parallel} if they
belong to the same wall of $X$ (in other words, if they are dual to
the same hyperplane of $X$). Frequently, we will abuse terminology and
refer to properties of ``walls'' when we really mean properties of the
corresponding hyperplanes. In particular, when we say that an
intersection of walls $W_1 \cap W_2$ is empty or nonempty, we mean to
refer to an intersection of hyperplanes.

One says a cube complex $X$ is {\em nonpositively curved} if the link
of each vertex of $X$ is a flag simplicial complex; recall that a
simplicial complex $L$ is {\em flag} if each clique $\mathcal{V}$ in
$L$ spans a $(|\mathcal{V}|-1)$-simplex. If $X$ is moreover simply
connected, then $X$ is said to be $\mathrm{CAT}(0)$. Implicit in this
terminology is a theorem of Gromov \cite{gromov} that the path metric
on a cube complex $X$ induced by the Euclidean metric on each of its
cubes is $\mathrm{CAT}(0)$ if and only if $X$ is $\mathrm{CAT}(0)$ in
the previous combinatorial sense. A key feature of a $\mathrm{CAT}(0)$
cube complex is that each of its hyperplanes is separating.

\subsection{Right-angled Coxeter groups}

Let $\Sigma$ be a finite simplicial graph with vertex set $S$. The {\em right-angled Coxeter group} $C_\Sigma$ with {\em nerve} $\Sigma$ is the group given by
the presentation with generating set $S$ subject to the relations that
two generators $s,t \in S$ commute if and only if $s$ and $t$ are
adjacent as vertices in $\Sigma$. The pair $(C_\Sigma, S)$ is a {\em
  right-angled Coxeter system}. A conjugate within $C_\Sigma$ of an
element of $S$ is a {\em reflection}. If the complement of $\Sigma$ is
connected, we say $(C_\Sigma, S)$ is {\em irreducible}; this is
equivalent to saying that $C_\Sigma$ does not decompose as a
nontrivial direct product.

If $(C,S)$ is a right-angled Coxeter system and $T \subset S$, we
denote by $C(T)$ the subgroup of $C$ generated by $T$. We refer to
such subgroups of $C$ as {\em standard subgroups}. For any
$T \subset S$, the pair $(C(T), T)$ is again a right-angled Coxeter
system.

\subsubsection{The Davis complex}

Given a finite simplicial graph $\Sigma$, there is a $\mathrm{CAT}(0)$
cube complex $\mathrm{D}(C_\Sigma, S)$, called the {\em Davis complex}
of $(C_\Sigma, S)$, on which the group $C_\Sigma$ acts by
combinatorial automorphisms, that may be constructed as follows. Let
$\mathrm{D}'(C_\Sigma, S)$ be the square complex (i.e.\
$2$-dimensional cube complex) obtained by attaching a square to each
labeled $4$-cycle of the form $stst$ in the Cayley graph $\mathrm{Cay}(C_\Sigma,S)$ of $C_\Sigma$
with respect to the generating set $S$, where $s,t \in S$. Then
$\mathrm{D}(C_\Sigma, S)$ is the unique nonpositively curved cube
complex with $2$-skeleton $\mathrm{D}'(C_\Sigma, S)$.

Each wall $W$ in $\mathrm{D}(C_\Sigma, S)$ is fixed by a unique
reflection $r$ in $C_\Sigma$, and conversely each reflection $r$ fixes
a unique wall. As any reflection $r$ is a conjugate of a unique
$s \in S$, the edges comprising any given wall $W$ of
$\mathrm{D}(C_\Sigma, S)$ are all labeled with a single generator
$\type{W} \in S$. We call $\type{W}$ the {\em type} of $W$. For each
$s \in S$, we also let $W(s)$ denote the unique wall fixed by the
reflection $s$.

\subsubsection{Quasiconvex subgroups}

Given a right-angled Coxeter system $(C,S)$, a subgroup $\Gamma < C$ is {\em quasiconvex} (with respect to the standard generating set $S$) if, viewing $\Gamma$ as a subset of the vertices of $\mathrm{D}(C,S)$, there is some $K > 0$ such that every combinatorial geodesic in $\mathrm{D}(C,S)$ with endpoints in $\Gamma$ lies in the combinatorial $K$-neighborhood of $\Gamma$. (For example, if $s_1, t_1, s_2, t_2$ are distinct elements of $S$ such that each element of $\{s_1, t_1\}$ commutes with each element of $\{s_2,t_2\}$, but $s_i$ and $t_i$ do not commute for $i=1,2$, then the cyclic subgroup $\langle s_1t_1s_2t_2 \rangle < C$ is {\em not} quasiconvex.) A subgroup $\Gamma < C$ is quasiconvex if and only if there is a $\Gamma$-invariant convex subcomplex $\widetilde{Y}$ of $\mathrm{D}(C,S)$ on which $\Gamma$ acts cocompactly; see \cite[Cor.~7.8]{HW08} or \cite[Thm.~H]{haglund2008finite}. 
In particular, standard subgroups of $C$ are quasiconvex. This second characterization of quasiconvexity is the one that will be relevant to us.

It turns out that the hyperbolic groups that are cubulated are precisely those that virtually embed as quasiconvex subgroups of right-angled Coxeter groups. Indeed, it follows from seminal work of Haglund--Wise \cite{HW08} and Agol \cite{agolVHC} that if a hyperbolic group $\Gamma$ acts properly discontinuously and cocompactly on a $\mathrm{CAT}(0)$ cube complex $\widetilde{X}$, then there is a finite-index subgroup $\Gamma' < \Gamma$, a right-angled Coxeter system $(C,S)$, an embedding $\iota\colon \Gamma' \rightarrow C$, and an $\iota$-equivariant embedding of $\widetilde{X}$ as a convex subcomplex of $\mathrm{D}(C,S)$. 

In more detail, Agol \cite{agolVHC} showed that there is a finite-index torsion-free subgroup $\Lambda < \Gamma$ such that the cube complex $\widetilde{X} / \Lambda$ is {\em special} in the sense of Haglund and Wise \cite{HW08}, answering a question of the latter two authors \cite[Prob.~11.7]{HW08}. One can then pass to a deeper finite-index subgroup $\Gamma' < \Lambda$ such that $X:= \widetilde{X} / \Gamma'$ is moreover {\em $C$-special}; see \cite[Prop.~3.10]{HW08}. The latter condition is designed to ensure the existence of a {\em local isometry} from the cube complex $X$ to the orbicomplex $\mathrm{D}(C_X, S_X)/C_X$ for some right-angled Coxeter system $(C_X, S_X)$ associated to $X$, inducing an embedding $\iota\colon \Gamma' \rightarrow C_X$ between (orbicomplex) fundamental groups, and lifting to an $\iota$-equivariant embedding of $\widetilde{X}$ as a convex subcomplex of $\mathrm{D}(C_X,S_X)$. (Alternatively, one can find a finite-index subgroup $\Gamma'' < \Lambda$ so that $\widetilde{X} / \Gamma''$ admits a local isometry into the Salvetti complex for a right-angled Artin group, and then apply 
\cite{dj2000}.)

We remark that, even in the case that the action of $\Gamma$ on $\widetilde{X}$ is the action of a hyperbolic right-angled Coxeter group on its Davis complex, the right-angled Coxeter group $C_X$ that one obtains via the process above may {\em not} be hyperbolic.

\subsection{Itineraries}

Throughout this paper, we will need a good understanding of the
combinatorial behavior of geodesics in right-angled Coxeter groups,
and especially geodesics lying in quasiconvex hyperbolic subgroups of
right-angled Coxeter groups. It is often convenient to work not
with the geodesics themselves, but rather some related combinatorial
data in the form of an \emph{itinerary}.

For the following, we fix a right-angled Coxeter system $(C, S)$.

\begin{definition}
  Recall that a wall $W$ in the Davis complex $\daviscx$ can be viewed
  as an equivalence class of edges in the Cayley graph $\Cay(C,
  S)$. An \emph{itinerary} is a sequence of walls $W_1, \ldots, W_n$,
  such that for some sequence of edges $e_i \in W_i$, the sequence
  $e_1 \cdots e_n$ is a geodesic edge path in $\Cay(C, S)$. We say that this
  edge path \emph{follows} the itinerary $\mc{W}$.
\end{definition}

A geodesic edge path $e_1 \cdots e_n$ in the Cayley graph $\Cay(C, S)$ always
determines a geodesic word in $S$, by reading off the generators
$s_i \in S$ labeling each edge $e_i$; conversely, a geodesic word in
$S$ gives rise to infinitely many different geodesic edge paths, one for each
possible base point of the path in $\Cay(C, S)$. So a geodesic word
records strictly less information than a geodesic edge path in $\Cay(C, S)$.

On the other hand an itinerary records strictly \emph{more}
information than a geodesic word but \emph{less} information than a geodesic
edge path: there may be many different edge paths which follow the
same itinerary $W_1, \ldots, W_n$, but each of these edge paths
determines the same geodesic word, namely the word
\[
  \type{W_1} \cdots \type{W_n}.
\]

\begin{definition}
  If $\mc{W}$ is an itinerary $W_1, \ldots, W_n$, we say the geodesic
  word $\type{W_1} \cdots \type{W_n}$ is \emph{traversed} by $\mc{W}$.
  We also say that $\mc{W}$ \emph{traverses} the unique element
  $\gamma \in C$ represented by this geodesic word.

  If $\alpha \in C$ is the initial vertex of some edge path following
  $\mc{W}$, then we say that $\mc{W}$ \emph{departs from}
  $\alpha$. Similarly, if $\beta \in C$ is the last vertex of an edge
  path following $\mc{W}$, then $\mc{W}$ \emph{arrives at} $\beta$. If
  $\alpha, \beta \in C$, and $\mc{W}$ departs from $\alpha$ and
  traverses $\alpha^{-1}\beta$, then we say that $\mc{W}$ \emph{joins}
  $\alpha$ to $\beta$. Note that this is stronger than saying that
  $\mc{W}$ departs from $\alpha$ and arrives at $\beta$.
\end{definition}

The proposition below follows directly from the fact that walls are
defined to be equivalence classes of edges in the Cayley graph of the
Coxeter system $(C,S)$:
\begin{prop}
  \label{prop:wall_expression}
  Let $W_1, \ldots, W_n$ be an itinerary departing from the identity,
  and let $s_i = \type{W_i}$ for all $1 \le i \le n$. Then
  $W_n = s_1 \cdots s_{n-1} W(s_n)$.
\end{prop}

In general a single itinerary can depart from different elements in
$C$ (and likewise can arrive at different elements in $C$). However,
an itinerary always traverses a unique element.

\begin{definition}
  \label{defn:element_traversed_itinerary}
  If $\mc{U} = W_1, \ldots, W_n$ is an itinerary, we let
  $\gamma(\mc{U}) = \gamma(W_1, \ldots, W_n)$ denote the group element
  traversed by $\mc{U}$.
  
  If $W_i, W_j$ are walls in $\mc{U}$, with $W_i$ appearing before
  $W_j$, we let $\gamma_{\mc{U}}(W_i, W_j)$ denote the group element
  traversed by the sub-itinerary of $\mc{U}$ beginning with $W_i$ and
  ending with $W_j$.
\end{definition}

\subsubsection{Partial order on walls}

As any wall in $\daviscx$ separates $\daviscx$ into two convex
components, it follows that for any $\alpha, \beta \in C$, the walls
appearing in an itinerary $W_1 \ldots , W_n$ joining $\alpha$ to
$\beta$ are precisely the walls in $\daviscx$ separating $\alpha$ from
$\beta$. Motivated by this, we introduce the following notation.

\begin{definition}
  For $\alpha, \beta \in C$, we let $\walls(\alpha, \beta)$ denote the
  set of walls in $\daviscx$ separating $\alpha$ from $\beta$. We
  write $\walls(\alpha)$ for $\walls(\identity, \alpha)$.

  The set $\walls(\alpha, \beta)$ is endowed with a partial order $<$,
  defined as follows: if $W_i, W_j \in \walls(\alpha, \beta)$, then
  $W_i < W_j$ if $W_i$ separates $\alpha$ from $W_j$ in $\daviscx$
  (equivalently, if $W_j$ separates $W_i$ from $\beta$).
\end{definition}

Recall that two elements $a, b$ in a poset are \emph{incomparable} if
neither $a < b$ nor $b < a$ holds. We say that two disjoint subsets
$A, B$ of a poset are \emph{completely incomparable} if every element
of $A$ is incomparable to every element in $B$.

It is immediate that if $\alpha, \beta \in C$ and $W_i, W_j$ are two
walls in $\walls(\alpha, \beta)$, then $W_i$ and $W_j$ are
incomparable with respect to $<$ if and only if $W_i \cap W_j$ is
nonempty. When this occurs, the generators $\type{W_i}$ and
$\type{W_j}$ must commute.

Every itinerary joining $\alpha$ to $\beta$ determines a total
ordering of the set $\walls(\alpha, \beta)$ which is compatible with
the partial ordering $<$. The proposition below says that all
compatible total orderings of $\walls(\alpha, \beta)$ arise in
precisely this way.
\begin{prop}
  \label{prop:itinerary_correspondence}
  Let $\alpha, \beta \in C$. There is a one-to-one correspondence
  between the following three sets:
  \begin{enumerate}
  \item Itineraries joining $\alpha$ to $\beta$,
  \item Geodesic words in $S$ representing $\alpha^{-1}\beta$,
  \item Total orderings of $\walls(\alpha, \beta)$ which are
    compatible with $<$.
  \end{enumerate}
\end{prop}
\begin{proof}
  The correspondence between the first two sets is immediate, once we
  recognize that itineraries joining $\alpha$ to $\beta$ are in
  one-to-one correspondence with geodesic edge paths in the Cayley
  graph $\Cay(C, S)$ joining $\alpha$ to $\beta$. We have already
  observed that any itinerary joining $\alpha$ to $\beta$ gives rise
  to an ordering on $\walls(\alpha, \beta)$ compatible with $<$, so we
  just need to check that any such ordering determines an
  itinerary.

  First observe that for any $\gamma \in C$ and any wall $W$ in
  $\daviscx$, if no walls separate $\gamma$ from $W$, then $W$
  contains a unique edge incident to $\gamma$. To see this, let
  $H_\pm$ denote the half-spaces in $\daviscx$ bounded by $W$, chosen
  so that $\gamma \in H_-$. We consider a minimal-length edge path $p$
  in $\Cay(C, S)$ joining $\gamma$ to $H_+$. The last edge $e_2$ in
  $p$ must belong to $W$, since it crosses from $H_-$ to $H_+$. If
  there is more than one edge in $p$ and $e_1$ is the next-to-last
  edge, then, by minimality of $p$, the edges $e_1$ and $e_2$ do not
  lie in a common square of $\mathrm{D}(C,S)$. It follows that the
  edge path $e_1e_2$ is a geodesic segment for the $\mathrm{CAT}(0)$
  metric on $\mathrm{D}(C,S)$ (indeed, it is more generally true that
  a local isometry between $\mathrm{CAT}(0)$ cube complexes is an
  isometric embedding with respect to their $\mathrm{CAT}(0)$ metrics;
  see \cite[Lem.~2.11]{HW08},
  \cite[Prop.~II.4.14]{bridsonhaefliger}). The nearest point
  projection to $e_i$ of the hyperplane $\Pi_i$ dual to $e_i$ is the
  midpoint of $e_i$, so that the $\Pi_i$ are disjoint. The wall
  corresponding to $\Pi_1$ thus separates $\gamma$ from $W$, a
  contradiction.

  Now, fix an ordering $W_1, \ldots, W_n$ on $\walls(\alpha, \beta)$
  which is compatible with $<$. The previous claim tells us that $W_1$
  contains a unique edge in $\Cay(C, S)$ incident to $\alpha$. If
  $\alpha'$ is the other endpoint of this edge, then
  $W_2, \ldots, W_n$ is an ordering on $\walls(\alpha', \beta)$,
  compatible with the partial ordering $<$ on this set. Proceeding
  iteratively, we then construct an edge path in $\Cay(C, S)$ from
  $\alpha$ to $\beta$ which crosses exactly the sequence of walls
  $W_1, \ldots, W_n$, meaning this sequence is an itinerary.
\end{proof}

\begin{definition}
  We say two itineraries $\mc{U}, \mc{U}'$ are \emph{equivalent} if there are elements
  $\alpha, \beta \in C$ so that both $\mc{U}$ and $\mc{U}'$ join
  $\alpha$ to $\beta$.
\end{definition}

\Cref{prop:itinerary_correspondence} means that equivalent itineraries
$\mc{U}, \mc{U}'$ always consist of the same set of walls. And, if
$\mc{U}$ joins $\alpha$ to $\beta$ for some $\alpha, \beta \in C$,
then so does any equivalent itinerary $\mc{U}'$. Thus the third
condition of \Cref{prop:itinerary_correspondence} ensures that the
definition of ``equivalence'' actually describes an equivalence
relation.

\subsubsection{Efficient itineraries}

Whenever $W$ and $W'$ are walls in $\daviscx$, then there is always
some itinerary $\mc{U}$ whose first wall is $W$ and whose last wall is
$W'$. Every wall $W_i$ in $\mc{U}$ must either separate $W$ from $W'$,
or intersect at least one of $W,
W'$. \Cref{prop:itinerary_correspondence} tells us that we can always
find another itinerary equivalent to $\mc{U}$ by putting all of the
walls in $\mc{U}$ intersecting $W$ or $W'$ either first or last. That
is, if there are any walls in $\mc{U}$ which intersect either $W$ or
$W'$, we can reorder the walls and restrict to a strictly shorter
sub-itinerary to get a new (non-equivalent) itinerary whose first wall
is $W$ and whose last wall is $W'$. On the other hand, if $\mc{U}$ is
\emph{any} itinerary with initial wall $W$ and final wall $W'$, then
$\mc{U}$ must contain every wall separating $W$ from $W'$.

\begin{definition}
  We say that an itinerary $\mc{U} = W_1, \ldots, W_n$ is
  \emph{efficient} if every wall $W_i$ with $1 < i < n$ is disjoint
  from both $W_1$ and $W_n$.
\end{definition}

Given any two distinct walls $W, W'$, the argument above shows that
there is always an efficient itinerary with initial wall $W$ and final
wall $W'$, and that each efficient itinerary between $W$ and $W'$ must
consist of the same set of walls. The itinerary orders these walls in
a way which is compatible with the partial ordering: $W_i < W_j$ if
$W_i$ separates $W$ from $W_j$. Thus
\Cref{prop:itinerary_correspondence} means that any pair of efficient
itineraries between $W, W'$ are equivalent, and we can define the
following.
\begin{definition}
  Let $W_1$, $W_2$ be distinct walls in $\daviscx$. We let
  $\gamma(W_1, W_2)$ denote the unique element in $C$ traversed by any
  efficient itinerary whose first wall is $W_1$ and whose last wall is
  $W_2$.
\end{definition}

In general, not every itinerary is equivalent to an efficient
itinerary, although this is ``almost'' true in the special case where
$\daviscx$ is hyperbolic; see \Cref{lem:minimal_wall_intersections}
below.

\section{Bounded product projections}
\label{sec:bounded_product_projections}

In this section we fix a right-angled Coxeter system $(C, S)$, and use
the setup from the previous section to prove some combinatorial
results about \emph{hyperbolic subcomplexes} of the Davis complex
$\daviscx$. Our main aim is to prove \Cref{prop:disjoint_walls_I},
which implies that every geodesic in a hyperbolic subcomplex of
$\daviscx$ is traversed by an itinerary consisting almost entirely of
``regularly-spaced'' pairwise disjoint walls. In later sections, we
will be able to work with geodesics in $C$ by only considering this
set of disjoint walls.

\begin{definition}
  \label{defn:bpp}
  Suppose that $(C, S)$ is a right-angled Coxeter system. Let
  $\gamma \in C$, and let $D > 0$. We say a group element
  $\gamma \in C$ has {\em $D$-bounded product projections} if every
  pair of disjoint completely incomparable subsets
  $A, B \subset \walls(\gamma)$ satisfies $\min(|A|, |B|) \le D$. We
  say that a subgroup $\Gamma \leq C$ has {\em $D$-bounded
    product projections} if every $\gamma \in \Gamma$ has $D$-bounded
  product projections. We just say $\Gamma$ has \emph{bounded product
    projections} if there exists some $D > 0$ so that $\Gamma$ has
  $D$-bounded product projections.
\end{definition}

More intuitively, the elements in $C$ with bounded product projections
are precisely those elements $\gamma$ whose geodesic representatives
do not ``travel diagonally'' in a combinatorially embedded Euclidean
2-plane $E \hookrightarrow \daviscx$; ``diagonally'' is in reference
to the product structure $E = \R \times \R$ induced by the cubulation
of $E$. A geodesic representing an element with $D$-bounded product
projections \emph{may} spend an arbitrary amount of time in a 2-flat,
but it must spend all but $D$ of its length traveling parallel to one
of the $\R$ factors.

The following is immediate from \cite[Thm.~4.1.3]{hagen2012geometry}.

\begin{lem}\label{lem:bpp_poset}
  Let $\Gamma$ be a quasiconvex subgroup of $C$. If $\Gamma$ is hyperbolic, then $\Gamma$ has bounded product projections.
\end{lem}

\begin{remark}
  \label{rem:bpp_hyperbolicity}
  It is shown in \cite{hagen2012geometry} that the converse of \Cref{lem:bpp_poset} also holds, that is, that if $\Gamma$ has bounded product projections, then $\Gamma$ is hyperbolic. In fact, this direction also follows from the proof of Theorem \ref{thm:mainthm}; our proof shows that if~$\Gamma$ has bounded product projections, then certain representations of $C$ restrict to Anosov representations of $\Gamma$, and it follows from \cite{KLP2018}, \cite{BPS} that any group admitting an Anosov representation is hyperbolic.
\end{remark}

Whenever $\mc{W}$ is an itinerary, we let $|\mc{W}|$ denote the number
of walls appearing in $\mc{W}$. If $\mc{W} = W_1, \ldots, W_n$ and
$\mc{U} = U_1, \ldots U_m$ are itineraries, we write $\mc{W}, \mc{U}$
for the concatenation
\[
  W_1, \ldots, W_n, U_1, \ldots, U_m,
\]
as long as this sequence of walls is also an itinerary.

\begin{prop}
  \label{prop:disjoint_walls_I}
  Given $D > 0$, there exists $R > 0$ (depending only on $D$)
  satisfying the following. Suppose that $\gamma \in C$ has
  $D$-bounded product projections. Then any itinerary traversing
  $\gamma$ is equivalent to an itinerary $\mc{U}$ of the form
  \[
    \{W_1\}, \mc{V}_1, \{W_2\}, \mc{V}_2, \ldots, \{W_n\}, \mc{V}_n,
  \]
  such that:
  \begin{enumerate}[label=(\alph*)]
  \item\label{item:smallsep} Every $\mc{V}_i$ satisfies
    $|\mc{V}_i| \le R$,
  \item\label{item:v_intersect} Every wall in $\mc{V}_i$ intersects
    $W_i$,
  \item\label{item:disjoint_daviscx} For every $i \ne j$ we have
    $W_i \cap W_j = \emptyset$,
  \item\label{item:few_intersections} Every wall $W_i$ intersects at
    most $R$ other walls in $\mc{U}$.
  \end{enumerate}
\end{prop}

\begin{remark}
  Conditions \ref{item:v_intersect} and \ref{item:few_intersections}
  in the proposition together imply \ref{item:smallsep}, but we still
  list \ref{item:smallsep} above because later we will use
  \ref{item:v_intersect}, \ref{item:disjoint_daviscx} and a stronger
  form of \ref{item:smallsep} to prove \ref{item:few_intersections}.
\end{remark}

\Cref{prop:disjoint_walls_I} tells us in particular that the length of
any geodesic in $\daviscx$ representing some $\gamma \in C$ with
$D$-bounded product projections can be estimated (up to a
multiplicative constant) as the maximal length of a \emph{chain}
$W_1 < W_2 < \ldots < W_n$ in $\walls(\gamma)$. In fact, this weaker
statement holds for \emph{arbitrary} $\gamma \in C$; the point of the
proposition is that when $\gamma$ has $D$-bounded product projections,
every wall in the chain is disjoint from almost every other wall in
$\walls(\gamma)$.

The proof of \Cref{prop:disjoint_walls_I} is purely combinatorial. In
fact, it relies only on the poset structure of the set of walls
separating a pair of elements in $C$.

We first prove a useful lemma:
\begin{lem}
  \label{lem:minimal_wall_intersections}
  Suppose that $\gamma = \alpha^{-1}\beta$ is a nontrivial element in
  $C$ with $D$-bounded product projections. Then there exists a
  minimal wall $W$ in $\walls(\alpha, \beta)$ such that
  $W \cap W' \ne \emptyset$ for at most $(2D + 1) \cdot 4^D$ walls
  $W'$ in $\walls(\alpha, \beta)$.
\end{lem}

In particular, the lemma tells us that any group element $\gamma$ with
$D$-bounded product projections is traversed by an itinerary which is
``nearly'' efficient.

\begin{proof}
  Without loss of generality we may assume $\alpha = \identity$ and
  consider the set of walls $\walls(\gamma)$. We let
  $\minwalls(\gamma)$ denote the minimal walls in $\walls(\gamma)$.

  Every minimal element in a poset is incomparable with every other
  minimal element. So, the bounded product projections property
  implies that the number of walls in $\minwalls(\gamma)$ is at most
  $2D + 1$, since otherwise we could partition $\minwalls(\gamma)$
  into disjoint completely incomparable subsets, both containing at
  least $D + 1$ elements.

  We can think of every wall $W \in \walls(\gamma)$ as having one of
  finitely many ``intersection types,'' determined by the walls in
  $\minwalls(\gamma)$ which $W$ intersects. Precisely, we define an
  ``intersection mapping''
  $I_{\min}:\walls(\gamma) \to 2^{\minwalls(\gamma)}$, by:
  \[
    I_{\min}(W) = \{V \in \minwalls(\gamma) : V \cap W \ne \emptyset\}.
  \]
  We claim the following:
  \begin{claim}
    \label{claim:incomparable_pairs}
    Let $U_1, U_2$ be subsets of $\minwalls(\gamma)$, such that
    $U_1 \nsubseteq U_2$ and $U_2 \nsubseteq U_1$. Then the sets of
    walls $A = I_{\min}^{-1}(U_1)$ and $B = I_{\min}^{-1}(U_2)$ are
    completely incomparable.
  \end{claim}
  To prove the claim, observe that if the hypothesis holds, then there
  is a wall $V_1 \in U_1 \setminus U_2$ and a wall
  $V_2 \in U_2 \setminus U_1$. Since $V_1$ does not intersect any wall
  in $B$, and $V_1$ is minimal, we have $V_1 < W_B$ for every
  $W_B \in B$. Similarly, we have $V_2 < W_A$ for every $W_A \in
  A$. And, since every wall $W_A \in A$ intersects $V_1$ nontrivially,
  each $W_A$ is incomparable with $V_1$. So, there is a total ordering
  of $\walls(\gamma)$ (compatible with $<$) where each $W_A \in A$
  precedes $V_1$, hence precedes every $W_B \in B$. Arguing
  symmetrically, there is also a compatible ordering of the walls
  where every $W_B \in B$ precedes every $W_A \in A$. Thus, $A$ and
  $B$ are completely incomparable, proving the claim.

  Next, consider the collection of subsets
  $T \subset 2^{\minwalls(\gamma)}$ given by
  \[
    T = \{U \subset \minwalls(\gamma) : |I_{\min}^{-1}(U)| > 2D + 1\}.
  \]
  The previous claim, together with the bounded product projections
  axiom, implies that if $U_1$ and $U_2$ are subsets of
  $\minwalls(\gamma)$, neither of which is a subset of the other, then
  either $|I_{\min}^{-1}(U_1)| \le D$ or
  $|I_{\min}^{-1}(U_2)| \le D$---so in particular at most one of
  $U_1, U_2$ can lie in $T$. This means that there is at most one
  maximal element in $T$, with respect to the partial ordering on
  $2^{\minwalls(\gamma)}$ given by inclusion. However, the element
  $\minwalls(\gamma) \in 2^{\minwalls(\gamma)}$ cannot itself lie in
  $T$: by definition, every element in
  $I_{\min}^{-1}(\minwalls(\gamma))$ is incomparable with every
  minimal wall, so every element in $I_{\min}^{-1}(\minwalls(\gamma))$
  is itself minimal in $\walls(\gamma)$ and therefore
  $|I_{\min}^{-1}(\minwalls(\gamma))| \le |\minwalls(\gamma)| \le 2D +
  1$.

  We conclude that $T$ is either empty, or it has a unique maximal
  element which is not all of $\minwalls(\gamma)$. In either case,
  there must be some wall $W_- \in \minwalls(\gamma)$ such that
  $W_- \notin U$ for any $U \in T$.

  In other words, for any set $U \subset \minwalls(\gamma)$ containing
  $W_-$, we have $|I_{\min}^{-1}(U)| \le 2D + 1$. And by definition,
  the set of walls in $\walls(\gamma)$ intersecting $W_-$ is the union
  of the sets $I_{\min}^{-1}(U)$ over all subsets
  $U \subset \minwalls(\gamma)$ with $W_- \in U$. Since the number of
  such sets $U$ is at most
  $2^{|\minwalls(\gamma) - 1|} \le 2^{2D} = 4^D$, we conclude that the
  number of walls in $\walls(\gamma)$ intersecting $W_-$ is at most
  $(2D + 1) \cdot 4^D$.
\end{proof}

\begin{proof}[Proof of \Cref{prop:disjoint_walls_I}]
  Consider an itinerary traversing $\gamma$. As in the proof of the
  previous lemma, without loss of generality we may assume that this
  itinerary joins $\identity$ to $\gamma$, and contains exactly the
  walls in $\walls(\gamma)$. Because of
  \Cref{prop:itinerary_correspondence}, our goal is to find a total
  ordering on $\walls(\gamma)$, compatible with $<$, which gives an
  itinerary $\mc{U}$ of the desired form.

  We will find the desired itinerary $\mc{U}$ iteratively. We let
  $R' = (2D + 1) \cdot 4^D$. Using
  \Cref{lem:minimal_wall_intersections}, we choose a minimal wall
  $W_1 \in \walls(\gamma)$ which intersects at most $R'$ other walls
  in $\walls(\gamma)$. Then, we let $\mc{V}_1$ be an itinerary
  consisting of the set of walls in $\walls(\gamma)$ which intersect
  $W_1$, arranged in an arbitrary order compatible with $<$. We obtain
  an itinerary traversing $\gamma$ of the form
  \[
    \{W_1\}, \mc{V}_1, \mc{V}_1',
  \]
  where $|\mc{V}_1| \le R'$, and every wall in $\mc{V}_1'$ is disjoint
  from $W_1$.

  Using \Cref{lem:minimal_wall_intersections} again, we pick a minimal
  wall $W_2$ in $\mc{V}_1'$ which intersects at most $R'$ walls in
  $\mc{V}_1'$. We let $\mc{V}_2$ be the walls in
  $\mc{V}_1' \minus \{W_2\}$ which intersect $W_2$ (again arranged in
  an arbitrary compatible order), and obtain another equivalent
  itinerary
  \[
    \{W_1\}, \mc{V}_1, \{W_2\}, \mc{V}_2, \mc{V}_2'.
  \]
  We proceeed iteratively in this fashion until we have eventually
  obtained an itinerary $\mc{U}$ of the form
  \[
    \{W_1\}, \mc{V}_1, \{W_2\}, \mc{V}_2, \ldots, \{W_n\}, V_n,
  \]
  such that each $\mc{V}_k$ satisfies $|\mc{V}_k| \le R'$, and for any
  $k, \ell$ with $k < \ell$, the wall $W_k$ is disjoint from both
  $W_\ell$ and every wall in $\mc{V}_\ell$.

  So, as long as we take $R \ge R' = (2D + 1) \cdot 4^D$, the
  itinerary $\mc{U}$ satisfies conditions \ref{item:smallsep},
  \ref{item:v_intersect}, and \ref{item:disjoint_daviscx} in the
  statement of the proposition. It remains to show that for some
  choice of $R$, this itinerary also satisfies condition
  \ref{item:few_intersections}. We claim that taking $R = R'D + D$ is
  sufficient.

  To see this, fix $k$, and consider the set $I(W_k)$ of walls in
  $\walls(\gamma)$ which intersect $W_k$. We wish to show that
  $|I(W_k)| \le R'D + D$. We know that $W_k$ is disjoint from every
  $W_i$ for $i \ne k$ and from every wall in $\mc{V}_\ell$ for every
  $\ell > k$. So, every wall in $I(W_k)$ is contained in some
  $\mc{V}_i$ for $i \le k$. Since each $\mc{V}_i$ contains at most
  $R'$ walls, there are at most $R'D$ walls contained in the union
  \[
    \bigcup_{j = k - D + 1}^k I(W_k) \cap \mc{V}_j.
  \]
  So, if $I_{k-D}$ denotes the set of walls
  \[
    I_{k-D} = \bigcup_{i=1}^{k - D} I(W_k) \cap \mc{V}_i,
  \]
  we must have $|I(W_k)| \le |I_{k-D}| + R'D$.

  For any $i, j$ with $i < j < k$, we have $W_i < W_j < W_k$. This
  means that if some wall $V$ is incomparable with both $W_i$ and
  $W_k$, it is also incomparable with $W_j$. Now, if
  $V \in I_{k - D}$, we know that $V \cap W_i$ is nonempty for some
  $i \le k - D$, and $V \cap W_k$ is nonempty by assumption, so
  necessarily $V$ intersects $W_j$ for every $j$ with $i \le j \le
  k$. In particular, $V$ intersects $W_j$ for every $j$ with
  $k - D \le j \le k$.

  That is, every wall in $I_{k-D}$ intersects each of the $D + 1$
  walls $W_{k-D}, \ldots, W_k$. But then the fact that $\gamma$ has
  $D$-bounded product projections implies that $|I_{k-D}| \le D$,
  hence $|I(W_k)| \le R'D + D$ as required.
\end{proof}

\begin{cor}
  \label{cor:disjoint_walls}
  The itinerary $\mc{U}$ coming from \Cref{prop:disjoint_walls_I}
  satisfies the following properties:
  \begin{enumerate}[label=(\arabic*)]
  \item\label{item:elts_in_coxbnhd} For every $i < j$, the group element
    $\gamma_{\mc{U}}(W_i, W_j)$ satisfies
    \[
      \gamma_{\mc{U}}(W_i, W_j) = \eta_i \cdot \gamma(W_i, W_j) \cdot
      \eta_j,
    \]
 where $|\eta_i|, |\eta_j| < R$.
\item\label{item:subitineraries_minimal} Suppose $\mc{U}'$ is
  equivalent to $\mc{U}$. Then any sub-itinerary of $\mc{U}'$ with
  length greater than $2R$ is equivalent to an itinerary of the form
    \[
      \mc{Y}_i, \mc{W}_{ij}, \mc{Y}_j,
    \]
    for some $i \le j$, where $\mc{W}_{ij}$ is an efficient itinerary
    between $W_i$ and $W_j$, and $|\mc{Y}_i|, |\mc{Y}_j| < R$.
  \end{enumerate}
\end{cor}
\begin{proof}
  \ref{item:elts_in_coxbnhd} Fix $i < j$, and consider the
  sub-itinerary $\mc{U}_{ij} = \{W_i\}, \mc{V}_i, \ldots \{W_j\}$. We
  know that at most $R$ walls in this sub-itinerary intersect $W_i$,
  and at most $R$ walls in this sub-itinerary intersect $W_j$, so
  there is an equivalent itinerary of the form
  $\mc{Y}_i, \{W_i\}, \mc{Y}, \{W_j\}, \mc{Y}_j$, where
  $|\mc{Y}_i|, |\mc{Y}_j| < R$, and the walls in $\mc{Y}_i$ and
  $\mc{Y}_j$ are precisely the walls in $\mc{U}_{ij}$ which
  respectively intersect $W_i$ and $W_j$. Then the walls in $\mc{Y}$
  are precisely the walls separating $W_i$ from $W_j$, so
  \[
    \gamma(\mc{U}_{ij}) = \gamma(\mc{Y}_i)\gamma(W_i,
    W_j)\gamma(\mc{Y}_j).
  \]
  \ref{item:subitineraries_minimal} Let $\mc{U}'$ be equivalent to
  $\mc{U}$, and let $\mc{Y}$ be a sub-itinerary of $\mc{U}'$. Let
  $\mathbf{W}$ be the ordered set of walls $W_1 < \ldots <
  W_n$. First, suppose that $\mc{Y}$ contains at least one wall
  $W_i \in \mathbf{W}$. We can then choose $W_i \le W_j$ to be
  (respectively) minimal and maximal walls in
  $\mathbf{W} \cap \mc{Y}$; then the argument from the previous case
  shows that $\mc{Y}$ is equivalent to an itinerary
  $\mc{Y}_i, \mc{W}_{ij}, \mc{Y}_j$ as required. So, we just need to
  show that if $|\mc{Y}| > 2R$ then $\mc{Y}$ contains at least one
  wall in $\mathbf{W}$.

  To see this, we consider the total orderings of the walls in
  $\mc{U}$ given by the equivalent itineraries $\mc{U}$ and $\mc{U}'$;
  both of these total orderings must be compatible with the partial
  order $<$ on all the walls in $\mc{U}$. For each
  $W_i \in \mathbf{W}$, we let $n(W_i)$ denote the index of $W_i$ in
  $\mc{U}$, and let $n'(W_i)$ denote the index of $W_i$ in
  $\mc{U}'$. Now, since each $W_i$ is independent from at most $R$
  other walls in $\mc{U}$ with respect to the partial order $<$, we
  must have $|n(W_i) - n'(W_i)| \le R$ for every $i$. Then, since
  every sub-itinerary of $\mc{U}$ with length at least $R$ contains at
  least one wall in $\mathbf{W}$, the same is true for every
  sub-itinerary of $\mc{U}'$ with length at least $2R$.
\end{proof}

\section{Reflection groups acting on convex projective domains}

Let $C$ be a right-angled Coxeter group with generating set $S$. In
this section, we discuss the theory of representations
$\rho\colon C \to \SL^\pm(|S|, \R)$ which are \emph{generated by linear
  reflections}, as studied by Vinberg \cite{Vinberg1971}. Such
representations give rise to a ``projective model'' for the Davis
complex of $(C, S)$, in the form of a \emph{convex projective domain}
$\Omega$ preserved by $\rho$. We refer to \cite{Vinberg1971},
\cite{DGKLM} for further background.

\subsection{Cartan matrices and simplicial representations}
\label{sec:cartan_simplicial}

Let $V$ be a real vector space with dimension $d$.

\begin{definition}
  A \emph{linear reflection} is an element in $\SL^\pm(V)$ with a
  $(d-1)$-dimensional eigenspace with eigenvalue 1, and a
  1-dimensional eigenspace with eigenvalue $-1$. Equivalently, a
  linear reflection is any element $r \in \SL^\pm(V)$ which can be
  written $r = \identity - v \ten \alpha$, where $\alpha \in V^*$,
  $v \in V$, and $\alpha(v) = 2$.

  We refer to the $-1$-eigenspace of $r$ (which is spanned by the
  vector $v$) as the \emph{polar} of $r$. The $1$-eigenspace of $r$,
  given by $\ker(\alpha)$, is the \emph{reflection hyperplane} of $r$.
\end{definition}

The vector $v$ and dual vector $\alpha$ are determined up to a choice
of scale: we can replace $v$ with $\lambda v$ and $\alpha$ with
$\lambda^{-1}\alpha$ for any nonzero $\lambda \in \R$ to obtain the
same reflection.

We say that a representation $\rho\colon C \to \SL^\pm(V)$ is
\emph{generated by reflections} if $\rho$ maps each $s \in S$ to some
linear reflection in $\SL^\pm(V)$. There is always at least one
discrete faithful representation $\rho\colon C \to \SLpm(|S|, \R)$ which is
generated by reflections, called the \emph{geometric
  representation}. In fact, whenever $C$ is an infinite right-angled
Coxeter group, there is an uncountable (continuous) family of
conjugacy classes of such representations.

Below, we explain how to construct the representations in this family.

\begin{definition}
  Let $(C, S)$ be a right-angled Coxeter system, and write
  $S = \{s_1, \ldots, s_n\}$. We say that an $|S| \times |S|$ real
  matrix $A$ is a \emph{Cartan matrix} for $(C, S)$ if it satisfies
  the following three criteria:
  \begin{enumerate}
  \item For every $i$, we have $A_{ii} = 2$.\label{cond:compatibility_diagonal}
  \item For all $i \ne j$ such that $s_i$ and $s_j$ commute, we have
    $A_{ij} = 0$. \label{cond:compatibility_commuting}
  \item For all $i \ne j$ such that $s_i$ and $s_j$ do \emph{not}
    commute, we have $A_{ij} < 0$ and
    $A_{ij}A_{ji} \ge 4$.\label{cond:compatibility_not_commuting}
  \end{enumerate}
\end{definition}

\begin{remark}
  It also makes sense to consider Cartan matrices for an   \emph{arbitrary} (i.e.\ not necessarily right-angled) Coxeter group $C$. The definition is slightly more   complicated in this case. We omit it as it is not relevant for the present paper.
\end{remark}

We can use any Cartan matrix $A$ for a Coxeter group $C$ to define a
representation $\rho_A\colon C \to \SLpm(|S|, \R)$: we let $n = |S|$, let
$\{e_1, \ldots, e_n\}$ be the standard basis for $\R^n$, and let
$\{e^1, \ldots, e^n\}$ be the corresponding dual basis. For each $i$,
we set $\alpha_i = e^i$, and let $v_i$ be the vector
\[
  v_i = \sum_{k=1}^n A_{ki}e_k.
\]
Then the group element $\identity - v_i \ten \alpha_i$ is a linear
reflection and for every $i,j$ we have $\alpha_i(v_j) = A_{ij}$. One
can check directly that the assignment
$s_i \mapsto (\identity - v_i \ten \alpha_i)$ determines a
representation of the Coxeter group $C$.

\begin{definition}
  \label{defn:cartan_matrix_rep}
  Given a Cartan matrix $A$ for a Coxeter group $C$, we let
  $\rho_A\colon C \to \SLpm(n, \R)$ denote the representation determined by
  the assignment
  \[
    s_i \mapsto (\identity - v_i \ten \alpha_i)
  \]
  described above.

  We will refer to $\rho_A$ as the \emph{simplicial representation}
  associated to $A$. The terminology is motivated by the fact that
  $\rho_A$ induces a discrete and faithful action of $C$ on a convex
  domain in projective space $\P(\R^n)$, with fundamental domain a
  simplex (see \Cref{thm:vinberg_action} below).
\end{definition}

\begin{definition}\label{defn:Tits_rep}
  Fix a right-angled Coxeter system $(C, S)$, and let $A$ be the
  unique Cartan matrix for $(C, S)$ which satisfies
  $A_{ij} = A_{ji} = -2$ for every $i \ne j$ such that $s_i$ and $s_j$
  do not commute. The simplicial representation $\rho_A$ associated to
  this Cartan matrix is called the \emph{geometric} (or \emph{Tits})
  representation of $C$.
\end{definition}

The geometric representation was first studied by Tits
\cite{bourbaki68}, who proved that it is always faithful with discrete
image in $\SLpm(|S|, \R)$. In \cite{Vinberg1971}, Vinberg proved that
the same is true for every simplicial representation $\rho$.

\begin{remark}
  If the Cartan matrix $A$ is symmetric, then one can define a
  (possibly degenerate) bilinear form $\langle \cdot, \cdot \rangle$
  on $\R^{|S|}$ by setting $\langle e_i, e_j \rangle = A_{ij}$. In
  this case, the simplicial representation $\rho_A$ defined above is
  isomorphic to the \emph{dual} of the representation
  \[
    s_i \mapsto (\identity - \langle \cdot, e_i \rangle e_i),
  \]
  which preserves the bilinear form $\langle \cdot , \cdot
  \rangle$. If the form $\langle \cdot, \cdot \rangle$ is
  nondegenerate, then it induces an isomorphism from $\R^{|S|}$ to its
  dual, and $\rho_A$ preserves the corresponding (nondegenerate)
  bilinear form (meaning $\rho_A(C)$ has image lying in some $O(p,q)$
  with $p + q = |S|$).
\end{remark}

\subsubsection{Nondegeneracy conditions}

Several times in this paper, we will want to consider the restriction
of a simplicial representation $\rho$ of a right-angled Coxeter group
$C$ to some standard subgroup $C(T) \leq C$. When $C(T)$ is a
proper standard subgroup, the restricted representation cannot be
irreducible, as it has an invariant subspace
$V_T = \spn\{v_s : s \in T\}$. So, for each subset $T \subseteq S$, we
let $\rho_T\colon C(T) \to \SLpm(V_T)$ denote the representation induced by
the restriction of $\rho$ to $C(T)$.

To facilitate inductive arguments, we would like to have a condition
which guarantees that the representation $\rho_T$ is isomorphic to a
simplicial representation of $C(T)$, and inherits some nondegeneracy
properties of the original representation $\rho$.

We know that the representation $\rho_T$ will be isomorphic to a
simplicial representation precisely when the set of restrictions
$\{\alpha_t|_{V_T} : t \in T\}$ is a basis for $V_T^*$. The Cartan
matrix associated to $\rho_T$ is then the principal submatrix of the
Cartan matrix for $\rho$ corresponding to the subset $T$. With this in
mind, we define the following:
\begin{definition}
  We will say that a matrix $A$ is \emph{fully nondegenerate} if all
  of its principal minors are nonzero.
\end{definition}

The lemma below is completely elementary, but key to several of our
later arguments.
\begin{lem}
  \label{lem:fully_nondegenerate_transverse}
  Let $\rho\colon C \to \SL_\pm(|S|, \R)$ be a simplicial representation of
  a right-angled Coxeter group $C$ with fully nondegenerate Cartan
  matrix. For any subset $T \subset S$, define
  $V_T = \spn\{v_s : s \in T\}$ and
  $V_T^\perp = \bigcap_{s \in T} \ker(\alpha_s)$. Then there is a
  $C(T)$-invariant decomposition $V = V_T \oplus V_T^\perp$, and the
  representation $\rho_T\colon C(T) \to \SLpm(V_T)$ is isomorphic to a
  simplicial representation.
\end{lem}
In the terminology of \cite{DGKLM},
\Cref{lem:fully_nondegenerate_transverse} says that the simplicial
representation $\rho$ associated to a fully nondegenerate Cartan
matrix is both \emph{reduced} and \emph{dual-reduced}, and that the
same is true for the representations
$\rho_T\colon C(T) \to \SLpm(V_T)$ for every $T \subseteq S$.

\begin{proof}
  The nondegeneracy of the Cartan matrix implies that the subsets
  $\{v_s : s \in T\}$ and $\{\alpha_s : s \in T\}$ are both linearly
  independent in $\R^{|S|}$ and its dual, respectively. This implies
  that the subspaces $V_T$ and $V_T^\perp$ have complementary
  dimension. Since the principal minor of the Cartan matrix
  corresponding to the subset $T$ is nonzero, the set of restrictions
  $\{\alpha_s|_{V_T} : s \in T\}$ is also linearly independent in the
  dual $V_T^*$. This implies that $V_T$ and $V_T^\perp$ are
  transverse, and that $\rho_T$ is isomorphic to a simplicial
  representation.
\end{proof}

\begin{remark}
  \label{rem:fully_nondegenerate_exists}
  Given a right-angled Coxeter group $C$, one can always find a fully
  nondegenerate Cartan matrix $A$ for $C$. In fact, the space of fully
  nondegenerate matrices is open and dense in the space of Cartan
  matrices for $C$, since each principal minor of $A$ is a polynomial
  in the parameters determining $A$ which is not identically zero. If
  desired, one can also arrange for this fully nondegenerate matrix to
  be symmetric, or to have integer entries.
\end{remark}

\subsection{Invariant convex projective domains}

Let $V$ be a real vector space. Recall that a subset
$\tilde{\Omega} \subset V$ is a \emph{convex cone} if it is convex and
invariant under multiplication by positive real numbers. A
\emph{convex domain} in $\P(V)$ is the image of an open convex cone
under the projectivization map $V \minus \{0\} \to \P(V)$. A convex
domain is \emph{properly convex} if its closure is a convex subset of
some affine chart in $\P(V)$ (equivalently, if its closure does not
contain any projective line).

Tits and Vinberg proved that simplicial representations are discrete
and faithful by showing that there is a certain $\rho$-invariant
convex domain in $\P(\R^{|S|})$ with nonempty interior, and that the
action of $C$ on this domain is faithful and properly
discontinuous. The structure of this domain is essential to this
paper, so we describe it below.
\begin{definition}
  \label{defn:vinberg_domain}
  Let $A$ be a Cartan matrix for a Coxeter group $C$, which determines
  vectors $v_s \in \R^{|S|}$ and dual vectors
  $\alpha_s \in (\R^{|S|})^*$ for each $s \in S$.
  \begin{itemize}
  \item The \emph{fundamental simplicial cone} for the representation
    $\rho_A$ is the set
    \[
      \tilde{\Delta} = \bigcap_{s \in S} \{v \in \R^{|S|} :
      \alpha_s(v) \le 0 \textrm{ for all } s \in S\}.
    \]
    The \emph{fundamental simplex} $\Delta$ is the projectivization of
    $\tilde{\Delta}$ in $\P(\R^{|S|})$.
  \item The \emph{Tits cone} is the interior of the set
    \[
     \bigcup_{\gamma \in C} \rho_A(\gamma) \tilde{\Delta}.
   \]
    The \emph{Vinberg domain} is the projectivization of the Tits cone
    in $\P(\R^{|S|})$. We usually denote the Tits cone by
    $\uvindomain$, and the Vinberg domain by $\vindomain$.
  \end{itemize}
\end{definition}

\begin{thm}[See Tits \cite{bourbaki68}, Vinberg \cite{Vinberg1971}]
  \label{thm:vinberg_action}
  Let $\rho$ be a simplicial representation of a Coxeter group
  $C$. Then:
  \begin{enumerate}[label=(\roman*)]
  \item The Tits cone $\uvindomain$ is a nonempty convex open subset
    of $\R^{|S|}$;
  \item the action of $C$ on $\vindomain$ is faithful and properly
    discontinuous;
  \item the simplex $\Delta \cap \vindomain$ is a fundamental domain
    for the action;
  \item the dual graph to the tiling of $\vindomain$ by copies of
    $\Delta \cap \vindomain$ is equivariantly identified with the
    Cayley graph of $C$ with generating set $S$.
  \end{enumerate}
\end{thm}

\begin{remark}
  We have only stated \Cref{thm:vinberg_action} for simplicial
  representations of the Coxeter group $C$. However, Vinberg's result
  also applies to a broad family of representations of $C$ generated
  by linear reflections. In particular, a version of the theorem
  applies to representations $\rho\colon C \to \SLpm(V)$ where $\dim V$ is
  not necessarily equal to $|S|$.
\end{remark}

While the Vinberg domain $\vindomain$ associated to a simplicial
representation $\rho$ is always a convex domain, it is not necessarily
\emph{properly} convex. However, Vinberg also gave conditions which
make it possible to guarantee that this holds in broad circumstances:
\begin{prop}[{See \cite[Lemma 15 and Proposition 22]{Vinberg1971}}]
  \label{prop:vinberg_properly_convex}
  Let $A$ be a Cartan matrix for an irreducible Coxeter group $C$, and
  suppose that $\det(A) \ne 0$. Then the following are equivalent:
  \begin{enumerate}
  \item The matrix $A$ is of \emph{negative type}, i.e.\ $A$ has a
    negative eigenvalue.
  \item The Vinberg domain $\vindomain$ for $\rho_A$ is properly
    convex.
  \item The Coxeter group $C$ is infinite.
  \end{enumerate}
\end{prop}

Sometimes, when working with representations of $C$ generated by
reflections, it will also be convenient to consider $\rho$-invariant
domains $\Omega \subset \P(\R^{|S|})$ \emph{other} than the Vinberg
domain.
\begin{definition}
  Let $\rho\colon C \to \SLpm(|S|, \R)$ be a simplicial representation
  of an irreducible right-angled Coxeter group $C$. We will call any
  $\rho$-invariant nonempty convex domain $\Omega \subset \vindomain$
  a \emph{reflection domain} for $\rho$.
\end{definition}

Typically, we lose nothing by requiring reflection domains to lie
inside of the Vinberg domain $\vindomain$, instead of merely asking
for them to be $\rho$-invariant convex open sets in
$\P(\R^{|S|})$. The reason is the following fact, due to
Danciger-Gu\'eritaud-Kassel-Lee-Marquis:
\begin{prop}[{See \cite[Proposition 4.1]{DGKLM}}]
  \label{prop:vinberg_maximal}
  Let $(C,S)$ be an irreducible infinite right-angled Coxeter system
  with $|S| > 2$, and let $\rho$ be a simplicial representation of $C$
  with nonsingular Cartan matrix. Then $\vindomain$ contains every
  $\rho$-invariant properly convex domain in $\P(\R^{|S|})$.
\end{prop}

\subsection{Dual domains and the dual Vinberg domain}
\label{sec:dual_domains}

Let $V$ be a real vector space and let $\tilde{\Omega}$ be an open
convex cone in $V$. The cone $\tilde{\Omega}$ determines a \emph{dual
  cone} $\tilde{\Omega}^* \subset \P(V^*)$, defined by
\[
  \tilde{\Omega}^* = \{w \in V^* : w(v) < 0 \quad \forall x \in
  \overline{\tilde{\Omega}} \minus \{0\}\}.
\]
Then if $\Omega \subset \P(V)$ is some convex domain, the dual domain
$\Omega^* \subset \P(V^*)$ is the projectivization of
$\tilde{\Omega}^*$, for some (any) convex cone $\tilde{\Omega}$
projecting to $\Omega$.

It follows immediately that $\Omega^*$ is invariant under the dual
action of any $g \in \SLpm(V)$ which preserves $\Omega$. In addition,
it is not hard to verify that if $\Omega$ is both properly convex and
open, then so is $\Omega^*$. We note further that duality reverses
inclusions: if $\Omega_1 \subset \Omega_2$, then
$\Omega_2^* \subset \Omega_1^*$.

\subsubsection{Dual simplicial representations}

We now let $V$ be the vector space $\R^{|S|}$, and let
$\rho\colon C \to \SLpm(V)$ be a simplicial representation for a
right-angled Coxeter group $C$. If the associated Cartan matrix is
nonsingular, then the dual representation $\rho^*:C \to \SLpm(V^*)$ is
isomorphic to a simplicial representation of $C$ whose Cartan matrix
is the transpose of the Cartan matrix of $\rho$. Since $\rho$
preserves the Vinberg domain $\vindomain$, the dual representation
preserves the dual domain $\vindomain^*$.

When $C$ is irreducible, infinite, and not virtually cyclic,
\Cref{prop:vinberg_maximal} implies that $\vindomain^*$ is contained
in the Vinberg domain $\dvindomain$ associated to the dual
(simplicial) representation $\rho^*$. Explicitly, if we let
$\tilde{\mathscr{D}}$ be the cone
\[
  \{w \in V^* : v_s(w) \le 0 \quad \forall s \in S\},
\]
and let $\mathscr{D}$ be the projectivization of
$\tilde{\mathscr{D}}$, then $\dvindomain$ is the projectivization of
the interior of
\[
  \udvindomain:= \bigcup_{\gamma \in C} \rho^*(\gamma)
  \tilde{\mathscr{D}}.
\]
\Cref{thm:vinberg_action} also applies to $\rho^*$, with $\mathscr{D}$
taking the place of $\Delta$ and $\dvindomain$ taking the place of
$\vindomain$.

We always have:
\begin{prop}[The dual of a reflection domain is a reflection domain]
  \label{prop:dual_vinberg_domains}
  Let $A$ be a nonsingular Cartan matrix for an infinite irreducible
  right-angled Coxeter group $C$, and let $\rho_A$ be the associated
  simplicial representation. Let $\vindomain, \dvindomain$ be the
  Vinberg domains for $\rho_A, \rho_A^*$, respectively. If
  $\Omega \subseteq \vindomain$ is a reflection domain for $\rho_A$,
  then $\Omega^* \subseteq \dvindomain$.
\end{prop}
When $C$ is not virtually cyclic this result follows from
\Cref{prop:vinberg_maximal}, and if $C$ is an infinite dihedral group
this can be checked directly.

\section{Walls and half-cones in reflection domains}

In this section we continue to explore some features of the projective
geometry of reflection domains. Our main goal is to translate our
combinatorial understanding of geodesics in a right-angled Coxeter
group into information about the action of the corresponding sequence
of group elements in $\SLpm(n, \R)$.

\begin{definition}
  Let $C$ be an irreducible right-angled Coxeter group, and let
  $\Omega$ be a reflection domain for a simplicial representation
  $\rho$ of $C$. A \emph{wall} $W$ in $\Omega$ is the fixed-point set
  in $\Omega$ of a reflection $r$ in $C$. Whenever $W$ is a wall, we
  let $\overline{W}$ denote the closure of $W$ in
  $\overline{\Omega}$. Equivalently, $\overline{W}$ is the closure of
  $W$ in projective space $\P(V)$.

  The \emph{polar} of a wall $W$ is the polar of the reflection $r$
  fixing $W$.
\end{definition}

Every wall in $\Omega$ is the intersection of some projective
hyperplane with $\Omega$, so if $\Omega$ is properly convex, then each
wall separates $\Omega$ into two convex components. Further, every
wall $W$ is a union of codimension-1 faces of tiles in $\Omega$,
corresponding exactly to the set of edges in the Cayley graph of $C$
fixed by the reflection that defines $W$. Consequently, we have the
following:
\begin{prop}
  \label{prop:projective_wall_correspondence}
  Let $\Omega$ be a reflection domain for a simplicial representation
  of an infinite irreducible right-angled Coxeter group $C$. There is
  an equivariant one-to-one correspondence between walls in $\Omega$
  and walls in the Davis complex $\daviscx$, which satisfies the
  following properties.
  \begin{enumerate}
  \item Two walls in $\Omega$ intersect if and only if the
    corresponding walls in $\daviscx$ intersect (equivalently, if and
    only if the corresponding reflections in $C$ commute).
  \item Fix a basepoint $x_0$ in the interior of $\Delta \cap \Omega$,
    and let $\gamma \in C$. The set of walls in $\Omega$ separating
    $x_0$ from $\rho(\gamma) \cdot x_0$ corresponds precisely to the
    set of walls in $\daviscx$ separating $\identity$ from $\gamma$.
  \end{enumerate}
\end{prop}

\Cref{prop:projective_wall_correspondence} tells us that the walls in
any reflection domain carry all of the same combinatorial information
as the walls in the Davis complex. So, throughout the rest of this
section, we will freely apply the combinatorial setup and results
regarding walls in Sections \ref{sec:cube_complexes_RACGs} and
\ref{sec:bounded_product_projections} to the walls in a reflection
domain $\Omega$.

In particular, for any group element $\gamma \in C$, we can regard the
poset of walls in $\daviscx$ separating $\identity$ from $\gamma$ as a
poset of projective walls in $\Omega$. Moreover, the notion of an
\emph{itinerary} traversing an element in $C$ also makes sense in this
context.
\begin{definition}
  If $W_1, \ldots, W_n$ is a sequence of walls in a reflection domain
  $\Omega$ corresponding to an itinerary of walls in the Davis complex
  $\daviscx$, we will say that $W_1, \ldots, W_n$ is an
  \emph{itinerary in $\Omega$} or an \emph{$\Omega$-itinerary}. The
  element $\gamma \in C$ \emph{traversed} by this itinerary is the
  element traversed by the corresponding itinerary in $\daviscx$.
\end{definition}

\subsection{Disjoint walls in $\vindomain$ and
  $\overline{\vindomain}$}

The projective walls in the Vinberg domain $\vindomain$ actually carry
some additional combinatorial information beyond what is encoded in
the corresponding walls in the Davis complex $\daviscx$: the
intersection pattern of the \emph{closures} of the set of walls in
$\overline{\vindomain}$ is also informed by the structure of the
Coxeter group $C$.

Specifically, we have the following:
\begin{lem}
  \label{lem:disjoint_walls}
  Let $C$ be a right-angled Coxeter group, and let $\vindomain$ be the
  Vinberg domain for a simplicial representation of $C$. If $W_1$,
  $W_2$ are walls in $\vindomain$, then
  $\overline{W}_1 \cap \overline{W}_2 \ne \emptyset$ if and only if
  $\gamma(W_1, W_2)$ lies in a proper standard subgroup of $C$.
\end{lem}

\Cref{lem:disjoint_walls} will follow from
\Cref{lem:itinerary_wall_intersect} and \Cref{lem:wall_intersect_gen}
below. For these two results, we assume that $\rho$ is a simplicial
representation for a right-angled Coxeter system $(C, S)$.
\begin{lem}
  \label{lem:itinerary_wall_intersect}
  If $W_1, \ldots, W_n$ is an efficient $\vindomain$-itinerary between
  $W_1$ and $W_n$, then
  $\overline{W_1} \cap \overline{W_n} = \bigcap_{i=1}^n
  \overline{W_i}$.
\end{lem}
\begin{proof}
  First, observe that if $W_1 \cap W_n \ne \emptyset$, then $W_1, W_n$
  is already an efficient itinerary between $W_1$ and $W_n$. So we can
  suppose that $W_1 \cap W_n = \emptyset$. It suffices to show that
  $\overline{W_1} \cap \overline{W_n} \subset \overline{W_i}$ for
  every $1 < i < n$, so fix a wall $W_i$ strictly between $W_1$ and
  $W_n$. Since $W_1, \ldots, W_n$ is an efficient itinerary, $W_i$
  separates $W_1$ from $W_n$: that is, $\vindomain \minus W_i$ has
  exactly two connected components $O_-, O_+$, with $W_1 \subset O_-$
  and $W_n \subset O_+$.

  By convexity, $\overline{\vindomain} \minus \overline{W_i}$ also has
  exactly two components $O_-', O_+'$, which satisfy
  $\overline{O_-} = \overline{O_-'}$ and
  $\overline{O_+} = \overline{O_+'}$ (as for walls, the closures are
  taken in projective space). We have
  $\overline{W_1} \subset \overline{O_-}$ and
  $\overline{W_n} \subset \overline{O_+}$, and therefore
  $\overline{W_1} \cap \overline{W_n} \subset \overline{O_-'} \cap
  \overline{O_+'} = \overline{W_i}$.
\end{proof}

\begin{lem}
  \label{lem:wall_intersect_gen}
  Let $W_1, \ldots, W_n$ be an $\vindomain$-itinerary departing from
  the identity, and let $s_i = \type{W_i}$ for $1 \le i \le n$. Then
  $\bigcap_{i=1}^n \overline{W(s_i)} = \bigcap_{i=1}^n
  \overline{W_i}$.
\end{lem}
\begin{proof}
  We induct on $n$. The case $n = 1$ is immediate. When $n > 1$, we
  let $\gamma_{n-1} = s_1 \cdots s_{n-1}$. By
  \Cref{prop:wall_expression} we have
  $W_n = \rho(\gamma_{n-1})W(s_n)$, so inductively we have
  \begin{align*}
    \bigcap_{i=1}^n \overline{W_i}
    = \left(\bigcap_{i=1}^{n-1} \overline{W_i} \right) \cap
    \overline{W_n} = \left(\bigcap_{i=1}^{n-1} \overline{W(s_i)} \right) \cap
    \rho(\gamma_{n-1}) \overline{W(s_n)}.
  \end{align*}
  Since $\rho(s_i)$ fixes $\overline{W(s_i)}$ pointwise for each
  $1 \le i \le n$, both $\rho(\gamma_{n-1})$ and
  $\rho(\gamma_{n-1}^{-1})$ fix the intersection
  $\bigcap_{i=1}^{n-1} \overline{W(s_i)}$ pointwise. So, we have
  \begin{align*}
    \left(\bigcap_{i=1}^{n-1} \overline{W(s_i)} \right) \cap
    \rho(\gamma_{n-1}) \overline{W(s_n)}
    &= \rho(\gamma_{n-1})\left(
      \left(\rho(\gamma_{n-1}^{-1})\bigcap_{i=1}^{n-1}\overline{W(s_i)}\right)
      \cap \overline{W(s_n)}\right)\\
    &= \rho(\gamma_{n-1}) \bigcap_{i=1}^n \overline{W(s_i)}
      = \bigcap_{i=1}^n \overline{W(s_i)}.
  \end{align*}
\end{proof}

\begin{proof}[Proof of \Cref{lem:disjoint_walls}]
  Fix an efficient itinerary $\mc{V} = V_1, \ldots, V_n$ whose first
  wall is $W_1$ and whose last wall is $W_2$. For any $h \in C$ we
  have $\elt(W_1, W_2) = \elt(hW_1, hW_2)$. So, after translating
  $W_1, W_2$ and $\mc{V}$ by an appropriate element $h$, we can assume
  that $\mc{V}$ departs from the identity.

  Set $s_i = \type{V_i}$ for $1 \le i \le n$, so that
  $\elt(W_1, W_2) = s_1 \cdots s_n$. From
  \Cref{lem:itinerary_wall_intersect} and
  \Cref{lem:wall_intersect_gen}, we know that
  $\overline{W_1} \cap \overline{W_2}$ is nonempty if and only if the
  intersection $\bigcap_{i=1}^n \overline{W(s_i)}$ is nonempty. But
  since the representation $\rho$ is simplicial, this occurs if and
  only if the set $S'$ of generators appearing in the geodesic word
  $s_1 \cdots s_n$ is a proper subset of $S$, i.e.\ if
  $\elt(V_1, V_n) = s_1 \cdots s_n$ lies in a proper standard subgroup
  of $C$.
\end{proof}

\subsection{Half-cones}

Consider an infinite geodesic sequence $\gamma_n$ in a right-angled
Coxeter group $C$. If we fix a reflection domain $\Omega$ for some
simplicial representation $\rho$ of $C$, the geodesic sequence
corresponds to an (infinite) $\Omega$-itinerary
$\mc{W} = W_1, W_2, \ldots$

Each wall $W_n$ in $\mc{W}$ cuts $\Omega$ into a pair of connected
components, which we call \emph{half-spaces}; we can fix a basepoint
$x_0 \in \Omega$, and let $H_n$ be the half-space which does
\emph{not} contain $x_0$.

By considering the infinite intersection of closed half-spaces
$\bigcap_{n=1}^\infty \overline{H_n}$, we can try and find a
well-defined ``limit point'' for the geodesic $\rho(\gamma_n)$ in
projective space $\P(V)$. However, we will run into some difficulties:
while it is often true that pairs of half-spaces are \emph{nested},
(meaning that $H_{n+k} \subset H_n$ for some $n, k > 0$), they will
never be \emph{strongly} nested, i.e.\ they will \emph{never} satisfy
$\overline{H_{n+k}} \subset H_n$. This makes it hard to guarantee that
the intersections $\bigcap_{i=1}^n \overline{H_n}$ decrease in size at
a ``uniform rate.''

We solve this problem by enlarging  the half-spaces $H_n$ to
subsets of projective space which we call
\emph{half-cones}. Half-cones satisfy the same nesting properties as
half-spaces (see \Cref{lem:halfcones_nest} below). They will often
(but not always) additionally satisfy the \emph{strong} nesting
property as well, which allows us to employ them in asymptotic
arguments later on. We will also be able to give a fairly complete
description of when the strong nesting property fails (see
\Cref{lem:halfcone_nesting_failure}), which will be key to later
inductive arguments.

For the precise definitions, we fix an irreducible infinite
right-angled Coxeter group $C$, let $V$ denote the vector space
$\R^{|S|}$, and let $\rho\colon C \to \SLpm(V)$ be a simplicial
representation of $C$. We let $\Omega \subset \vindomain$ be a
reflection domain for $\rho$, so that $\Omega$ is the projectivization
of an invariant convex sub-cone $\tilde{\Omega} \subset \uvindomain$
of the Tits cone.

We let $\tilde{\Delta}_\Omega$ denote the set
$\tilde{\Delta} \cap \tilde{\Omega}$. Then the projectivization
$\Delta_\Omega$ of $\tilde{\Delta}_\Omega$ is a fundamental domain for
the action of $C$ on $\Omega$.

\begin{definition}
  Let $W$ be a wall in $\Omega$, so that $W$ is the fixed-point set of
  a reflection $\identity - v \ten \alpha$ for some $v \in V$,
  $\alpha \in V^*$ with $\alpha(v) = 2$. Up to replacing $\alpha, v$
  with $-\alpha, -v$, we may assume that $\alpha(x) \le 0$ for
  every $x \in \tilde{\Delta}_\Omega$.

  \begin{itemize}
  \item We let $\uhalfspace_+(W)$ denote the interior of the convex
    cone $\{x \in \tilde{\Omega} : \alpha(x) \ge 0\}$. The
    \emph{positive half-space} over $W$, denoted $\halfspace_+(W)$, is
    the projectivization of $\uhalfspace_+(W)$. Similarly,
    $\uhalfspace_-(W)$ is the interior of the cone
    $\{x \in \tilde{\Omega} : \alpha(x) \le 0\}$, and the
    \emph{negative half-space} $\halfspace_-(W)$ is the
    projectivization of $\uhalfspace_-(W)$.
  \item We let $\uhalfcone_+(W)$ denote the interior of the convex
    hull (in $V$) of $\ker(\alpha) \cap \tilde{\Omega}$ and $v$. The
    \emph{positive half-cone} over $W$, denoted $\halfcone_+(W)$, is
    the projectivization of $\uhalfcone_+(W)$. We define the set
    $\uhalfcone_-(W)$ and the \emph{negative half-cone}
    $\halfcone_-(W)$ in the same way, but with $v$ replaced with
    $-v$. See \Cref{fig:2ii_halfcone}.
  \end{itemize}
\end{definition}

For $s \in S$, we let $\halfspace_\pm(s)$ and $\halfcone_\pm(s)$
denote the positive and negative half-spaces and half-cones over the
wall $W(s)$ fixed by $\rho(s)$.

Note that the positive half-space $\halfspace_+(W)$ over a wall $W$ is
precisely the connected component of $\Omega \minus W$ whose closure
does \emph{not} contain the fundamental domain $\Delta_\Omega$. This
also gives a way to distinguish positive and negative half-cones (see
\Cref{lem:halfspaces_in_halfcones} below).

\begin{figure}[h]
  \centering
  \def\svgwidth{.8\textwidth}
\begingroup%
  \makeatletter%
  \providecommand\color[2][]{%
    \errmessage{(Inkscape) Color is used for the text in Inkscape, but the package 'color.sty' is not loaded}%
    \renewcommand\color[2][]{}%
  }%
  \providecommand\transparent[1]{%
    \errmessage{(Inkscape) Transparency is used (non-zero) for the text in Inkscape, but the package 'transparent.sty' is not loaded}%
    \renewcommand\transparent[1]{}%
  }%
  \providecommand\rotatebox[2]{#2}%
  \newcommand*\fsize{\dimexpr\f@size pt\relax}%
  \newcommand*\lineheight[1]{\fontsize{\fsize}{#1\fsize}\selectfont}%
  \ifx\svgwidth\undefined%
    \setlength{\unitlength}{568.79998779bp}%
    \ifx\svgscale\undefined%
      \relax%
    \else%
      \setlength{\unitlength}{\unitlength * \real{\svgscale}}%
    \fi%
  \else%
    \setlength{\unitlength}{\svgwidth}%
  \fi%
  \global\let\svgwidth\undefined%
  \global\let\svgscale\undefined%
  \makeatother%
  \begin{picture}(1,0.61012661)%
    \lineheight{1}%
    \setlength\tabcolsep{0pt}%
    \put(0,0){\includegraphics[width=\unitlength,page=1]{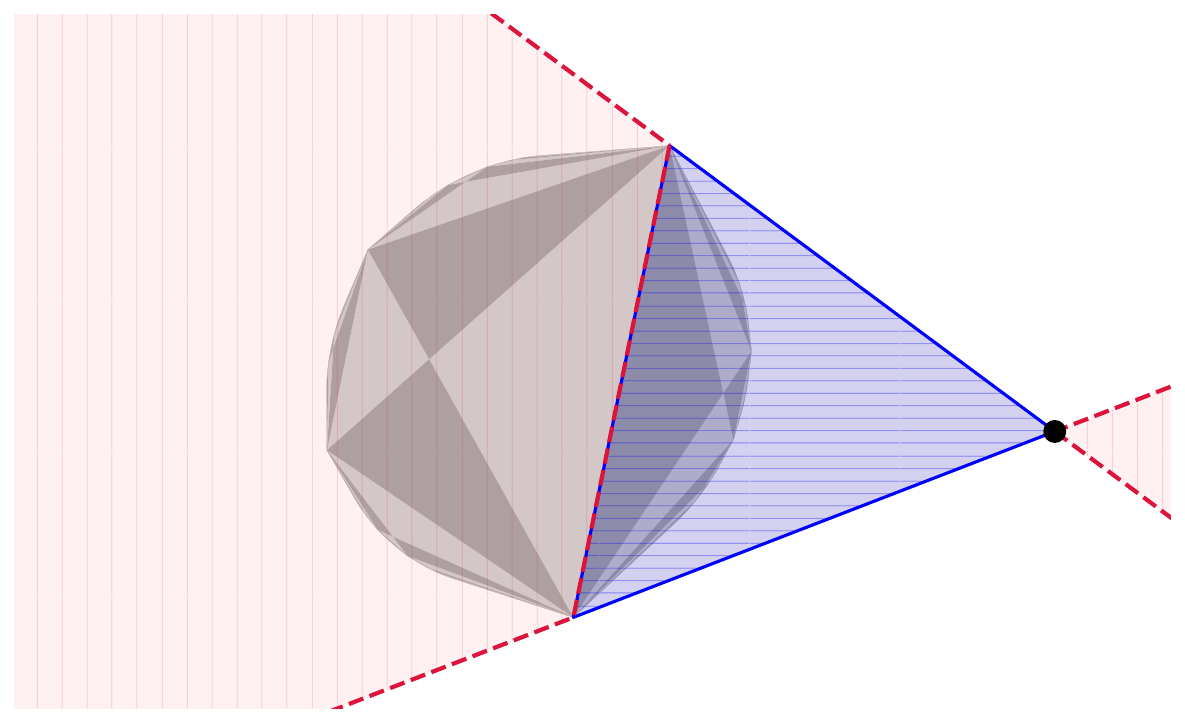}}%
    \put(0.70712025,0.14667726){\color[rgb]{0,0,1}\makebox(0,0)[lt]{\lineheight{1.25}\smash{\begin{tabular}[t]{l}$\halfcone_+(W)$\end{tabular}}}}%
    \put(0.14667722,0.35379749){\color[rgb]{0.8627451,0.07843137,0.23529412}\makebox(0,0)[lt]{\lineheight{1.25}\smash{\begin{tabular}[t]{l}$\halfcone_-(W)$\end{tabular}}}}%
    \put(0.88500002,0.28069621){\makebox(0,0)[lt]{\lineheight{1.25}\smash{\begin{tabular}[t]{l}$v$\end{tabular}}}}%
  \end{picture}%
\endgroup%

  \caption{Positive and negative half-cones over a wall $W$ in a
    Vinberg domain $\vindomain$. The polar of $W$ is $v$. The negative
    half-cone over $W$ is not contained in the depicted affine chart.}
  \label{fig:2ii_halfcone}
\end{figure}
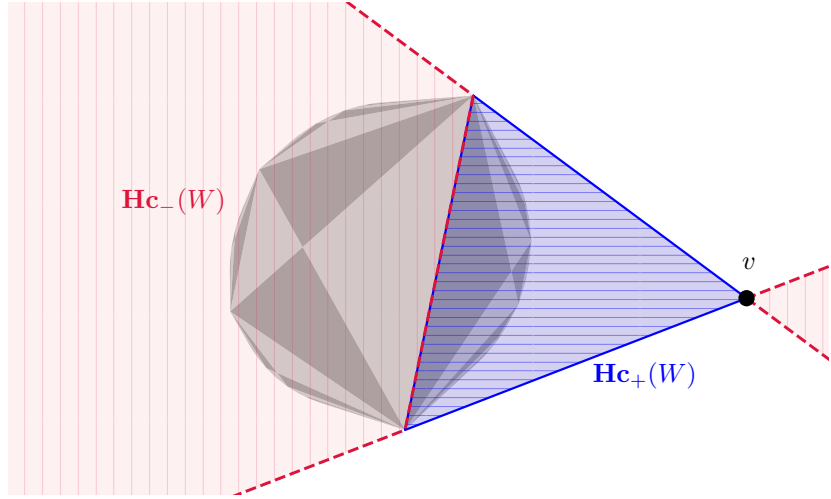

\begin{lem}
  \label{lem:halfcone_formula}
  Let $W_1, \ldots, W_n$ be an $\Omega$-itinerary departing from the
  identity, traversing a geodesic word $s_1 \cdots s_n$ in $C$. Then
  $\halfspace_\pm(W_n) = \rho(s_1 \cdots s_{n-1})\halfspace_\pm(s_n)$
  and
  $\halfcone_\pm(W_n) = \rho(s_1 \cdots s_{n-1})\halfcone_\pm(s_n)$.
\end{lem}
\begin{proof}
  We know from \Cref{prop:wall_expression} that $W_n$ is the wall
  $\rho(s_1 \cdots s_{n-1})W(s_n)$. So, if we fix $\alpha \in V^*$ so
  that $W_n = [\ker \alpha \cap \tilde{\Omega}]$ and
  $\alpha(\tilde{\Delta}_\Omega) \le 0$, we know that either
  $\alpha = \rho^*(s_1 \cdots s_{n-1})\alpha_{s_n}$ or
  $\alpha = -\rho^*(s_1 \cdots s_{n-1})\alpha_{s_n}$.

  Now, since the walls $W_1, \ldots, W_n$ are precisely the walls
  separating $\tilde{\Delta}_\Omega$ from
  $\rho(s_1 \cdots s_n)\tilde{\Delta}_\Omega$, we must have
  \[
    \alpha(\rho(s_1 \cdots s_n)\tilde{\Delta}_\Omega) \ge 0,
  \]
  or equivalently
  $(\rho^*(s_n \cdots s_1)\alpha)(\tilde{\Delta}_\Omega) \ge 0$. So,
  since
  $(\rho^*(s_n)\alpha_{s_n})(\tilde{\Delta}_\Omega) =
  -\alpha_{s_n}(\tilde{\Delta}_\Omega) \ge 0$, we have
  $\alpha = \rho^*(s_1 \cdots s_{n-1})\alpha_{s_n}$. This proves the
  desired result for half-spaces.

  A dual argument then shows that, if $v \in V$ is a lift of the polar
  of $W$ with $\alpha(v) = 2$, then
  $\rho(s_1 \cdots s_{n-1})v = v_{s_n}$, which proves the result for
  half-cones.
\end{proof}

\begin{lem}
  \label{lem:halfspaces_in_halfcones}
  For every wall $W$ in $\Omega$, we have
  $\halfspace_\pm(W) \subset \halfcone_\pm(W)$.
\end{lem}
\begin{proof}
  By \Cref{lem:halfcone_formula}, we just to check that
  $\halfspace_\pm(s) \subset \halfcone_\pm(s)$ for every $s \in S$. In
  fact, since the reflection $\rho(s)$ interchanges $\halfspace_+(s)$
  with $\halfspace_-(s)$, and $\halfcone_+(s)$ with $\halfcone_-(s)$,
  we just need to verify that
  $\halfspace_+(s) \subset \halfcone_+(s)$.

  We write $\rho(s) = \identity - v_s \ten \alpha_s$, and consider
  $x \in \uhalfspace_+(s) \minus \ker(\alpha_s)$. Since $\rho(s)$
  interchanges $\uhalfspace_+(s)$ and $\uhalfspace_-(s)$, we know that
  $\rho(s)x \in \uhalfspace_-(s)$. Then since $\tilde{\Omega}$ is a
  convex cone, the line segment $\ell \subset V$ joining $x$ to
  $\rho(s)x$ lies in $\tilde{\Omega}$. The segment $\ell$ is
  $\rho(s)$-invariant, and it must be transverse to $\ker(\alpha_s)$
  at a point $x_0 \in \ker(\alpha_s) \cap \tilde{\Omega}$. So, the
  endpoints of $\ell$ have the form $x_0 \pm t v_s$ for some $t >
  0$. But then since $\alpha_s(x) > 0$ by assumption we must have
  $x = x_0 + tv_s$, hence $x \in \uhalfcone_+(s)$.
\end{proof}

\subsection{Dual walls and dual half-cones}

\begin{definition}
  Given $\Omega$ a reflection domain for a simplicial representation
  $\rho$, and $W$ a wall in $\Omega$, fixed by a reflection
  $\rho(\gamma)$ for $\gamma \in C$, we write $W^*$ to denote the wall in
  the dual reflection domain $\Omega^*$ fixed by the reflection
  $\rho^*(\gamma)$.
\end{definition}

We observe the following useful consequence of
\Cref{lem:disjoint_walls}:
\begin{lem}
  \label{lem:dual_disjoint_walls}
  Let $C$ be an irreducible right-angled Coxeter group, and let
  $\vindomain$ be the Vinberg domain for a simplicial representation
  $\rho$ of $C$ with nonsingular Cartan matrix. Suppose that
  $W_1, W_2$ are two walls in $\vindomain$ such that
  $\overline{W}_1 \cap \overline{W}_2 = \emptyset$. Then
  $\overline{W_1^*} \cap \overline{W_2^*} = \emptyset$.
\end{lem}
\begin{proof}
  Let $\gamma_1$ and $\gamma_2$ be the reflections in $C$ such that
  $\rho(\gamma_i)$ fixes $W_i$ for $i = 1,2$. By
  \Cref{lem:disjoint_walls}, the walls $W_1$ and $W_2$ have disjoint
  closures if and only if $\gamma(W_1, W_2)$ does not lie in a proper
  standard subgroup of $C$. But this condition depends only on
  $\gamma_1$ and $\gamma_2$, and not on the specific simplicial
  representation $\rho$.

  Further, since the Cartan matrix of $\rho$ is nonsingular, the dual
  representation $\rho^*$ is also a simplicial representation. This
  means that $W_1$ and $W_2$ have disjoint closures if and only if the
  walls in $\dvindomain$ fixed by $\rho^*(\gamma_1)$,
  $\rho^*(\gamma_2)$ have disjoint closures (see
  \Cref{sec:dual_domains}). However, $\dvindomain$ contains
  $\vindomain^*$ (\Cref{prop:dual_vinberg_domains}), which means that
  these two walls contain $W_1^*$ and $W_2^*$.
\end{proof}

If $\tilde{\Omega} \subset \uvindomain$ is the invariant cone over a
reflection domain for $\rho$, then the dual cone $\tilde{\Omega}^*$ is
an invariant sub-cone of $\udvindomain$, projecting to the dual
reflection domain $\Omega^*$.  So, the previous section tells us that
we can define half-spaces $\halfspace_\pm(W^*)$ and half-cones
$\halfcone_\pm(W^*)$ over the wall $W^*$ in $\Omega^*$.

The half-cones $\halfcone_\pm(W)$ are themselves properly convex
domains, which means that we can also define their duals
$\halfcone_\pm(W)^* \subset \P(V^*)$ (see the beginning of
\Cref{sec:dual_domains}).
\begin{lem}
  \label{lem:halfcone_duality}
  For any wall $W$ in a reflection domain $\Omega$, we have
  $\halfcone_+(W)^* = \halfcone_-(W^*)$.
\end{lem}
\begin{proof}
  We will verify that this holds when $W$ is the projective wall $W_s$
  fixed by a reflection $\rho(s)$ for $s \in S$. In this case $W_s^*$
  is the projective wall in $\Omega^*$ fixed by $\rho^*(s)$. We can
  use \Cref{lem:halfcone_formula} (and its dual version) to see that
  the general case follows from this one: let $\mc{W}$ be an
  $\Omega$-itinerary $W_1, \ldots, W_n$ departing from the identity,
  with $W_n = W$, and let $s_1 \cdots s_n$ be the geodesic word
  traversed by $\mc{W}$. Then:
  \begin{align*}
    \halfcone_+(W)^* &= (\rho(s_1 \cdots s_{n-1})\halfcone_+(W_{s_n}))^*\\
                     &= \rho^*(s_1 \cdots s_{n-1})\halfcone_+(W_{s_n})^*\\
                     &= \rho^*(s_1 \cdots s_{n-1})\halfcone_-(W_{s_n}^*)\\
                     &= \halfcone_-(W^*).
  \end{align*}
  So, now fix $s \in S$. By definition, we have
  \[
    \uhalfcone_+(W_s) = \{x + tv_s : x \in \ker(\alpha_s) \cap
    \tilde{\Omega}, t > 0\},
  \]
  which means that $\uhalfcone_+(W_s)^*$ is the interior of the set
  \[
    \{w \in V_S^* : w(x + tv_s) \le 0 \textrm{ for all }x \in
    \ker(\alpha_s) \cap \tilde{\Omega}, t > 0\}.
  \]
  Now let $w$ be any point in $\uhalfcone_+(W_s)^*$. Since $v_s$ does
  not lie in $\ker(\alpha_s)$, we can uniquely write $w = w_1 + w_2$
  for $w_1 \in \R \alpha_s$ and $w_2 \in \ker(v_s)$ (viewing $v_s$ as
  an element of $V^{**}$). Since $v_s$ is a nonzero vector in the
  closure of $\uhalfcone_+(W_s)$, we know $w(v_s) < 0$. Then because
  $w(v_s) = w_1(v_s)$, and $\alpha_s(v_s) = 2$, we must have
  $w_1 = t\alpha_s$ for $t < 0$. And, any
  $x \in \ker(\alpha_s) \cap \tilde{\Omega}$ lies in the closure of
  $\uhalfcone_+(W_s)$ also, so we know $w(x) < 0$, hence $w_2(x) < 0$.

  By \Cref{lem:halfspaces_in_halfcones}, we know that for any
  $v \in \tilde{\Omega}$, we can write $v = x + tv_s$ for
  $x \in \ker(\alpha_s) \cap \tilde{\Omega}$ and $t \in \R$. Then
  $w_2(v) = w_2(x) < 0$, which means that
  $w_2 \in \ker(v_s) \cap \tilde{\Omega}^*$. This tells us that we can
  write $\uhalfcone_+(W_s)^*$ as the interior of
  \[
    \{w + t\alpha_s : w \in \ker(v_s) \cap \tilde{\Omega}^*, t < 0\},
  \]
  which is precisely the definition of $\uhalfcone_-(W_s^*)$.
\end{proof}

\subsection{Nested half-cones}
\label{sec:half_cone_nesting}

For the rest of the section, we will work exclusively with the Vinberg
domain $\vindomain$ for a simplicial representation $\rho$, rather
than an arbitrary reflection domain. Our goal is to understand when
the half-cones over a pair of walls $W_1, W_2$ in $\vindomain$ are
(strongly) nested inside of each other.

\begin{lem}[Half-cones nest]
  \label{lem:halfcones_nest}
  Let $C$ be an irreducible infinite right-angled Coxeter group, and
  let $\rho\colon C \to \SLpm(V)$ be a simplicial representation. Suppose
  that $W_1$ and $W_2$ are walls in $\vindomain$ such that
  $W_1 \cap W_2 = \emptyset$ and
  $\halfspace_+(W_2) \subset \halfspace_+(W_1)$. Then
  $\halfcone_+(W_2) \subset \halfcone_+(W_1)$. Moreover, if
  $\overline{W_1} \cap \overline{W_2} = \emptyset$, then
  $\partial \halfcone_+(W_1) \cap \overline{\halfcone_+(W_2)}$ is
  contained in $\partial \halfcone_+(W_1) \cap \overline{W_2}$.
\end{lem}
\begin{proof}
  We let $\gamma_1, \gamma_2$ be the elements of $C$ such that
  $\rho(\gamma_i)$ fixes $W_i$ for $i = 1,2$. We fix
  $\alpha_1, \alpha_2 \in V^*$ so that $\uhalfspace_+(W_1)$ lies in
  the positive half-space of $V$ defined by $\alpha_1$, and similarly
  for $\alpha_2, \uhalfspace_+(W_2)$. Then we let
  $\tilde{W}_1 = \ker \alpha_1 \cap \uvindomain$ and
  $\tilde{W}_2 = \ker \alpha_2 \cap \uvindomain$, and fix vectors
  $v_1, v_2 \in V$ so that $\uhalfcone_+(W_i)$ is the interior of the
  convex hull of $\tilde{W}_i$ and $v_i$ for $i = 1,2$. Since
  $\halfspace_+(W_2) \subset \halfspace_+(W_1)$, we know that
  $\tilde{W}_2 \subset \uhalfcone_+(W_1)$. We will show that
  $v_2 \in \overline{\uhalfcone_+(W_1)}$, and that if
  $\overline{W_2} \cap \overline{W_1} = \emptyset$, then
  $v_2 \in \uhalfcone_+(W_1)$. Then both parts of the lemma will
  follow from convexity of $\halfcone_+(W_1)$.

  Consider the 2-dimensional subspace spanned $V_{12}$ by $v_1$ and
  $v_2$. If $V_{12} \cap \overline{\uvindomain}$ is trivial, then
  since $\vindomain$ is properly convex there is some
  $w \in \uvindomain^*$ which vanishes on both $v_1$ and $v_2$. But
  then $w$ is fixed by both $\rho^*(\gamma_1)$ and $\rho^*(\gamma_2)$,
  implying that $W_1^* \cap W_2^*$ is nonempty. But this occurs if and
  only if the reflections fixing $W_1^*$ and $W_2^*$ commute, which we
  know is not the case because $W_1 \cap W_2 = \emptyset$. We conclude
  that $V_{12} \cap \overline{\uvindomain}$ is nontrivial, and let $x$
  be a nonzero point in this intersection.

  Next, suppose that $V_{12} \cap \uvindomain$ is trivial. Repeating
  the same argument from before, we see there is some nonzero
  $w \in \overline{\uvindomain^*}$ which vanishes on $v_1$ and $v_2$
  and therefore
  $\overline{W_1^*} \cap \overline{W_2^*} \ne \emptyset$. But by
  \Cref{lem:dual_disjoint_walls} this is impossible if
  $\overline{W_1} \cap \overline{W_2} = \emptyset$. We conclude that
  if $\overline{W_1} \cap \overline{W_2} = \emptyset$, then we can
  take our nonzero point $x \in V_{12} \cap \overline{\uvindomain}$ to
  lie in $V_{12} \cap \uvindomain$.

  Now, we can view $v_1$ as a linear functional on $V^*$, whose kernel
  intersects the interior of the dual domain $\uvindomain^*$ in the
  reflection wall for $\rho^*(\gamma_1)$. Equivalently, $v_1$ does not
  lie in the closure of the dual domain to $\uvindomain^*$ in
  $V^{**}$. But this dual domain is canonically $\uvindomain$, which
  means that $v_1 \notin \overline{\uvindomain}$.
  
  Because of this, \Cref{lem:halfspaces_in_halfcones} implies that
  $x = x_0 + tv_1$ for some $x_0 \in \overline{\tilde{W}_1}$ and
  $t \in \R$, hence $v_2 = x_0 + t'v_1$ for $t' \in \R$. Thus, we know
  that $v_2$ lies in the closure of the set
  $U = \uhalfcone_-(W_1) \cup \uhalfcone_+(W_1) \cup
  \tilde{W}_1$. Moreover, if
  $\overline{W_1} \cap \overline{W_2} = \emptyset$, we may take
  $x_0 \in \tilde{W}_1$, which implies that in fact $v_2 \in U$.

  Now, since $\uhalfspace_+(W_1)$ contains $\uhalfspace_+(W_2)$, we
  must have $\tilde{W_1} \subset \uhalfspace_-(W_2)$ and therefore
  \[
    \rho(\gamma_2)\tilde{W_1} \subset \uhalfspace_+(W_2) \subset
    \uhalfspace_+(W_1).
  \]
  Letting $y \in \tilde{W_1}$, we see that
  $\alpha_1(\rho(\gamma_2)y) > 0$, i.e.\
  $\alpha_1(y - \alpha_2(y)v_2) > 0$ and therefore
  $\alpha_1(v_2) > 0$. Thus, $v_2$ cannot lie in the closure of
  $\uhalfcone_-(W_1)$. This proves that
  $v_2 \in \overline{\uhalfcone_+(W_1)}$, and that
  $v_2 \in \uhalfcone_+(W_1)$ if
  $\overline{W_1} \cap \overline{W_2} = \emptyset$.
\end{proof}

\subsubsection{Strong nesting}
\label{sec:strong_nesting}
Now consider a pair of disjoint walls $W_1, W_2$ in $\vindomain$, such
that $\halfspace_+(W_2) \subset \halfspace_+(W_1)$. We want to know
precisely when the half-cone $\halfcone_+(W_2)$ is \emph{strongly}
nested inside of $\halfcone_+(W_1)$, i.e.\ when the closure of
$\halfcone_+(W_2)$ is contained in $\halfcone_+(W_1)$.

It is clear that this strong nesting cannot occur if
$\overline{W_1} \cap \overline{W_2}$ is nonempty (meaning, by
\Cref{lem:disjoint_walls}, that the group element $\elt(W_1, W_2)$
lies in a proper standard subgroup of $C$). One might hope that this
is the only situation in which strong nesting fails, and indeed, if
this were true, it would greatly simplify several inductive arguments
in the final section of this paper. But it turns out that this is not
the case; see \Cref{sec:appendix} for an explicit counterexample.

However, in the special case where a geodesic representing the group
element $\gamma = \elt(W_1, W_2)$ lies in a hyperbolic subcomplex of
$\daviscx$, the lemma below gives us some control over the situations
where the half-cones over $W_1, W_2$ fail to strongly nest.

\begin{lem}
  \label{lem:halfcone_nesting_failure}
  Let $\mc{W} = W_0, \ldots, W_{n+1}$ be an efficient itinerary in
  $\vindomain$, and suppose that
  $\halfcone_+(W_{n+1}) \subset \halfcone_+(W_0)$ but
  $\overline{\halfcone_+(W_{n+1})} \not\subset
  \halfcone_+(W_0)$. Suppose further that $\elt(\mc{W})$ has
  $D$-bounded product projections. Then, for a constant $R$ depending
  only on $D$, $\mc{W}$ is equivalent to an $\vindomain$-itinerary of
  the form $\mc{U}_-, \mc{U}_0, \mc{U}_+$, such that:
  \begin{enumerate}
  \item The itinerary $\mc{U}_0$ satisfies $|\mc{U}_0| < 2R$;
  \item The itineraries $\mc{U}_-, \mc{U}_+$ are both efficient;
  \item Both intersections $\bigcap_{U \in \mc{U}_+} \overline{U}$ and
    $\bigcap_{U \in \mc{U}_-}\overline{U}$ are nonempty.
  \end{enumerate}
\end{lem}

In light of \Cref{lem:disjoint_walls}, the lemma above tells us that a
pair of disjoint walls in the Vinberg domain will fail to have
strictly nested half-cones only if the group element traversed by an
efficient itinerary between them is (roughly) a product of two
elements which each lie in a proper standard subgroup.

\begin{proof}
  First, if $\overline{W_0} \cap \overline{W_{n+1}}$ is nonempty, then
  by \Cref{lem:itinerary_wall_intersect} we can just take
  $\mc{U}_- = \mc{W}$ and $\mc{U}_0, \mc{U}_+$ to be empty. So, assume
  $\overline{W_0} \cap \overline{W_{n+1}} = \emptyset$.
  
  From \Cref{lem:halfcones_nest}, if $\partial \halfcone_+(W_0)$
  intersects $\overline{\halfcone_+(W_{n+1})}$, then
  $\partial W_{n+1}$ contains a point $x_{n+1}$ in the boundary of
  $\halfcone_+(W_0)$. Then since
  $\overline{W_0} \cap \overline{W_{n+1}} = \emptyset$, there is a
  point $x_0 \in \partial W_0$ and a nontrivial projective segment
  $\ell \subset \partial \halfcone_+(W_0)$ joining $x_0$ to $x_{n+1}$,
  such that the projective span $L$ of $\ell$ contains the tip $v_0$
  of $\halfcone_+(W_0)$. By convexity, $\ell$ is contained in
  $\overline{\vindomain}$. And, since
  $\ell \subset \partial \halfcone_+(W_0)$,
  \Cref{lem:halfspaces_in_halfcones} implies that
  $\ell \subset \partial \vindomain$. Further, the reflection $R_0$
  fixing $W_0$ acts by a reflection on $L$ fixing $x_0$, and preserves
  $\partial \vindomain$. So the union $\ell' = \ell \cup R_0\ell$ is a
  projective segment in $\partial \vindomain$ containing $x_0$ in its
  interior.

  Let $v_{n+1}$ be the polar of $W_{n+1}$, and consider the projective
  subspace $P$ spanned by $L$ and $v_{n+1}$. First, suppose that $P$
  does not intersect the interior of $\vindomain$. Then, there is a
  \emph{supporting hyperplane} $H$ of $\vindomain$ containing $P$,
  i.e.\ a hyperplane in $\P(V)$ which intersects
  $\overline{\vindomain}$ but not $\vindomain$. In particular, this
  supporting hyperplane contains both $v_{n+1}$ and $v_0$, so it is
  preserved by the reflections fixing $W_0$ and $W_{n+1}$. So, $H$
  lies in $\overline{W_0^*} \cap \overline{W_{n+1}^*}$. But,
  \Cref{lem:dual_disjoint_walls} implies that this intersection is
  empty once $\overline{W_0} \cap \overline{W_{n+1}}$ is empty.

  We conclude that $P$ intersects the interior of $\vindomain$. In
  particular $L$ is a proper subspace of $P$, so $P$ is a projective
  $2$-plane. Since each wall $W_i$ for $1 < i < n$ is disjoint from
  both $W_0$ and $W_{n+1}$, the intersection $W_i \cap P$ cannot span
  $P$. So, $W_i \cap P$ is a projective segment $\ell_i$. Each
  $\ell_i$ must separate $W_0 \cap P$ from $W_{n+1} \cap P$ in
  $\vindomain \cap P$, so the closure of $\ell_i$ intersects the
  (closed) segment $\ell$ at a point $x_i$. See
  \Cref{fig:halfcone_nest_failure}.

  Then, by reordering the walls in $\mc{W}$, we can obtain an
  equivalent efficient itinerary where the corresponding ordering of
  the walls $W_i$ is compatible with the (partial) ordering of the
  $x_i$ along $\ell$. We partition this itinerary into two pieces
  $\mc{U}_-', \mc{U}_+'$, where $\mc{U}_-'$ consists of the walls
  $W_i$ such that $x_i \in \ell \minus \{x_{n+1}\}$, and $\mc{U}_+'$
  consists of the walls $W_i$ such that $x_i = x_{n+1}$.

  \begin{figure}[h]
    \centering
\begingroup%
  \makeatletter%
  \providecommand\color[2][]{%
    \errmessage{(Inkscape) Color is used for the text in Inkscape, but the package 'color.sty' is not loaded}%
    \renewcommand\color[2][]{}%
  }%
  \providecommand\transparent[1]{%
    \errmessage{(Inkscape) Transparency is used (non-zero) for the text in Inkscape, but the package 'transparent.sty' is not loaded}%
    \renewcommand\transparent[1]{}%
  }%
  \providecommand\rotatebox[2]{#2}%
  \newcommand*\fsize{\dimexpr\f@size pt\relax}%
  \newcommand*\lineheight[1]{\fontsize{\fsize}{#1\fsize}\selectfont}%
  \ifx\svgwidth\undefined%
    \setlength{\unitlength}{250.92628539bp}%
    \ifx\svgscale\undefined%
      \relax%
    \else%
      \setlength{\unitlength}{\unitlength * \real{\svgscale}}%
    \fi%
  \else%
    \setlength{\unitlength}{\svgwidth}%
  \fi%
  \global\let\svgwidth\undefined%
  \global\let\svgscale\undefined%
  \makeatother%
  \begin{picture}(1,0.52495097)%
    \lineheight{1}%
    \setlength\tabcolsep{0pt}%
    \put(0,0){\includegraphics[width=\unitlength,page=1]{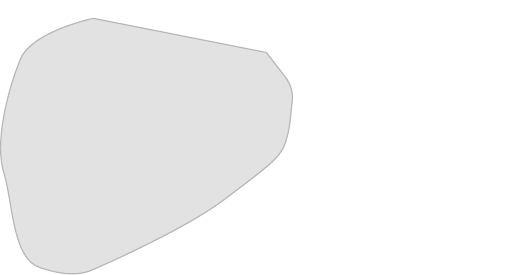}}%
    \put(0.38901352,0.45971627){\color[rgb]{0,0,0}\makebox(0,0)[lt]{\lineheight{1.25}\smash{\begin{tabular}[t]{l}$\ell$\end{tabular}}}}%
    \put(0,0){\includegraphics[width=\unitlength,page=2]{halfcone_nest_failure.pdf}}%
    \put(0.19859008,0.28302354){\color[rgb]{0,0,0}\makebox(0,0)[lt]{\lineheight{1.25}\smash{\begin{tabular}[t]{l}$W_0$\end{tabular}}}}%
    \put(0.90359557,0.37584059){\color[rgb]{0,0,0}\makebox(0,0)[lt]{\lineheight{1.25}\smash{\begin{tabular}[t]{l}$v_0$\end{tabular}}}}%
    \put(0.26183131,0.49222095){\color[rgb]{0,0,0}\makebox(0,0)[lt]{\lineheight{1.25}\smash{\begin{tabular}[t]{l}$x_0$\end{tabular}}}}%
    \put(0.49523612,0.44364163){\color[rgb]{0,0,0}\makebox(0,0)[lt]{\lineheight{1.25}\smash{\begin{tabular}[t]{l}$x_{n+1}$\end{tabular}}}}%
    \put(0,0){\includegraphics[width=\unitlength,page=3]{halfcone_nest_failure.pdf}}%
  \end{picture}%
\endgroup%

    \caption{Illustration for the proof of
      \Cref{lem:halfcone_nesting_failure}. The walls separating $W_0$
      from $W_{n+1}$ all intersect either the interior of
      $\ell \cup R_0\ell$, or the endpoint $x_{n+1}$.}
    \label{fig:halfcone_nest_failure}
  \end{figure}

  \Cref{lem:minimal_wall_intersections} then says that there is a
  uniform constant $R$ so that $\mc{U}_+'$ is equivalent to an
  itinerary $\mc{U}_0'', \mc{U}_+$, where
  $|\mc{U}_0''| \le R$ and $\mc{U}_+$ contains a unique
  minimal wall. Since $\mc{U}_+$ also contains the (unique) maximal
  wall $W_{n+1}$ in $\mc{W}$, it must be efficient. Similarly, we can
  find an itinerary $\mc{U}_-, \mc{U}_0'$ equivalent to $\mc{U}'_-$,
  so that $\mc{U}_-$ is efficient, and $|\mc{U}_0'| \le R$. We let
  $\mc{U}_0 = \mc{U}_0', \mc{U}_0''$, so that $\mc{W}$ is equivalent
  to $\mc{U}_-, \mc{U}_0, \mc{U}_+$.

  Since $x_{n+1} \in \overline{W_i}$ for each $W_i$ in $\mc{U}_+$, 
  we already know that the intersection
  $\bigcap_{U \in \mc{U}_+} \overline{U}$ is nonempty. To see that the
  intersection $\bigcap_{U \in \mc{U}_-}\overline{U}$ is also
  nonempty, let $W_m$ be the (unique) maximal wall in $\mc{U}_-$, so
  that $\overline{W_m}$ intersects the interior of the projective line
  segment $\ell' \subset \partial \vindomain$ at the point $x_m$. Any
  supporting hyperplane of $\vindomain$ at $x_m$ must contain $\ell'$,
  which means that any such hyperplane contains the projective line
  $L$, and in particular contains the point $v_0$.

  Let $v_m$ be the polar of the wall $W_m$. By
  \Cref{lem:halfspaces_in_halfcones}, the projective line spanned by
  $x_m$ and $v_m$ does not intersect $\vindomain$, so it is contained
  in some supporting hyperplane $H$ of $\vindomain$ at $x_m$. Then $H$
  is an element of $\overline{\vindomain^*}$ which contains both $v_m$
  and $v_0$, so it is preserved by the reflections fixing $W_0$ and
  $W_m$. That is, $H \in \overline{W_0^*} \cap \overline{W_m^*}$, so
  \Cref{lem:dual_disjoint_walls} implies that the intersection
  $\overline{W_0} \cap \overline{W_m}$ is also nonempty. Then we are
  done after applying \Cref{lem:itinerary_wall_intersect}.
\end{proof}

\section{Singular values, stable and unstable subspaces, and
  regularity}
\label{sec:singular_values}

The definition of an Anosov representation used in this paper
(\Cref{defn:anosov}) is rooted in linear algebra---specifically, in
the \emph{singular value decomposition} of matrices in $\GL(d,
\R)$. The main purpose of this section is to provide various methods
for estimating singular values and singular value gaps, which we will
use throughout the rest of the paper.

Much of the content of this section can also be interpreted in terms
of the geometry and dynamics of semisimple Lie groups or their
associated Riemannian symmetric spaces, and some of the cited results
rely on this perspective (or other complicated tools) for their
proofs. However, we will state all of the estimates we need in
elementary terms, and refer to \cite[Sect. 7]{BPS} for an overview of
the connection between the different viewpoints.

\subsection{Singular values} \label{subsec:singular values}

We equip $\R^d$ with its standard Euclidean inner product. Given a
matrix $g \in \GL(d,\R)$, recall that the $i$\textsuperscript{th}
singular value $\sigma_i(g)$ is the $i$\textsuperscript{th} largest
eigenvalue of the linear map $\sqrt{gg^T}$, where $\sqrt{gg^T}$ is the
unique positive-definite matrix squaring to the positive-definite
matrix $gg^T$.

The definition of $\sigma_i(g)$ depends on the choice of
positive-definite inner product on $\R^d$. We obtain different
singular values for $g$ by choosing a different inner product
$\langle \cdot, \cdot \rangle$, and using the adjoint of $g$ with
respect to $\langle \cdot, \cdot \rangle$ in place of the transpose
matrix $g^T$. Geometrically, the $i$\textsuperscript{th} singular
value is the length of the $i$\textsuperscript{th} longest axis of the
image ellipsoid
\[
  \{gv : v \in \R^d, \|v\| = 1\},
\]
measured with respect to the chosen inner product on $\R^d$.

If $\sigma_i(g) > \sigma_{i+1}(g)$, then we define the unstable
subspace $\unstable{i}(g)$ as the subspace of $\R^d$ spanned by the
$i$ longest axes of this image ellipsoid. If
$\sigma_{d-i}(g) > \sigma_{d-i+1}(g)$, we also have the stable space
$\stable{i}(g) = \unstable{i}(g^{-1})$, which is the subspace spanned
by the $i$ least expanded axes of the unit sphere
$\{v \in \R^d : \|v\| = 1\}$.

All of this information is combined in the singular value
decomposition (or $KAK$ decomposition) of $g$: any element
$g \in \GL(d, \R)$ can be written as a product $g = ka\ell$, where $a$
is a diagonal matrix with nonzero entries
$\sigma_1, \ldots, \sigma_d$, and $k, \ell \in \mathrm{O}(d)$ are such
that $k$ takes the span of the first $i$ standard basis vectors to
$E_i^+(g)$, and $\ell$ takes $E_i^-(g)$ to the span of the last $i$
standard basis vectors.
  
We let $\mu_i(g)$ denote $\log\sigma_i(g)$, and collect the logarithms
together in the vector
$\mu(g) := (\mu_1(g), \mu_2(g), \dots, \mu_d(g))$. We will also write
$\mu_{i,i+1}(g)$ as shorthand for $\mu_i(g) - \mu_{i+1}(g)$.

\subsection{Additivity estimates for singular value gaps}
As noted in the introduction, a representation
$\rho\colon \Gamma \to \GL(d,\R)$ of a finitely-generated group
$\Gamma$ is $1$-Anosov if and only if there exist constants $A, B > 0$
such that $\mu_{1,2}(\rho(\gamma)) \geq A|\gamma| - B$ for all
$\gamma \in \Gamma$.

This definition depends on a choice of word metric determined by a
choice of finite generating set for the group $\Gamma$, but this does
not matter since all such word metrics are quasi-isometrically
equivalent. It also depends on the choice of inner product used to
define singular values, but this also turns out not to matter. One
justification is the following error estimate on logarithms of
singular values:

\begin{lem}[Additivity estimate for $\mu$; see {\cite[Fact
    2.18]{ggkw2017anosov}}]
  \label{lem:additive_root_bound_general}
  For any norm $||\cdot||$ on $\R^d$, there is a constant $K > 0$ such
  that for any $g, h_1, h_2 \in \SLpm(d, \R)$, we have
  \[
    ||\mu(h_1gh_2) - \mu(g)|| \le K(||\mu(h_1)|| + ||\mu(h_2)||).
  \]
  In particular, for any $1 \le i < d$, there is some $K' > 0$ such
  that
  \[
    |\mu_{i,i+1}(h_1gh_2) - \mu_{i,i+1}(g)| \le K'(||\mu(h_1)|| +
    ||\mu(h_2)||).
  \]
\end{lem}
As any pair of positive-definite inner products on $\R^d$ differ by
composition with some fixed element $h \in \GL(d, \R)$, the lemma
implies in particular that choosing a different inner product on
$\R^d$ to define singular values only changes each $\mu_k(g)$ by a
uniformly bounded additive amount.

The lemma also tells us that if $\Gamma$ is any finitely generated
subgroup of $\SLpm(d, \R)$, equipped with a word metric $|\cdot|$,
then $||\mu(\gamma)||$ is at most $K|\gamma|$ for a uniform constant
$K$. Thus, we can also apply the lemma to obtain the following useful
result:
\begin{lem}[``Triangle inequality'' for $\mu_{i,i+1}$]
  \label{lem:root_triangle_inequality}
  Let $\Gamma$ be a finitely generated subgroup of $\SLpm(d,\R)$,
  equipped with a word metric $|\cdot|$. There exists a constant
  $K > 0$ such that for any $\gamma, \eta_1, \eta_2 \in \Gamma$, we
  have
  \[
    |\mu_{i,i+1}(\eta_1\gamma\eta_2) - \mu_{i,i+1}(\gamma)| \le
    K(|\eta_1| + |\eta_2|).
  \]
\end{lem}

We will frequently apply this lemma throughout the rest of the paper,
as it allows us to estimate singular value gaps for the images of
``nearby'' elements in a finitely generated group $\Gamma$ under some
representation $\rho$.

We will also sometimes want to estimate singular value gaps for the
product of a pair of \emph{large} elements $\gamma_1, \gamma_2$ in
$\Gamma$. This is possible as long as the stable and unstable
subspaces of $\rho(\gamma_1), \rho(\gamma_2)$ exist and are
\emph{uniformly transverse}. In this situation, we can apply the
following:
\begin{lem}[{Uniform transversality estimate for $\mu_{1,2}$; see
    \cite[Lemma A.7]{BPS}}]
  \label{lem:root_additivity}
  Let $g$, $h$ be two elements of $\SLpm(d, \R)$, and suppose that
  $\toproot(g) > 0$ and $\toproot(h) > 0$. If
  $\theta = \angle(\stable{d-1}(g), \unstable{1}(h))$, then we have
  \[
    \toproot(gh) \ge \toproot(g) + \toproot(h) + 2\log(\sin\theta).
  \]
\end{lem}

\subsection{Regularity, uniform gaps, and the local-to-global
  principle}

The next estimates in this section are less elementary in nature,
but can still be stated in terms of linear algebra.
\begin{definition}[Uniform regularity]
  \label{defn:uniform_regular}
  Let $g_n$ be a (finite or infinite) sequence in $\GL(d, \R)$. We say
  that the sequence $g_n$ is \emph{$(A, B)$-regular} for
  constants $A \ge 0, B \ge 0$ if for all $n, m$, we have
  \begin{equation}
    \label{eq:uniform_regularity}
    \mu_{1,2}(g_n^{-1}g_m) \ge A\mu_{1,d}(g_n^{-1}g_m) - B.
  \end{equation}
  If an infinite sequence $g_n$ is $(A, B)$-regular for some $A, B$,
  then we say that $g_n$ is \emph{uniformly regular}.
\end{definition}
Sequences as in \Cref{defn:uniform_regular} are a special case of the
``coarsely uniformly regular'' sequences defined by
Kapovich-Leeb-Porti in \cite{KLP2018}; in general a different type of
``regularity'' can be defined for each singular value gap
$\mu_{i,i+1}$. For sequences lying in a finitely generated subgroup
$\Gamma < \GL(d, \R)$, we can also strengthen this notion of uniform
regularity as follows:
\begin{definition} 
  \label{defn:URU_geodesic}
  Let $\Gamma$ be a finitely generated subgroup of $\GL(d,\R)$, and
  let $\gamma_n$ be a sequence in $\Gamma$ which is a geodesic with
  respect to a word metric $|\cdot|$. We say the geodesic $\gamma_n$
  \emph{has $(A,B)$-gaps} if
  \[
    \mu_{1,2}(\gamma_n^{-1}\gamma_m) \ge A|m - n| - B.
  \]
  If a geodesic $\gamma_n$ has $(A,B)$-gaps for some $A,B$, we also
  say it has \emph{uniform gaps}.
\end{definition}
A geodesic $\gamma_n$ with uniform gaps is always uniformly regular:
for any $g \in \GL(d, \R)$, we have
$\mu_{1,d}(g) = \log||g|| + \log||g^{-1}||$, where $||\cdot||$ is the
operator norm on $\GL(d, \R)$. This means that there is a constant $K$
(depending only on $\Gamma$ and on the choice of finite generating set
determining $|\cdot|$) so that
\[
  \mu_{1,d}(\gamma_n^{-1} \gamma_m) \leq K|m-n|.
\]

Geodesics with uniform gaps as in \Cref{defn:URU_geodesic} are a
special case of the \emph{uniformly regular undistorted} (or
\emph{URU}) sequences defined in the work of Kapovich--Leeb--Porti
\cite{klp2017characterizations}; ``undistorted'' refers to the fact
that these sequences map to quasi-geodesics in the Riemannian
symmetric space associated to the semisimple Lie group $\PGL(d,
\R)$.
We note that, in this language, 1-Anosov representations are precisely those which send sequences in the group which are geodesic with respect to a word metric to uniformly regular and undistorted sequences. 

A key feature of geodesics with uniform gaps is that they have
well-defined ``limit points'' in both projective space $\P(\R^d)$ and
its dual (indeed, this is one way of defining limit maps for 1-Anosov representations). To state this result, we let $\dproj$ denote the metric on
projective space obtained by viewing $\P(\R^d)$ as the quotient of the
unit sphere in $\R^d$. Then we let $\dproj^*$ denote the metric on the
Grassmannian $\mathrm{Gr}_{d-1}(\R^d)$ obtained by viewing the
projectivization of each $(d-1)$-plane in $\mathrm{Gr}_{d-1}(\R^d)$ as
a subset of $\P(\R^d)$, and taking Hausdorff distance with respect to
$\dproj$.
\begin{lem}[{Uniform convergence for geodesics with uniform gaps; see
    \cite[Proposition 2.4]{BPS}}]
  \label{lem:uniform_convergence_regular}
  Let $\Gamma$ be a finitely generated subgroup of $\GL(d, \R)$. If
  $\gamma_n$ is a geodesic in $\Gamma$ with $(A,B)$-gaps, then the
  limits
  \[
    \lim_{n \to \infty} \unstable{1}(\gamma_n), \qquad \lim_{n \to
      \infty} \stable{d-1}(\gamma_n)
  \]
  exist in $\P(\R^d)$ and $\mathrm{Gr}_{d-1}(\R^d)$, respectively, and
  are uniform in $A, B$ with respect to the metrics $\dproj$ and
  $\dproj^*$.
\end{lem}

\begin{remark}
  \Cref{lem:uniform_convergence_regular} also follows from the
  \emph{higher-rank Morse lemma} of Kapovich--Leeb--Porti
  \cite{KLP2018} (discussed briefly below), but we refer to the result
  in \cite{BPS} since it is closer to the form that we actually
  need. Both of these results are stronger and more technical than
  what we have stated here---in particular, they imply a version of
  \Cref{lem:uniform_convergence_regular} which does not require the
  sequence $\gamma_n$ to lie in a finitely generated subgroup of
  $\GL(d, \R)$.

  We refer also to \cite[Sect. 5]{ggkw2017anosov} for a related result
  that holds under the assumption that $\Gamma$ is a word-hyperbolic
  group.
\end{remark}

The next estimate in this section follows from several results of
Kapovich--Leeb--Porti. The first, proved in \cite{KLP2018}, is a
higher-rank version of the Morse lemma. It implies that uniformly
regular undistorted sequences in a finitely generated subgroup
$\Gamma < \SLpm(d, \R)$ are essentially equivalent to uniformly
\emph{Morse} quasi-geodesic sequences in the associated Riemannian
symmetric space $X$, i.e. quasi-geodesic sequences which stay
uniformly close to certain ``Morse subsets'' in $X$. The next result,
proved in \cite{KLP2014} (or \cite{KLP2023}), says that these Morse
quasi-geodesics satisfy a \emph{local-to-global} principle; a more
precise version of this theorem was later proved by Riestenberg
\cite{Riestenberg}.

Combining these results gives us a local-to-global principle for
geodesics with uniform gaps in a finitely generated group $\Gamma$,
which we can state as follows:
\begin{thm}[{Local-to-global principle; see \cite[Theorem
    1.3]{KLP2018} and \cite[Theorem 1.1]{KLP2023} or \cite[Theorem
    7.18]{KLP2014} or \cite[Theorem 1.1]{Riestenberg}}]
  \label{thm:local_to_global}
  Let $\Gamma$ be a finitely generated subgroup of $\SLpm(d,
  \R)$. Given $A, B > 0$, there exist constants
  $A', B', \lambda > 0$ satisfying the following. Suppose that
  $\gamma_n$ is a sequence in $\Gamma$ which is a geodesic with
  respect to a word metric $|\cdot|$, and that for every $i < j$ with
  $|i - j| \le \lambda$, the sub-geodesic $\gamma_i, \ldots, \gamma_j$
  has $(A, B)$-gaps. Then $\gamma_n$ has $(A',B')$-gaps.
\end{thm}

\subsection{Estimating singular value gaps using the Hilbert metric}

Our last estimate on singular values comes from convex projective
geometry. Any properly convex domain $\Omega$ in real projective space
$\P(\R^d)$ can be endowed with a natural metric $d_{\Omega}$ called
the \emph{Hilbert metric}. The Hilbert metric on $\Omega$ is always
proper, geodesic, and invariant under any projective transformations
preserving $\Omega$.

We will use the Hilbert metric as a tool to study the behavior of
\emph{nested sequences} of properly convex domains in
$\P(\R^d)$. Roughly, we will use the Hilbert metric to measure how
much a projective transformation $g$ \emph{fails} to preserve a given
properly convex domain $\Omega$, and then use this to estimate the
singular value gaps of $g$.

We provide this estimate in
\Cref{lem:hilbert_metric_estimates_sv_gaps} below, but first we recall
some basic definitions.

\begin{definition}
  Let $\Omega$ be a properly convex domain. For $x, y \in \Omega$, we
  define the Hilbert distance $d_\Omega(x, y)$ by
  \begin{equation*}
    d_{\Omega}(x,y) = \frac12\log([a, b; x, y]),
  \end{equation*}
  where $a, b$ are the unique points in $\partial \Omega$ such that
  $a, x, y, b$ lie on a projective line in that order, and
  $[\cdot, \cdot; \cdot, \cdot]$ is the projective cross-ratio with
  formula
  \begin{equation*}
    [a, b ; c, d] = \frac{\norm{a-c}}{\norm{a-d}}\cdot\frac{\norm{b-d}}{\norm{b-c}}.
  \end{equation*}
\end{definition}

The distances appearing in the cross-ratio formula can be measured
using any Euclidean metric on any affine chart containing
$a, b, c, d$; the value of the cross-ratio does not depend on the
choice of chart or metric. From this, it follows immediately that for
any properly convex domain $\Omega \subset \P(V)$, any $g \in \GL(V)$,
and any $x, y \in \Omega$, we have
\begin{equation}
  d_{\Omega}(x, y) = d_{g\Omega}(gx, gy).
\end{equation}

For our purposes, the most important property of the Hilbert metric is
the following standard result, which can be verified by a
straightforward computation:
\begin{lem}[Expansion of the Hilbert metric on nested domains]
  \label{lem:hilbert_metric_contracts}
  Let $\Omega_1, \Omega_2$ be properly convex domains in $\P(\R^d)$,
  and suppose $\Omega_1 \subseteq \Omega_2$. Then for all $x, y \in
  \Omega_1$, we have
  \[
    d_{\Omega_1}(x,y) \ge d_{\Omega_2}(x,y).
  \]
  Further, if the strong inclusion
  $\overline{\Omega_1} \subset \Omega_2$ also holds, then there is a
  constant $\lambda > 1$ (depending only on $\Omega_1$ and $\Omega_2$)
  so that for all $x, y \in \Omega_1$, we have
  \[
    d_{\Omega_1}(x,y) \ge \lambda \cdot d_{\Omega_2}(x,y).
  \]
\end{lem}

Our main application of the Hilbert metric in this paper is the
following. Here, and elsewhere, if $\Omega$ is a properly convex
domain and $X \subseteq \Omega$, then $\mathrm{diam}_{\Omega}(X)$
refers to the diameter of $X$ with respect to the Hilbert metric
$d_\Omega$.
\begin{lem}[Hilbert metric estimates singular value gaps]
  \label{lem:hilbert_metric_estimates_sv_gaps}
  Fix properly convex domains $\Omega_1, \Omega_2 \subset
  \P(\R^d)$. There exists a constant $D > 0$, depending only on
  $\Omega_1, \Omega_2$, such that for any $g \in \GL(d, \R)$ with
  $g\overline{\Omega_1} \subset \Omega_2$, we have
  \begin{equation}\label{eq:hilbert_metric_estimates_sv_gaps}
    \mu_{1,2}(g) \ge -\log(    \mathrm{diam}_{\Omega_2}(g\Omega_1)) - D.
  \end{equation}
\end{lem}
\begin{proof}
  Let $A_d$ be the standard affine chart
  $\{[a_1 : \ldots : a_{d-1} : 1] \subset \P(\R^d) : a_i \in \R\}$,
  equipped with the Euclidean metric induced by the standard inner
  product on $\R^d$. We let $B$ be the unit ball in $A_d$ centered at
  $[e_d] = [0 : \ldots : 0 : 1]$, and first consider the special case
  of the lemma where $\Omega_1 = \Omega_2 = B$, and $g$ fixes both
  $e_d$ and its orthogonal complement
  $e_d^\perp = \spn\{e_1, \ldots, e_{d-1}\}$. That is, we assume that
  $g$ is block-diagonal, of the form
  \[
    \begin{pmatrix}
      H \\
      & \lambda
    \end{pmatrix}
  \]
  for $H \in \GL(e_d^\perp)$ and $\lambda \in \R$. Then $gB$ is an
  ellipsoid in $A_d$, centered at $[e_d]$, whose semi-major axis has
  length $\sigma_1(H) / |\lambda|$. Then if $g\overline{B} \subset B$,
  we have $\sigma_1(H) / |\lambda| < 1$ and thus
  $\sigma_1(g)/\sigma_2(g) = |\lambda|/\sigma_1(H)$.

  The diameter of $gB$ with respect to the Hilbert metric $d_B$ is
  therefore given by
  \begin{align*}
    \diam_B(gB) &= \log\left(\frac{1+e^{-\mu_{1,2}(g)}}{1 -
                  e^{-\mu_{1,2}(g)}}\right)\\ &\geq 2\log(1 + e^{-\mu_{1,2}(g)}) \\
                &\geq 2e^{-\mu_{1,2}(g)}.
  \end{align*}
  This proves that \eqref{eq:hilbert_metric_estimates_sv_gaps} holds
  in this special case, when we take $D = \log 2$. The rest of the
  proof amounts to reducing the general case to this one.

  First, we reduce the general situation to the case where
  $\overline{\Omega_2} \subset \Omega_1$, by choosing some
  $h_1 \in \SL(d, \R)$ so that
  $\overline{\Omega_2} \subset h_1\Omega_1$. Then, if
  $\Omega_1' = h_1\Omega_1$, we have
  \[
    \diam_{\Omega_2}(g\Omega_1) = \diam_{\Omega_2}(gh_1^{-1}\Omega_1'),
  \]
  and by \Cref{lem:root_triangle_inequality} the difference
  $|\toproot(g) - \toproot(gh_1^{-1})|$ is bounded by a constant only
  depending on $h_1$. So, we can replace $\Omega_1$ with $\Omega_1'$
  and $g$ with $gh_1^{-1}$. This introduces uniformly bounded additive error
  to the left-hand side of
  \eqref{eq:hilbert_metric_estimates_sv_gaps}, and does not affect the
  right-hand side.

  Now, since $g\overline{\Omega_1} \subset \Omega_2 \subset \Omega_1$,
  the Brouwer fixed-point theorem implies that $g$ fixes some
  $x \in \Omega_2$. Since duality reverses inclusions, we also know
  that
  $g^{-1}\overline{\Omega_2^*} \subset \Omega_1^* \subset \Omega_2^*$,
  and therefore $g$ also fixes an element of $\Omega_1^*$, which we
  can view as a hyperplane $U \subset \R^d$. Since
  $\overline{\Omega_2} \subset \Omega_1$, there is a uniform lower
  bound $\eps > 0$ on the distance between $x$ and $U$, in a fixed
  metric on $\P(\R^d)$. So, there is a compact set $K = K(\eps)$ in
  $\SL(d, \R)$ such that for some $h \in K$, the spaces $hx, hU$ are
  orthogonal with respect to the standard inner product on $\R^d$.

  By projective invariance of the Hilbert metric we know that
  \[
    \diam_{\Omega_2}(g\Omega_1) = \diam_{h\Omega_2}(hg\Omega_1).
  \]
  So, if we replace $\Omega_1, \Omega_2$ with $h\Omega_1, h\Omega_2$,
  and replace $g$ with $hgh^{-1}$, the right-hand side of
  \eqref{eq:hilbert_metric_estimates_sv_gaps} again stays the same,
  and (again applying \Cref{lem:root_triangle_inequality}) the
  left-hand side only changes by an additive amount, bounded uniformly
  in terms of $\Omega_1, \Omega_2$. Then, after further conjugating by
  an orthogonal matrix, we can assume the fixed point of $g$ in
  $\Omega_2$ is $[e_d]$, and that $\overline{\Omega_1}$ lies in the
  $g$-invariant affine chart $A_d$.

  For any $r > 0$, we let $B_r$ denote the open ball of radius $r$
  centered at $[e_d]$ in the affine chart $A_d$. Since
  $\Omega_1, \Omega_2$ are both bounded open convex subsets in $A_d$,
  we can find radii $r_1, r_2 > 0$ such that
  $B_{r_1} \subset \Omega_1$, and
  $\overline{\Omega_2} \subset B_{r_2}$. We let $h_1, h_2$ denote the
  dilations in $A_d$ by a factor of $r_1, r_2$, respectively, and we
  let $g' = h_2^{-1}gh_1$.

  Using \Cref{lem:hilbert_metric_contracts}, we see that
  \[
    \diam_{B_1}(g'B_1) = \diam_{B_{r_2}}(gB_{r_1}) <
    \diam_{\Omega_2}(g\Omega_1),
  \]
  and one more application of \Cref{lem:root_triangle_inequality}
  tells us that the difference between $\toproot(g)$ and $\toproot(g')$ is bounded by a constant depending only on $\Omega_1, \Omega_2$. This reduces the problem to the case that we
  originally proved.
\end{proof}

\section{Stable and unstable subspaces in half-cones}

Let $(C, S)$ be an infinite irreducible right-angled Coxeter system
with $|S| = d$, and let $\rho\colon C \to \SLpm(d, \R)$ be a simplicial
representation. As in the previous section, we assume $V = \R^d$ is
endowed with its standard inner product, which allows us to define
$\sigma_k(g)$ and $\mu_k(g)$ for every $g \in \SLpm(V)$; if
$\mu_{k, k+1}(g) \ne 0$ we can also define $\unstable{k}(g)$ and
$\stable{d-k}(g)$.

In this section, we establish some estimates on the location of the
subspaces $\unstable{1}(\rho(\gamma))$ and
$\stable{d-1}(\rho(\gamma))$ for certain elements $\gamma \in C$, in
terms of the walls of an $\vindomain$-itinerary $\mc{W}$ traversing
$\gamma$. Ultimately, we will use these estimates to argue that the
stable and unstable subspaces of certain group elements
$\rho(\gamma_1), \rho(\gamma_2)$ are \emph{uniformly transverse},
which allows us to apply \Cref{lem:root_additivity} to get an estimate
on $\mu_{1,2}(\rho(\gamma_1\gamma_2))$.

The idea behind the first estimate is as follows. An unstable subspace
$\unstable{1}(g)$ should be thought of as an ``attracting'' subspace
for $g$ in $\P(V)$. If $\gamma$ is a group element separated from the
identity by walls $W_1, \ldots, W_n$, and $x_0$ is a fixed basepoint
in a reflection domain $\Omega$, then $\rho(\gamma) \cdot x_0$ is
separated from $x_0$ by this same set of walls, and thus an
``attracting'' subspace for $\rho(\gamma)$ should at least lie in the
positive half-space and positive half-cone bounded by the wall
$W_1$. In fact, this ``attracting'' subspace should lie in the
positive half-cone bounded by the wall $W_k$, as long as $k$ is much
smaller than $n$.

To make this precise, we first need a general fact about group actions
on convex projective domains:
\begin{lem}
  \label{lem:p1_divergent_accumulation}
  Let $\Omega$ be a properly convex domain, and let
  $g_n \in \SLpm(d, \R)$ be any sequence of elements such that
  $g_n\Omega = \Omega$, and $\mu_{1,2}(g_n) \to \infty$. Then for any
  point $x \in \Omega$, the set of accumulation points of the sequence
  $g_nx$ is the same as the set of accumulation points of the sequence
  $\unstable{1}(g_n)$.
\end{lem}
\begin{proof}
  Fix $x \in \Omega$, and choose an arbitrary subsequence of $g_n$ so
  that both of the sequences $g_nx$ and $\unstable{1}(g_n)$
  converge. We will show that these sequences converge to the same
  point.

  Let $\Aut(\Omega) < \SLpm(d, \R)$ denote the subgroup
  $\{g \in \SLpm(d, \R) : g \Omega = \Omega\}$. The Hilbert metric on
  $\Omega$ (see the previous section) is always proper and
  $\Aut(\Omega)$-invariant, which implies that the action of
  $\Aut(\Omega)$ on $\Omega$ is proper. Using this, one can check (see
  e.g.\ \cite[Proposition 5.6]{IZ21}) that any accumulation point of
  the sequence of stable subspaces $\stable{d-1}(g_n)$ is a supporting
  hyperplane of $\Omega$ (recall from the proof of
  \Cref{lem:halfcone_nesting_failure} that a supporting hyperplane is
  a projective hyperplane which intersects $\overline{\Omega}$, but
  not $\Omega$ itself). Up to subsequence, $\stable{d-1}(g_n)$
  converges to a fixed supporting hyperplane $H$. We can then appeal
  to the singular value decomposition (or $KAK$ decomposition, see
  \Cref{subsec:singular values}) of $g_n$ to see that if $z$ is any
  point in $\P(V) \minus H$, the sequence $g_nz$ converges to the
  limit of $\unstable{1}(g_n)$. In particular this holds for $z = x$
  and we are done.
\end{proof}

Before stating and proving our first estimate, we establish some more
notation. As in the previous section, we let $\dproj$ denote a
Riemannian metric on $\P(V)$ induced by the standard inner product on
$V$.
\begin{definition}
  For $x \in \P(V)$ and $\eps > 0$, we let $B_\eps(x)$ denote the open
  $\eps$-ball about $x$, with respect to the metric
  $\dproj$. Similarly, if $Z \subset \P(V)$ is a subset, $N_\eps(Z)$
  denotes the $\eps$-neighborhood of $Z$ in $\P(V)$.

  For $\eps > 0$, the ``$-\eps$-neighborhood'' of a set
  $Z \subset \P(V)$ is by convention the complement of the
  $\eps$-neighborhood of the complement of $Z$, i.e.\ the set
  $N_{-\eps}(Z) := \{x \in Z : B_\eps(x) \subset Z\}.$
\end{definition}

We now prove the first main estimate we need from this section:
\begin{lem}[Unstable subspaces lie close to half-cones]
  \label{lem:stable_subspace_in_halfcones}
  For any $\eps > 0$ and any $N > 0$, there exists a constant $M > 0$
  satisfying the following. Suppose that $\mc{W}$ is an
  $\Omega$-itinerary $W_1, \ldots, W_n$ traversing a group element
  $\gamma \in C$, and suppose that $\mu_{1,2}(\rho(\gamma)) > M$. Then
  for any $k \le \min(n, N)$:
  \begin{enumerate}
  \item If $\mc{W}$ departs from the identity, then the unstable
    subspace $\unstable{1}(\rho(\gamma))$ is contained in
    $N_\eps(\halfcone_+(W_k))$.
  \item If $\mc{W}$ arrives at the identity, then the stable subspace
    $\stable{d-1}(\rho(\gamma))$ (viewed as a subset of $\P(V)$) is
    disjoint from $N_{-\eps}(\halfcone_-(W_{n-k}))$.
  \end{enumerate}
\end{lem}
\begin{proof}
  Let $\eps, N$ be given, and fix $k \le N$. We will prove the first
  statement by contradiction: suppose that there is a sequence of
  group elements $\gamma_m \in C$ so that as $m \to \infty$,
  $\mu_{1,2}(\rho(\gamma_m))$ tends to infinity, but for some
  $\Omega$-itinerary $W_{0, m}, \ldots W_{|\gamma_m|, m}$ joining the
  identity to $\gamma_m$, the point $\unstable{1}(\rho(\gamma_m))$
  does \emph{not} lie in the $\eps$-neighborhood of
  $\halfcone_+(W_{k,m})$. Since $W_{k,m}$ is the last wall in an
  itinerary of length $k$ departing from the identity, after
  extracting a subsequence we can assume that $W_{k,m} = W_k$ for a
  wall $W_k$ independent of $m$.

  Fix a basepoint $x$ in the fundamental domain $\Delta \cap
  \Omega$. Then by definition, $W_k$ separates $x$ from
  $\rho(\gamma_m)x$, so $\rho(\gamma_m)x$ lies in the halfspace
  $\halfspace_+(W_k)$, and therefore $\rho(\gamma_m)x$ accumulates in
  $\overline{\halfspace_+(W_k)}$. Then by
  \Cref{lem:p1_divergent_accumulation}, $\unstable{1}(\rho(\gamma_m))$
  also accumulates in $\overline{\halfspace_+(W_k)}$. But we know that
  $\overline{\halfspace_+(W_k)} \subset \overline{\halfcone_+(W_k)}$
  by \Cref{lem:halfspaces_in_halfcones}, so eventually
  $\unstable{1}(\rho(\gamma_m))$ lies in the $\eps$-neighborhood of
  $\halfcone_+(W_k)$, contradiction.

  The second statement follows from the first via duality. Here are
  the details. Given $\gamma \in C$, we consider the group element
  $\rho^*(\gamma^{-1}) \in \SLpm(V^*)$. If $W_1, \ldots, W_n$ is an
  $\Omega$-itinerary traversing $\gamma$ and arriving at the identity,
  then $W_n, \ldots, W_1$ is an $\Omega$-itinerary traversing
  $\gamma^{-1}$ and departing from the identity. The sequence of dual
  walls $W_n^*, \ldots, W_1^*$ is an $\Omega^*$-itinerary, also
  traversing $\gamma^{-1}$ and departing from the identity.

  Observe that the singular values of $\rho^*(\gamma^{-1})$ are
  precisely the same as the singular values of $\rho(\gamma)$, so in
  particular
  $\mu_{1,2}(\rho^*(\gamma^{-1})) = \mu_{1,2}(\rho(\gamma))$. Then,
  the first part of the lemma implies that if
  $\mu_{1,2}(\rho(\gamma))$ is sufficiently large, the unstable
  subspace $\unstable{1}(\rho^*(\gamma^{-1})) \in \P(V^*)$ is
  contained in an arbitrarily small neighborhood of
  $\halfcone_+(W_{n-k}^*) = \halfcone_-(W_{n-k})^*$, with respect to
  the metric $\ddproj$ on $\P(V^*)$; recall that this is the metric on
  $\P(V^*)$ obtained by viewing $\P(V^*)$ as the space of projective
  hyperplanes in $\P(V)$, and then taking Hausdorff distance with
  respect to $\dproj$.

  Therefore, if $K \subset \P(V)$ is the kernel of
  $\unstable{1}(\rho^*(\gamma^{-1}))$, then as long as
  $\mu_{1,2}(\rho(\gamma))$ is sufficiently large, there is a
  projective hyperplane $H$, within Hausdorff distance $\eps$ of $K$,
  so that every point in the closure of $\halfcone_-(W_{n-k})$ is not
  contained in $H$.

  The inner product on $V$ induces an isomorphism $V^* \to V$, which
  identifies $\rho^*(\gamma^{-1}) \in \SLpm(V^*)$ with the transpose
  $\rho(\gamma)^T \in \SLpm(V)$. The kernel $K$ of
  $\unstable{1}(\rho^*(\gamma^{-1}))$ is the orthogonal complement of
  the subspace $\unstable{1}(\rho(\gamma)^T)$, which is in turn the
  stable subspace $\stable{d-1}(\rho(\gamma))$. We conclude that if
  $x \in \halfcone_-(W_{n-k})$ is contained in
  $\stable{d-1}(\rho(\gamma))$, then $\dproj(x, H) < \eps$, hence $x$
  is distance at most $\eps$ from a point in the complement of
  $\halfcone_-(W_{n-k})$. That is, no point in
  $N_{-\eps}(\halfcone_-(W_{n-k}))$ is contained in
  $\stable{d-1}(\rho(\gamma))$.
\end{proof}

\subsection{Bounding unstable subspaces away from walls}

The estimate given by the previous lemma tells us that unstable
subspaces for $\rho(\gamma)$ are located near half-cones over the
walls $W_1, \ldots, W_n$ in $\vindomain$ separating a basepoint $x_0$
from $\rho(\gamma)x_0$. We will want to apply this lemma to see that
this unstable subspace is uniformly far from some hyperplane in the
complement of $\halfcone_+(W_1)$.

If the sequence of halfcones $\halfcone_+(W_k)$ is \emph{strongly
  nested} (see \Cref{sec:strong_nesting}), then this poses little
problem. However, if the closed walls separating $x_0$ from
$\rho(\gamma)x_0$ have a common intersection in the boundary of
$\vindomain$, then so do the boundaries of the half-cones over these
walls. A priori, the unstable subspace $\unstable{1}(\rho(\gamma))$
could lie near this intersection of boundaries, and we would have no
way to bound the distance between $\unstable{1}(\rho(\gamma))$ and the
complement of $\halfcone_+(W_1)$.
\begin{figure}[h]
  \centering
\begingroup%
  \makeatletter%
  \providecommand\color[2][]{%
    \errmessage{(Inkscape) Color is used for the text in Inkscape, but the package 'color.sty' is not loaded}%
    \renewcommand\color[2][]{}%
  }%
  \providecommand\transparent[1]{%
    \errmessage{(Inkscape) Transparency is used (non-zero) for the text in Inkscape, but the package 'transparent.sty' is not loaded}%
    \renewcommand\transparent[1]{}%
  }%
  \providecommand\rotatebox[2]{#2}%
  \newcommand*\fsize{\dimexpr\f@size pt\relax}%
  \newcommand*\lineheight[1]{\fontsize{\fsize}{#1\fsize}\selectfont}%
  \ifx\svgwidth\undefined%
    \setlength{\unitlength}{234.58628497bp}%
    \ifx\svgscale\undefined%
      \relax%
    \else%
      \setlength{\unitlength}{\unitlength * \real{\svgscale}}%
    \fi%
  \else%
    \setlength{\unitlength}{\svgwidth}%
  \fi%
  \global\let\svgwidth\undefined%
  \global\let\svgscale\undefined%
  \makeatother%
  \begin{picture}(1,0.53910533)%
    \lineheight{1}%
    \setlength\tabcolsep{0pt}%
    \put(0,0){\includegraphics[width=\unitlength,page=1]{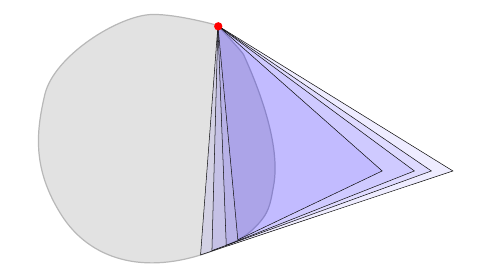}}%
    \put(0.16320365,0.21009496){\color[rgb]{1,0,0}\makebox(0,0)[lt]{\lineheight{1.25}\smash{\begin{tabular}[t]{l}$\P(V_T)$\end{tabular}}}}%
    \put(0.45062242,0.51224891){\color[rgb]{1,0,0}\makebox(0,0)[lt]{\lineheight{1.25}\smash{\begin{tabular}[t]{l}$\P(V_T^\perp)$\end{tabular}}}}%
    \put(0,0){\includegraphics[width=\unitlength,page=2]{wall_intersection.pdf}}%
  \end{picture}%
\endgroup%

  \caption{A sequence of nested half-cones over walls in an itinerary
    $W_1, \ldots, W_n$ where
    $\overline{W_1} \cap \overline{W_n} \ne
    \emptyset$. \Cref{lem:stable_subspace_in_halfcones} and
    \Cref{lem:stable_subspaces_standard_subgroups} show that if the
    element traversed by this itinerary has large singular value gaps,
    then its unstable subspace lies in $\halfcone_+(W_k)$ for some large
    $k$, and is far from $\overline{W_1} \cap \overline{W_n}$. In
    particular, the subspace is far from the complement of
    $\halfcone_+(W_1)$.}
  \label{fig:wall_intersection}
\end{figure}

The lemma below tells us that this does not occur. The idea is that,
when the closed walls $\overline{W_1}, \ldots, \overline{W_n}$ have a
common intersection in $\dee \vindomain$, then $\gamma$ lies ``close''
to a proper standard subgroup $C(T) < C$. The intersection of
these walls lies in a ``trivial subspace'' for the action of this
subgroup, which must be far from any attracting subspace for
$\gamma$. See \Cref{fig:wall_intersection}.

\begin{lem}
  \label{lem:stable_subspaces_standard_subgroups}
  Let $\Omega$ be a reflection domain for a simplicial representation
  $\rho$ with fully nondegenerate Cartan matrix. Then there exist
  $\eps_0, M > 0$ satisfying the following. Suppose that
  $W_1, \ldots, W_n$ is an $\Omega$-itinerary departing from the
  identity, satisfying $W_1 \cap W_n = \emptyset$ and
  $\overline{W_1} \cap \overline{W_n} \ne \emptyset$, and let
  $\gamma = \elt(W_1, W_n)$. If $\mu_{1,2}(\rho(\gamma)) > M$, then
  the distance between $\unstable{1}(\rho(\gamma))$ and
  $\overline{W_1} \cap \overline{W_n}$ is at least $\eps_0$.
\end{lem}
\begin{proof}
  We can replace the itinerary $W_1, \ldots, W_n$ with an efficient
  itinerary from $W_1$ to $W_n$ also departing from the identity,
  since this does not affect the element $\elt(W_1, W_n)$. For each
  $i$, we write $s_i \in S$ for the type $\type{W_i}$ of the wall
  $i$. By \Cref{lem:itinerary_wall_intersect} the intersection
  $\overline{W_1} \cap \overline{W_n}$ is equal to the intersection
  $\bigcap_{i=1}^{n} \overline{W_i}$, and by
  \Cref{lem:wall_intersect_gen} this intersection is in turn contained
  in $\bigcap_{i=1}^{n}\P(\ker(\alpha_{s_i}))$.

  Since this intersection is nonempty, the set of elements $s_i$ for
  $1 \le i \le n$ lie in a proper subset $T \subsetneq S$. Then
  $\gamma \in C(T)$. Since the Cartan matrix for $\rho$ is fully
  nondegenerate, \Cref{lem:fully_nondegenerate_transverse} implies
  that $\rho(\gamma)$ acts block-diagonally on $V$ with respect to a
  decomposition $V_T \oplus V_T^\perp$, where
  \[
    V_T = \spn\{v_s : s \in T\}, \qquad V_T^\perp = \bigcap_{s \in T}
    \ker(\alpha_s).
  \]
  Since these subspaces are transverse, we can choose $\eps_0 > 0$ so
  that the distance between the projective subspaces
  $\P(V_T), \P(V_T^\perp) \subset \P(V)$ is at least $2\eps_0$. Since
  there are only finitely many possible choices for $T$, we may choose
  this $\eps_0$ independently of $\gamma$.
  
  Next, we can choose a linear map $h \in \SL(V)$, depending only on
  $T$, so that the decomposition $hV_T \oplus hV_T^\perp$ is
  orthogonal. Then the conjugate $g = h\rho(\gamma)h^{-1}$ acts
  block-diagonally on a pair of orthogonal subspaces, and
  \Cref{lem:root_triangle_inequality} implies that there is a bound on
  the difference $|\mu_{1,2}(\rho(\gamma)) - \mu_{1,2}(g)|$, depending
  only on $h$ (and therefore only on $T$). In particular, this bound
  is independent of $\gamma$, so if $\mu_{1,2}(\rho(\gamma))$ is
  sufficiently large, we know that $\mu_{1,2}(g) > 0$. Since $g$ acts
  by the identity on $hV_T^\perp$, this means that $\unstable{1}(g)$
  lies in the subspace $hV_T$.

  Then, Lemma A.4 and A.5 in \cite{BPS} imply that we have
  \begin{align*}
  \dproj\left( \unstable{1}(\rho(\gamma)), h^{-1} \unstable{1}(g) \right) & \leq \dproj\left( \unstable{1}(\rho(\gamma)), \unstable{1}(\rho(\gamma) h^{-1}) \right) + \dproj\left( \unstable{1}(\rho(\gamma)h^{-1}), h^{-1} \unstable{1}(g) \right) \\
   & \leq e^{\mu_{1,d}(h) -\mu_{1,2}(\rho(\gamma))} + e^{\mu_{1,d}(h) - \mu_{1,2}(\rho(\gamma) h^{-1})} \\
   & \leq e^{-\mu_{1,2}(\rho(\gamma))} \left( e^{\mu_{1,d}(h)}  + e^{2\mu_{1,d}(h)} \right)
  \end{align*}
which is smaller than $\eps_0$ once
$\mu_{1,2}(\rho(\gamma))$ is sufficiently large. Since $h^{-1}\unstable{1}(g)$ lies in $V_T$, we conclude that the unstable
  subspace $\unstable{1}(\rho(\gamma))$ lies within distance $\eps_0$
  of $\P(V_T)$, hence it lies at least distance $\eps_0$ from
  $\P(V_T^\perp)$. But we have ensured that the intersection
  $\overline{W_1} \cap \overline{W_n}$ lies in $\P(V_T^\perp)$, so we
  get the desired result.
\end{proof}

\section{Proof of main theorem}
\label{sec:main_thm}

In this section we finally set about proving \Cref{thm:mainthm}. We
assume that $\Gamma$ is a hyperbolic quasiconvex subgroup of a
right-angled Coxeter group $C$ with generating set $S$. We let
$d = |S|$, and fix a simplicial representation
$\rho\colon C \to \SLpm(d, \R)$ with fully nondegenerate Cartan
matrix.

We want to prove the following uniform inequality: there exist
constants $A, B >0$ such that for any $\gamma \in \Gamma$,
$$
\mu_{1,2}(\rho(\gamma)) \geq A|\gamma| - B.
$$
To do so, it will be more convenient to work on a larger subset of $C$
containing $\Gamma$:
\begin{definition}
  Let $D > 0$. We let $\BP(D) \subset C$ denote the subset of group
  elements with $D$-bounded product projections (see \Cref{defn:bpp}).
\end{definition}

\begin{remark}
  If $C$ is not hyperbolic, then $\BP(D)$ is not a subgroup of $C$ for
  any $D$.
\end{remark}

\Cref{lem:bpp_poset} says precisely that there exists a uniform
constant $D$ such that $\Gamma \subseteq \BP(D)$. We choose to work
with the set $\BP(D)$ rather than the subgroup $\Gamma$ because
$\BP(D)$ has the following useful property: if $w$ is a geodesic word
representing some $\gamma \in \BP(D)$, and $w'$ is a subword of $w$,
then $w'$ also represents an element of $\BP(D)$. This means that we
can prove statements about group elements in $\BP(D)$ by cutting those
elements up into smaller pieces and applying inductive arguments.

Recall that for each subset $T \subseteq S$, the $C(T)$-invariant
subspace $V_T \subseteq V$ is given by $\spn\{v_s : s \in T\}$. We let
$\rho_T\colon C(T) \to \SLpm(V_T)$ denote the representation of $C(T)$
obtained by restricting the $\rho$-action of $C(T)$ to $V_T$.

Below, we will show the following.
\begin{prop}
  \label{prop:inductive_statement}
  For any subset $T \subseteq S$ and any $D > 0$, there exist
  constants $A, B > 0$ such that, if $\gamma \in \BP(D) \cap C(T)$,
  then $\mu_{1,2}(\rho_T(\gamma)) \ge A |\gamma| - B$.
\end{prop}

This gives us the main theorem by taking $T = S$.

\subsection{Setting up the inductive statement}

Since we have assumed that the simplicial representation $\rho$ has
fully nondegenerate Cartan matrix,
\Cref{lem:fully_nondegenerate_transverse} implies that for any subset
$T \subseteq S$, the restriction of the representation $\rho$ to the
subgroup $C(T)$ decomposes as a product $\rho_{T} \oplus \mathrm{id}$
on $V_{T} \oplus V_T^\perp$, where $\mathrm{id}$ denotes the trivial
representation $C(T) \to \SLpm(V_T^\perp)$. After a bounded change in
inner product (which only changes singular values up to bounded
multiplicative error), we can assume that the subspaces
$V_T, V_T^\perp$ are orthogonal.

So, when $\gamma \in C(T)$, the inequality in
\Cref{prop:inductive_statement} is satisfied with uniform constants
for $\rho_T(\gamma)$ if and only if it is satisfied for $\rho(\gamma)$
(possibly with different constants). Further, the representation
$\rho_T$ is also a simplicial representation of the right-angled
Coxeter group $C(T)$ with fully nondegenerate Cartan matrix.

So, to prove \Cref{prop:inductive_statement}, we assume inductively
that there are constants $A, B > 0$ so that for any proper subset
$T \subsetneq S$, if $\gamma \in C(T) \cap \BP(D)$, then
\begin{equation}
  \label{eq:subgroup_inequality}
  \mu_{1,2}(\rho(\gamma)) \ge A |\gamma| - B.
\end{equation}
We wish to show that, possibly after changing the constants $A$ and
$B$, this inequality also holds for any $\gamma \in \BP(D)$. If
$|S| = 1$, the conclusion of \Cref{prop:inductive_statement} is
obvious, since in this case $C \cong \Z/2$. So we just need to
consider the inductive step. Before doing so, we make some simplifying
arguments.

First, if $C$ is finite, then so is $\BP(D)$, and the desired
inequality holds trivially. So, we may assume $C$ is infinite.

Next, for any $\gamma \in \BP(D)$, we can find a (possibly trivial)
product decomposition of $C$ into $C(S_1) \times C(S_2)$ so that
$C(S_1)$ is irreducible and the image of $\gamma$ under the projection
$C \to C(S_2)$ has length at most $D$; this holds because every wall
in $\daviscx$ whose type lies in $S_1$ intersects every wall whose
type lies in $S_2$, and vice versa. Then, writing
$\gamma = \gamma_1\gamma_2$ for $\gamma_1 \in C(S_1)$,
$\gamma_2 \in C(S_2)$, we apply \Cref{lem:root_triangle_inequality} to
see that
\[
  \mu_{1,2}(\rho(\gamma)) \ge \mu_{1,2}(\rho(\gamma_1)) - KD
\]
for a uniform constant $K$. So, by replacing $\gamma$ with $\gamma_1$,
for the purposes of proving \Cref{prop:inductive_statement} we can
assume that $\gamma \in \BP(D)$ always lies in a subgroup $C(T)$ of
$C$ which is both infinite and irreducible.

If this subgroup is proper, then we can directly apply the inductive
assumption \eqref{eq:subgroup_inequality} to obtain the desired
lower bound on $\mu_{1,2}(\rho(\gamma))$. So we may assume that $C$ is
itself irreducible.

We have therefore reduced the inductive step of our proposition to the
following statement:
\begin{prop}
  \label{prop:inductive_statement_2}
  Suppose that $C$ is infinite and irreducible, and fix $D > 0$. If
  there exist uniform constants $A_0, B_0 > 0$ such that, for any
  proper subset $T \subsetneq S$ and any $\eta \in C(T) \cap \BP(D)$,
  we have
  \begin{equation} \label{eq:vertex_sv_gap} \mu_{1,2}(\rho(\eta)) \ge
    A_0|\eta| - B_0,
  \end{equation}
  then there are constants $A, B > 0$ such that for any
  $\gamma \in \BP(D)$, we have
  \[
    \mu_{1,2}(\rho(\gamma)) \ge A|\gamma| - B.
  \]
\end{prop}

The rest of the section is devoted to the proof of this
proposition. So, from now on, we assume that the hypotheses of the
proposition hold. Since we assume $C$ is infinite and irreducible, we
may assume that the Vinberg domain $\vindomain$ is properly convex by
\Cref{prop:vinberg_properly_convex}.

\subsection{Cutting itineraries into pieces}

Fix an element $\gamma \in \BP(D)$. The first step in proving the
estimate in \Cref{prop:inductive_statement_2} is to cut an itinerary
traversing $\gamma$ into several sub-itineraries, as follows. We let
$\mc{U}$ be an itinerary traversing $\gamma$ of the form
\[
  \{W_1\}, \mc{V}_1, \{W_2\}, \mc{V}_2, \ldots, \{W_n\}, \mc{V}_n,
\]
which satisfies the conclusions of \Cref{prop:disjoint_walls_I} (and
of \Cref{cor:disjoint_walls}). In particular, in the partial order $<$
on walls in $\mc{U}$, we have $W_1 < \ldots < W_n$. We let
$\mathbf{W}$ denote the set $\{W_i : 1 \le i \le n\}$, and for any
$W_i, W_j \in \mathbf{W}$ with $W_i < W_j$, we let $\mc{U}(W_i, W_j)$
denote the sub-itinerary of $\mc{U}$ starting with $W_i$ and ending
with $W_j$. We also assume that the itinerary $\mc{U}$ departs from
the identity; this means that
$\halfspace_+(W_j) \subset \halfspace_+(W_i)$ whenever $i < j$.

We then choose a subset $\mathbf{Z} = \{Z_1, \ldots, Z_N\}$ of the
walls $W_i$ as follows: we take $Z_1 = W_1$. Then, for each $j > 1$,
we take $Z_j$ to be the first wall $W_i > Z_{j-1}$ such that
$\overline{\halfcone_+(W_i)} \subset \halfcone_+(Z_{j-1})$. If there
is no such wall, and $Z_{j-1}$ is not already maximal in $\mathbf{W}$,
we let $Z_j = Z_N$ be the maximal wall in $\mathbf{W}$. It follows
that for every $i < j < N$, we have
\begin{equation}
  \label{eq:itinerary_strict_halfcones}
  \overline{\halfcone_+(Z_j)} \subset \halfcone_+(Z_i).
\end{equation}

Each pair of walls $Z_i, Z_{i+1}$ now determines a sub-itinerary
$\mc{U}(Z_i, Z_{i+1})$ of $\mc{U}$.

\subsection{Outline of the rest of the proof}
\label{sec:cutting_gluing}

We separate our sub-itineraries $\mc{U}(Z_i, Z_{i+1})$ into ``short''
and ``long'' ones, where the distinction depends on some uniform
threshold on the length to be specified later.

{\bf If our itinerary spends a uniform proportion of its lifetime
  inside ``short'' sub-itineraries} (i.e.\ a proportion above a
uniform threshold $\tau$, specified independent of $\gamma$), then the
number of ``short'' sub-itineraries is at least a constant times
$|\gamma|$. Using \eqref{eq:itinerary_strict_halfcones}, we obtain a
sequence of strictly nested half-cones. The Hilbert diameters of these
half-cones are decreasing at a uniform exponential rate. Then we use
\Cref{lem:hilbert_metric_estimates_sv_gaps} to get a corresponding
uniform exponential estimate on the singular value gap.

{\bf Otherwise, we can assume the geodesic spends a large fraction of
  its lifetime inside of ``long'' sub-itineraries.} For this case, we
use the results of \Cref{sec:half_cone_nesting} to observe that each
sub-itinerary traverses a group element which (up to uniformly bounded
error) either lies in a proper standard subgroup of $C$, or else can
be decomposed into into a product of two group elements, each of which
lies in a proper standard subgroup of $C$. Using our inductive
assumption, we can then show that each of our sub-itineraries
traverses a group element whose singular value gap is uniformly
exponential in its length.

We then need to show that when we ``glue together'' all of these
sub-itineraries, we obtain an itinerary which also traverses a group
element with uniformly exponential singular value gap. There are
several different ``gluing'' steps in this process, but each of them
essentially relies on just two techniques. First, if two adjacent
``long'' sub-itineraries are separated by a short sub-itinerary whose
initial and final walls are well-separated in $\vindomain$, then we
can use this separation to estimate positions of stable and unstable
subspaces, and use uniform transversality (\Cref{lem:root_additivity})
to estimate singular value gaps.

On the other hand, if we \emph{cannot} separate two adjacent ``long''
sub-itineraries $\mc{V}_-, \mc{V}_+$ by a pair of well-separated
walls, then every sub-itinerary which ``overlaps'' both $\mc{V}_-$ and
$\mc{V}_+$ must (up to some uniform error) traverse a group element
lying in a proper standard subgroup of $C$. Once again, our inductive
assumption implies that these overlapping sub-itineraries also have
uniformly exponential singular value gaps. By applying the
Kapovich-Leeb-Porti local-to-global principle
(\Cref{thm:local_to_global}), we get a uniform gap estimate for the
geodesic traversed by the concatenation $\mc{V}_-, \mc{V}_+$. The
length constant in the local-to-global lemma is ultimately what
determines the threshold for a ``long'' sub-itinerary.

By repeatedly using these two gluing arguments, we are eventually able
to show that if the proportion of time spent in ``long''
sub-itineraries is sufficiently close to 1, then we obtain a uniform
gap estimate for the entire itinerary. Combining this estimate with
the previous case completes the proof.

\subsection{Estimating time spent in ``short'' sub-itineraries}

Given any itinerary $\mc{V}$, recall that the length $|\mc{V}|$ is the
word-length of the element $\mc{V}$ traverses, or equivalently the
number of walls appearing in $\mc{V}$. If $W_i < W_j < W_k$ are three
walls in $\mathbf{W}$, then we have
\[
  |\mc{U}(W_i, W_k)| = |\mc{U}(W_i, W_j)| + |\mc{U}(W_j, W_k)| - 1,
\]
since the wall $W_j$ is counted twice on the right-hand side. It will
therefore be convenient to introduce some additional notation, and let
$\udist(W_i, W_j)$ denote $|\mc{U}(W_i, W_j)| - 1$. With this
notation, for any $W_i < W_j < W_k$, we have
\begin{equation}
  \label{eq:udist_additive}
  \udist(W_i, W_k) = \udist(W_i, W_j) + \udist(W_j, W_k).
\end{equation}

Now suppose we have fixed a quantity $L > 0$. For any pair of walls
$W_i < W_j$ in $\mathbf{W}$, we define the truncated length
$\utruncl{W_i, W_j}$ by
\[
  \utruncl{W_i, W_j} :=
  \begin{cases}
    \udist(W_i, W_j), &\udist(W_i, W_j) < L\\
    0, &\udist(W_i, W_j) \ge L.
  \end{cases}
\]
Similarly, we define $\ulongl{W_i, W_j}$ to be $\udist(W_i, W_j)$ if
this is at least $L$, and $0$ otherwise.

Given the collection of walls $\mathbf{Z} = Z_1 < \ldots < Z_N$, we
define the quantity
\[
  r_L(\mathbf{Z}) := \frac{1}{\udist(Z_1, Z_N)}\sum_{i=1}^{N-1}
  \utruncl{Z_i, Z_{i+1}}.
\]
Roughly, $r_L(\mathbf{Z})$ is the proportion of time $\mc{U}$ spends
inside of ``short'' sub-itineraries bounded by elements of
$\mathbf{Z}$, where an itinerary is ``short'' if its length is at most
$L$.

\begin{notation}
  To simplify notation, for the rest of the paper, if $\mc{W}$ is an
  itinerary, we abbreviate $\toproot(\rho(\elt(\mc{W})))$ to
  $\iroot(\mc{W})$. Further, if $W_1, W_2$ are walls in $\vindomain$,
  we write $\iroot(W_1, W_2)$ for $\toproot(\rho(\elt(W_1, W_2)))$.

  Similarly, we write $\stableunstable{i}(\mc{W})$ for $\stableunstable{i}(\rho(\elt(\mc{W})))$, and $\stableunstable{i}(W_1, W_2)$ for
  $\stableunstable{i}(\rho(\elt(W_1, W_2)))$.
\end{notation}

\subsection{Mostly short sub-itineraries}

We first want to prove the following:
\begin{prop}
  \label{prop:short_intervals_bound}
  For any $L > 0$ and $\tau \in (0, 1)$, there are constants
  $A \ge 0, B > 0$ (depending on $L, \tau$) such that if
  $r_L(\mathbf{Z}) \ge \tau$, then $\iroot(\mc{U}) \ge A|\mc{U}| - B$.
\end{prop}
\begin{proof}
  We let $L > 0$ and $\tau \in (0, 1)$ be given, and let $m$ denote
  the number of indices $i$ in $1, \ldots, N$ such that
  $\udist(Z_i, Z_{i+1}) < L$. Thus, we have
  \[
    m \cdot L \ge \sum_{i=1}^{N-1} \utruncl{Z_i, Z_{i+1}},
  \]
  and since we assume $r_L(\mathbf{Z}) \ge \tau$, we have
  \[
    m \ge \frac{\tau}{L} \cdot \udist(Z_1, Z_N).
  \]
  We now consider the sequence of half-cones
  $\halfcone_+(Z_1), \ldots, \halfcone_+(Z_N)$. For every index $n$,
  let $\gamma_n$ denote the group element traversed by the itinerary
  $\mc{U}(Z_1, \ldots, Z_n)$, and let $s_n = \type{Z_n}$. Then
  \Cref{lem:halfcone_formula} implies that
  \[
    \halfcone_+(Z_n) = \rho(\gamma_n) \halfcone_-(s_n),
  \]
  where, for any $s \in S$, $\halfcone_+(s)$ denotes the positive
  half-cone over the reflection wall for $\rho(s)$.

  Let $d_n$ denote the Hilbert metric on the half-cone
  $\halfcone_+(Z_n)$. As the half-cones $\halfcone_+(Z_{n-1})$ and
  $\halfcone_+(Z_n)$ are always nested,
  \Cref{lem:hilbert_metric_contracts} implies that $d_n \ge d_{n-1}$
  for all $n \le N$.

  When $n < N$, then the half-cones $\halfcone_+(Z_{n-1})$ and
  $\halfcone(Z_n)$ are \emph{strongly} nested. Equivalently,
  \begin{equation}
    \label{eq:halfcone_sequence_nest}
    \rho(\gamma_{n-1}^{-1}\gamma_n)\overline{\halfcone_-(s_n)} \subset
    \halfcone_-(s_{n-1}).
  \end{equation}
  For any $n$, we have
  $\gamma_{n-1}^{-1}\gamma_n = s_{n-1}\elt(Z_{n-1}, Z_n)$. So, if
  $\udist(Z_{n-1}, Z_n) < L$, there only are finitely many possible
  choices for the group element $\rho(\gamma_{n-1}^{-1}\gamma_n)$ in
  the inclusion \eqref{eq:halfcone_sequence_nest} above, as well as
  finitely many possible choices for the half-cones $\halfcone_-(s_n)$
  and $\halfcone_-(s_{n-1})$.

  We can then apply the sharper form of
  \Cref{lem:hilbert_metric_contracts} to see that there is a uniform
  constant $\lambda > 1$ (depending only on $\rho, L$) so that, if
  $n < N$, $\udist(Z_{n-1}, Z_n) < L$ and $x, y \in \halfcone_+(Z_n)$,
  then
  \begin{equation}
    \label{eq:hilbert_contract}
    d_n(x,y) \ge \lambda \cdot d_{n-1}(x,y).
  \end{equation}
  Now let $\ell$ be the last index smaller than $N$ such that
  $\udist(Z_{\ell - 1}, Z_\ell) < L$, and let $D_\ell$ be the diameter
  of $\halfcone_+(Z_\ell)$ with respect to the Hilbert metric on
  $\halfcone_+(Z_{\ell - 1})$, or equivalently the diameter of
  $\rho(\gamma_{\ell-1}^{-1}\gamma_\ell)\halfcone_-(s_\ell)$ with
  respect to the Hilbert metric on $\halfcone_-(s_{\ell-1})$. Since
  there are only finitely many possible values for these two strongly
  nested half-cones, we can bound $D_\ell$ beneath a uniform constant
  $D$, which depends only on $L$.

  From \eqref{eq:hilbert_contract}, the inequality $d_n \ge d_{n-1}$,
  and \eqref{eq:halfcone_sequence_nest} it follows that the diameter
  of the half-cone $\halfcone_+(Z_N) = \rho(\gamma_N)\halfcone_-(s_N)$
  with respect to the Hilbert metric on $\halfcone_+(Z_1)$ is at most
  $\lambda^{-(m-1)}D$.
  
  As there are only finitely many possible values for
  $\halfcone_+(Z_1)$ and $\halfcone_-(s_N)$, we can then apply
  \Cref{lem:hilbert_metric_estimates_sv_gaps} to see that
  \[
    \iroot(\rho(\gamma_N)) \ge m \cdot \log(\lambda) - B
  \]
  for some uniform constant $B$. By \Cref{prop:disjoint_walls_I}, the
  group element $\gamma_N = \elt_{\mc{U}}(Z_1, Z_N)$ is within
  uniformly bounded distance of $\elt(\mc{U})$ in the Coxeter group
  $C$, and $m$ is uniformly linear in $\udist(Z_1, Z_N)$ (which is
  within uniformly bounded additive error of $|\mc{U}|$). So we can
  apply \Cref{lem:root_triangle_inequality} to complete the proof.
\end{proof}

\subsection{Mostly long sub-itineraries}

Our goal for the rest of the section is to prove the following:
\begin{prop}
  \label{prop:long_intervals_bound}
  There exist constants $A > 0$, $B, L \ge 0$ and $\tau \in (0,1)$ so
  that if $r_L(\mathbf{Z}) \le \tau$, then
  $\iroot(\mc{U}) \ge A|\mc{U}| - B$.
\end{prop}

Since the constants $A, B, L, \tau$ above can all be chosen
independently of $\mc{U}$ and $\gamma$, putting
\Cref{prop:long_intervals_bound} together with
\Cref{prop:short_intervals_bound} will imply
\Cref{prop:inductive_statement_2}, which in turn gives us the desired
\Cref{prop:inductive_statement} (and thus our main
theorem). Throughout the section, we will obtain many intermediate
estimates for various sub-itineraries $\mc{U}'$ of $\mc{U}$, which all
roughly have the form
\[
  \iroot(\mc{U}') \ge A|\mc{U}'| - B
\]
for some constants $A, B$. We will not keep track of the different
constants $A,B$ for all of these different estimates.

Most of this section involves obtaining and refining various estimates
for $\iroot(\mc{U}(W_i, W_j))$ and $\iroot(W_i, W_j)$ for certain
walls $W_i, W_j \in \mathbf{W}$. Due to \Cref{cor:disjoint_walls}.\ref{item:elts_in_coxbnhd} and \Cref{lem:root_triangle_inequality}, we
can go back and forth between $\elt_{\mc{U}}(W_i, W_j)$ and
$\elt(W_i, W_j)$ when we make these estimates. Precisely, we can say
the following:
\begin{lem}
  \label{lem:itinerary_wall_estimates}
  There are uniform constants $R_1, R_2$ so that for any walls
  $W_i < W_j$ in $\mathbf{W}$, we have
  \[
    |\iroot(\mc{U}(W_i, W_j)) - \iroot(W_i, W_j)| < R_1
  \]
  and
  \[
    \left| |\elt(W_i, W_j)| - \udist(W_i, W_j) \right| < R_2.
  \]
\end{lem}

Using \Cref{lem:root_triangle_inequality} and the above, we also
obtain the following:
\begin{lem}
  \label{lem:wall_additive_error}
  There is a uniform constant $K > 0$ so that for any walls
  $W_i < W_j < W_k$ in $\mathbf{W}$, we have
  \begin{align*}
        |\iroot(W_i, W_k) - \iroot(W_i, W_j)| &\leq K \cdot \udist(W_j, W_k), \\
        |\iroot(W_i, W_k) - \iroot(W_j, W_k)| &\leq K \cdot \udist(W_i, W_j).
  \end{align*}
\end{lem}

\begin{definition}
  Let $W_i, W_\ell$ be walls in $\mathbf{W}$ with $W_i < W_\ell$. For
  constants $A, B > 0$, we say that the pair of walls $W_i, W_\ell$
  \emph{has $(A, B)$-gaps} if for every pair of walls $W_j, W_k$ with
  $W_i \le W_j < W_k \le W_\ell$, we have
  \begin{equation}
    \label{eq:regular_pair_inequality}
    \iroot(W_j, W_k) \ge A|\elt(W_j, W_k)| - B.
  \end{equation}
  If $\mc{V}$ is any itinerary equivalent to a sub-itinerary of
  $\mc{U}$, we say that $\mc{V}$ \emph{has $(A, B)$-gaps} if every
  pair of walls $W_i < W_\ell$ in $\mc{V} \cap \mathbf{W}$ has
  $(A, B)$-gaps.
\end{definition}

\begin{remark}
  \Cref{lem:itinerary_wall_estimates} means that we could equivalently
  define pairs of walls with $(A,B)$-gaps by replacing
  $\iroot(W_j, W_k)$ with $\iroot(\mc{U}(W_j, W_k))$ and
  $|\elt(W_j, W_k)|$ with $|\mc{U}(W_j, W_k)|$ (or with
  $\udist(W_j, W_k)$) in \eqref{eq:regular_pair_inequality} above---if
  we do, the constants $A, B$ change, but by a controlled amount
  depending only on the group $C$ and the representation $\rho$.

  Using \Cref{lem:root_triangle_inequality} and
  \Cref{cor:disjoint_walls}.\ref{item:subitineraries_minimal}, one can
  also see that, if $\mc{V}$ is a sub-itinerary of some itinerary
  equivalent to $\mc{U}$, then $\mc{V}$ has $(A,B)$-gaps if and only
  the geodesic following $\mc{V}$ has uniform gaps in the sense of
  \Cref{defn:URU_geodesic} (with constants depending on $A, B$).
\end{remark}

\subsubsection{Finding sub-itineraries with uniform gaps}

We have used the set of walls $\mathbf{Z}$ to cut the itinerary
$\mc{U}$ into pieces, so our first main task is to show that each
piece corresponds to a geodesic segment in $\Gamma$ with uniform
gaps. This will follow from:
\begin{prop}
  \label{prop:halfcone_intervals_regular}
  There are constants $A, B > 0$ so that every pair of walls $W < W'$
  in $\mathbf{W}$ which satisfies
  $\overline{\halfcone_+(W')} \not\subset \halfcone_+(W)$ has
  $(A, B)$-gaps.
\end{prop}

We need several intermediate lemmas in order to prove
\Cref{prop:halfcone_intervals_regular}. First, we observe that our
inductive assumption \eqref{eq:vertex_sv_gap} from
\Cref{prop:inductive_statement_2} and the quasiconvexity of standard
subgroups in $C$ implies the following:
\begin{lem}
  \label{lem:inductive_statement_itineraries}
  There are uniform constants $A, B$ so that, if $\mc{V}$ is any
  itinerary equivalent to a sub-itinerary of $\mc{U}$, and
  $\elt(\mc{V}) \in C(T)$ for some $T \subsetneq S$, then $\mc{V}$ has
  $(A, B)$-gaps.
\end{lem}

We next need two different ``straightness'' lemmas for sub-itineraries
of $\mc{U}$. Both \Cref{lem:wellsep_walls_I} and
\Cref{lem:wellsep_walls_II} below essentially say that if two adjacent
itineraries $\mc{V}_-, \mc{V}_+$ in $\mc{U}$ have sufficiently large
gaps, and there are walls between $\mc{V}_-$ and $\mc{V}_+$ which are
well-separated in $\vindomain$, then $\iroot(\mc{V})$ is approximately
$\iroot(\mc{V}_-) + \iroot(\mc{V}_+)$. The conclusion of both lemmas
is nearly the same, but in \Cref{lem:wellsep_walls_II}, we assume
slightly weaker separation hypotheses and slightly stronger hypotheses
on gaps than in \Cref{lem:wellsep_walls_I} (we refer to
\Cref{sec:appendix} for a discussion of the difference between the
separation hypotheses in these lemmas). We need both lemmas for the
arguments throughout the rest of the section.

\begin{lem}[Additivity for well-separated itineraries, I]
  \label{lem:wellsep_walls_I}
  Given $\lambda > 0$, there are constants $M, D > 0$ so that the
  following holds. Suppose that $W_i < W_j < W_k < W_\ell$ are walls in
  $\mathbf{W}$ satisfying:
  \begin{enumerate}
  \item\label{item:roots_large_I} $\iroot(W_i, W_j) > M$ and
    $\iroot(W_k, W_\ell) > M$;
  \item $\overline{\halfcone_+(W_k)} \subset \halfcone_+(W_j)$.
  \item $\udist(W_j, W_k) < \lambda$.
  \end{enumerate}
  Then we have
  \[
    \iroot(W_i, W_\ell) \ge \iroot(W_i, W_j) + \iroot(W_k, W_\ell) - D.
  \]
\end{lem}
\begin{proof}
  Let $\mc{V}$ be an efficient itinerary from $W_i$ to $W_\ell$, and
  write $\mc{V}$ as a concatenation $\mc{V}_-, \mc{V}_+$ so that
  $\mc{V}_-$ contains $W_i, W_j$ and $\mc{V}_+$ contains
  $W_k, W_\ell$. We change basepoints so that $\mc{V}_-$ arrives at
  the identity and $\mc{V}_+$ departs from the identity. After doing
  so, the inclusion $\halfspace_+(W_k) \subset \halfspace_+(W_j)$
  becomes $\halfspace_+(W_k) \subset \halfspace_-(W_j)$, which means
  we can assume that
  $\overline{\halfcone_+(W_k)} \subset \halfcone_-(W_j)$. We write
  $H_- = \halfcone_-(W_j)$ and $H_+ = \halfcone_+(W_k)$.

  We choose $\eps > 0$ so that
  $N_{2\eps}(H_+) \subset N_{-2\eps}(H_-)$. Since the distance
  $|\elt(W_j, W_k)| \le |\mc{U}(W_j, W_k)|$ is bounded by $\lambda$,
  this $\eps$ can be chosen uniformly in $\lambda$. Next, we choose
  $M' > 0$ (depending only on $\eps$ and $\lambda$) as in
  \Cref{lem:stable_subspace_in_halfcones}, so that if
  $\iroot(\mc{V}_+) > M'$, then
  $\unstable{1}(\rho(\elt(\mc{V}_+))) \in N_\eps(H_+)$, and if
  $\iroot(\mc{V}_-) > M'$, then $\stable{d-1}(\rho(\elt(\mc{V}_-)))$
  is disjoint from $N_{-\eps}(H_-)$. Finally, since
  $\udist(W_j, W_k) < \lambda$, we know that
  $\elt(\mc{V}_-) = \elt(W_i, W_j)\eta_-$ and
  $\elt(\mc{V}_+) = \eta_+\elt(W_k, W_\ell)$ for $\eta_\pm$ satisfying
  $|\eta_\pm| < \lambda$. So, by \Cref{lem:root_triangle_inequality},
  we can find $M$ so that if $\iroot(W_i, W_j) > M$, then
  $\iroot(\mc{V}_-) > M'$, and similarly for $W_k, W_\ell$ and
  $\mc{V}_+$.

  Now, applying \Cref{lem:stable_subspace_in_halfcones} and our
  hypothesis \eqref{item:roots_large_I} above, we see that an
  $\eps$-neighborhood of $\unstable{1}(\mc{V}_+)$ is disjoint from
  $\stable{d-1}(\mc{V}_-)$. That is, this pair of subspaces is
  uniformly transverse (see \Cref{fig:wellsep_walls_I}).

  \begin{figure}[h]
    \centering
\begingroup%
  \makeatletter%
  \providecommand\color[2][]{%
    \errmessage{(Inkscape) Color is used for the text in Inkscape, but the package 'color.sty' is not loaded}%
    \renewcommand\color[2][]{}%
  }%
  \providecommand\transparent[1]{%
    \errmessage{(Inkscape) Transparency is used (non-zero) for the text in Inkscape, but the package 'transparent.sty' is not loaded}%
    \renewcommand\transparent[1]{}%
  }%
  \providecommand\rotatebox[2]{#2}%
  \newcommand*\fsize{\dimexpr\f@size pt\relax}%
  \newcommand*\lineheight[1]{\fontsize{\fsize}{#1\fsize}\selectfont}%
  \ifx\svgwidth\undefined%
    \setlength{\unitlength}{293.83265458bp}%
    \ifx\svgscale\undefined%
      \relax%
    \else%
      \setlength{\unitlength}{\unitlength * \real{\svgscale}}%
    \fi%
  \else%
    \setlength{\unitlength}{\svgwidth}%
  \fi%
  \global\let\svgwidth\undefined%
  \global\let\svgscale\undefined%
  \makeatother%
  \begin{picture}(1,0.64657306)%
    \lineheight{1}%
    \setlength\tabcolsep{0pt}%
    \put(0,0){\includegraphics[width=\unitlength,page=1]{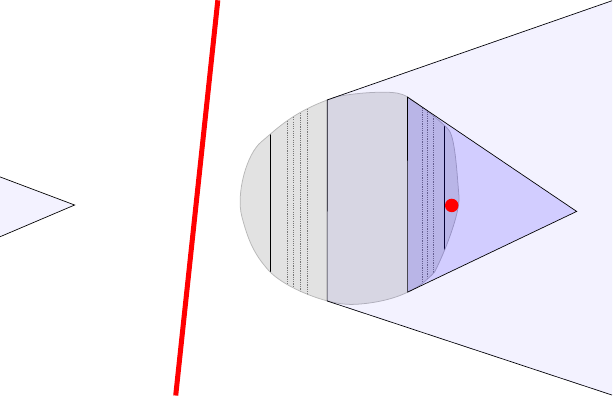}}%
    \put(0.413395,0.15658747){\color[rgb]{0,0,0}\makebox(0,0)[lt]{\lineheight{1.25}\smash{\begin{tabular}[t]{l}$W_i$\end{tabular}}}}%
    \put(0.51767147,0.10619846){\color[rgb]{0,0,0}\makebox(0,0)[lt]{\lineheight{1.25}\smash{\begin{tabular}[t]{l}$W_j$\end{tabular}}}}%
    \put(0.64858548,0.13235185){\color[rgb]{0,0,0}\makebox(0,0)[lt]{\lineheight{1.25}\smash{\begin{tabular}[t]{l}$W_k$\end{tabular}}}}%
    \put(0.7346711,0.23378655){\color[rgb]{0,0,0}\makebox(0,0)[lt]{\lineheight{1.25}\smash{\begin{tabular}[t]{l}$W_\ell$\end{tabular}}}}%
    \put(0.35558552,0.50827983){\color[rgb]{1,0,0.00392157}\makebox(0,0)[lt]{\lineheight{1.25}\smash{\begin{tabular}[t]{l}$\stable{d-1}(\mc{V}_-)$\end{tabular}}}}%
    \put(0.73374629,0.33212143){\color[rgb]{1,0,0.00392157}\makebox(0,0)[lt]{\lineheight{1.25}\smash{\begin{tabular}[t]{l}$\unstable{1}(\mc{V}_+)$\end{tabular}}}}%
  \end{picture}%
\endgroup%

    \caption{Illustration for the proof of
      \Cref{lem:wellsep_walls_I}. Since the half-cones over $W_j, W_k$
      strongly nest, the subspaces $\unstable{1}(\mc{V}_+)$ and
      $\stable{d-1}(\mc{V}_-)$ are uniformly transverse.}
    \label{fig:wellsep_walls_I}
  \end{figure}

  Thus, the additivity estimate (\Cref{lem:root_additivity}) together
  with \Cref{lem:wall_additive_error} gives us the desired bound.
\end{proof}

\begin{lem}[Additivity for well-separated itineraries, II]
  \label{lem:wellsep_walls_II}
  Given $A, \lambda > 0$ and $B \ge 0$, there is a constant $D > 0$ so that
  the following holds. Suppose $W_i < W_j < W_k < W_\ell$ are walls in
  $\mathbf{W}$ satisfying:
  \begin{enumerate}
  \item The pairs of walls $W_i, W_j$ and $W_k, W_\ell$ both have
    $(A, B)$-gaps.
  \item The intersection $\overline{W_j} \cap \overline{W_k}$ is
    empty.
  \item We have $\udist(W_j, W_k) < \lambda$.
  \end{enumerate}
  Then we have
  \begin{equation}
    \label{eq:disjoint_wall_ineq}
    \iroot(W_i, W_\ell) \ge \iroot(W_i, W_j) + \iroot(W_k, W_\ell) - D.
  \end{equation}
\end{lem}
\begin{proof}
  First, observe that if either $|\elt(W_i, W_j)|$ or
  $|\elt(W_k, W_\ell)|$ is smaller than some constant $N$, then we can
  use \Cref{lem:wall_additive_error} to find some constant $D$
  depending only on $N$ so that the inequality
  \eqref{eq:disjoint_wall_ineq} holds. So, throughout the proof, we
  will be able to assume that both $|\elt(W_i, W_j)|$ and
  $|\elt(W_k, W_\ell)|$ are larger than any given constant that depends
  only on $A, B, \lambda$. Since we have assumed that the pairs
  $W_i, W_j$ and $W_k, W_\ell$ have $(A,B)$-gaps, this means we can also
  assume that $\iroot(W_i, W_j)$ and $\iroot(W_k, W_\ell)$ are larger
  than any given constant depending on $A, B, \lambda$.

  \Cref{lem:halfcones_nest} implies that we have a nesting of
  half-cones $\halfcone_+(W_k) \subset \halfcone_+(W_j)$. If this
  nesting is strict, the reasoning from the previous paragraph tells
  us that we can apply \Cref{lem:wellsep_walls_I} and we will be
  done. So, we now assume that $\overline{\halfcone_+(W_k)}$ is not
  contained in $\halfcone_+(W_j)$.

  As in the proof of \Cref{lem:wellsep_walls_I}, we let $\mc{V}$ be an
  efficient itinerary from $W_i$ to $W_\ell$, and write it as a
  concatenation $\mc{V}_-, \mc{V}_+$ for $\mc{V}_-$ containing
  $W_i, W_j$ and $\mc{V}_+$ containing $W_k, W_\ell$; in fact, we can
  ensure that the first wall of $\mc{V}_+$ is $W_k$. We can again
  assume that $\mc{V}_-$ arrives at the identity and $\mc{V}_+$
  departs from the identity. As before, set $H_- = \halfcone_-(W_j)$
  and $H_+ = \halfcone_+(W_k)$, so that $H_+$ (but not
  $\overline{H_+}$) is contained in $H_-$.

  For each wall $W \in \mc{V}_+ \cap \mathbf{W}$, we let
  $\unstable{1}(W)$ denote the 1-dimensional unstable subspace of
  $\rho(\elt_{\mc{V}_+}(W))$, where $\elt_{\mc{V}_+}(W)$ is the group
  element traversed by a sub-itinerary of $\mc{V}_+$ starting with the
  initial wall of $\mc{V}_+$ and ending in $W$.

  Let $\eps_0, M_0$ be the fixed constants from
  \Cref{lem:stable_subspaces_standard_subgroups}. We know that
  $\mc{V}_+ = \mc{U}(W_k, W_\ell)$ has $(A, B)$-gaps. So, by applying the
  uniform convergence property for geodesics with uniform
  gaps in $\SLpm(d, \R)$ (\Cref{lem:uniform_convergence_regular}), we
  see that there is a fixed constant $N$ (depending only on $\eps_0$,
  $A$, $B$) so that if $W$ is any wall in $\mathbf{W}$ separating
  $W_k$ and $W_\ell$, with $|\elt(W_k, W)| \ge N$, we have
  \begin{equation}
    \label{eq:uniform_dist_estimate}
    \dproj(\unstable{1}(W), \unstable{1}(\mc{V}_+)) < \eps_0/2.
  \end{equation}
  By increasing $N$ if necessary, we can again apply the fact that
  $\mc{V}_+$ has uniform gaps to also ensure that as long as
  $|\elt(W_k, W)| \ge N$, then $\iroot(W_k, W) > M_0$.
  
  Now, we may assume that $N > 2R$, and that
  $|\elt(W_k, W_\ell)| \ge N$. Then, \Cref{cor:disjoint_walls}
  \ref{item:subitineraries_minimal} implies that there is a uniform
  $R > 0$ and some wall $W$ in $\mathbf{W}$ with $W_k < W < W_\ell$
  and $N < |\elt(W_k, W)| \le N + R$.  We let $H_+'$ denote the
  half-cone $\halfcone_+(W)$; from \Cref{lem:halfcones_nest}, we know
  that $H_+' \subset H_+ \subset H_-$. This tells us that
  $\dee H_+' \cap \dee H_- \subset \dee H_+ \cap \dee H_-$. But,
  \Cref{lem:halfcones_nest} also implies that
  $\dee H_+' \cap \dee H_-$ is a subset of $\dee W$, and that
  $\dee H_+ \cap \dee H_-$ is a subset of $\dee W_k$, which means we
  have
  \begin{equation}
    \label{eq:halfcone_boundary_intersection}
    \dee H_+' \cap \dee H_- \subset \overline{W} \cap \overline{W_k}.
  \end{equation}

  We now consider two possibilities.
  \begin{description}
  \item[Case 1] $\overline{W_k} \cap \overline{W} = \emptyset$.  In
    this case, the inclusion \eqref{eq:halfcone_boundary_intersection}
    means that the inclusion of half-cones $H_+' \subset H_-$ is
    strict, i.e.\ that $\overline{H_+'} \subset H_-$. So, we can
    replace $W_k$ with $W$ and apply \Cref{lem:wellsep_walls_I} (with
    $\lambda + N + R$ in place of $\lambda$); since the pairs
    $W_i, W_j$ and $W_k, W_\ell$ have uniform gaps, as long as
    $|\elt(W_i, W_j)|$ and $|\elt(W_k, W_\ell)|$ are sufficiently large,
    the hypotheses of \Cref{lem:wellsep_walls_I} are satisfied.

  \item[Case 2] $\overline{W_k} \cap \overline{W} \ne \emptyset$.  In
    this case, we will show that the subspaces
    $\unstable{1}(\mc{V}_+)$ and $\stable{d-1}(\mc{V}_-)$ are
    uniformly transverse by arguing that $\unstable{1}(\mc{V}_+)$ lies
    close to $H_+'$ and far from $\overline{W} \cap \overline{W_k}$,
    and $\stable{d-1}(\mc{V}_-)$ does not lie too close to $H_-$; see
    \Cref{fig:wellsep_walls_II}.

  \begin{figure}[h]
    \centering
    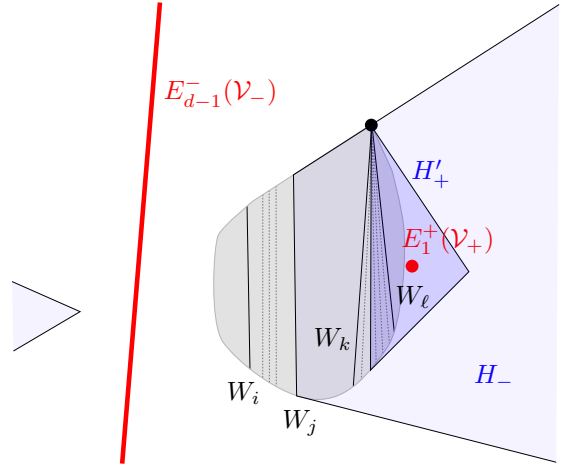
    \caption{The subspaces $\unstable{1}(\mc{V}_+)$ and
      $\stable{d-1}(\mc{V}_-)$ are uniformly transverse, since
      $\unstable{1}(\mc{V}_+)$ cannot lie close to the intersection of
      the boundaries of the half-cones $H_-, H_+'$.}
    \label{fig:wellsep_walls_II}
  \end{figure}

  Since the closures $\overline{H_+'}$ and $\overline{H_-}$ are
  compact, we can choose some $\eps > 0$ so that the intersection
  \[
    N_{2\eps}(H_+') \cap N_{\eps}(\P(V) \minus H_-)
  \]
  is contained in the $\eps_0/4$-neighborhood of
  $\dee H_+' \cap \dee H_-$ (and thus, by
  \eqref{eq:halfcone_boundary_intersection}, in the
  $\eps_0/4$-neighborhood of $\overline{W} \cap
  \overline{W_k}$). Since $|\elt(W_j, W)| < \lambda + N + R$, we can
  choose this $\eps$ uniformly in $\lambda, N, R, \eps_0$. We can also
  ensure that $\eps < \eps_0/4$.

  We now choose a constant $M$ which is larger than the corresponding
  $M$ from \Cref{lem:stable_subspace_in_halfcones}, taking the
  constant $\eps$ in the lemma to be our chosen $\eps$, and $N$ to be
  $\lambda + N + R$. We may assume that both $\iroot(\mc{V}_+)$ and
  $\iroot(\mc{V}_-)$ are at least $M$.

  We claim that the distance between $\unstable{1}(\mc{V}_+)$ and
  $\stable{d-1}(\mc{V}_-)$ must then be at least $\eps$. To see this,
  suppose for a contradiction that there is some point
  $x \in \stable{d-1}(\mc{V}_-)$ which is contained in the
  $\eps$-neighborhood of
  $\unstable{1}(\mc{V}_+)$. \Cref{lem:stable_subspace_in_halfcones}
  implies that $\unstable{1}(\mc{V}_+)$ is contained in $N_\eps(H_+')$
  and $\stable{d-1}(\mc{V}_-)$ is disjoint from $N_{-\eps}(H_-)$. This
  means that $x$ must lie in the intersection
  \[
    N_{2\eps}(H_+') \cap (\P(V) \minus N_{-\eps}(H_-)) =
    N_{2\eps}(H_+') \cap N_\eps(\P(V) \minus H_-).
  \]
  In turn, this means that $x$ is contained in the
  $\eps_0/4$-neighborhood of $\overline{W_k} \cap \overline{W}$ and so
  $\unstable{1}(\mc{V}_+)$ is contained in the $\eps_0/2$-neighborhood
  of $\overline{W_k} \cap \overline{W}$. Then, using
  \eqref{eq:uniform_dist_estimate}, we see that $\unstable{1}(W)$ is
  contained in the $\eps_0$-neighborhood of
  $\overline{W_k} \cap \overline{W}$. However, this contradicts
  \Cref{lem:stable_subspaces_standard_subgroups}.

  We finally conclude that the subspaces $\unstable{1}(\mc{V}_+)$ and
  $\stable{d-1}(\mc{V}_-)$ are uniformly transverse, which gives us
  the desired inequality by \Cref{lem:root_additivity} and
  \Cref{lem:wall_additive_error}. \qedhere
  \end{description}
\end{proof}

The last intermediate lemma we need for the proof of
\Cref{prop:halfcone_intervals_regular} is a direct consequence of the
Kapovich--Leeb--Porti local-to-global principle
(\Cref{thm:local_to_global}):
\begin{lem}[Local-to-global for pairs with uniform gaps]
  \label{lem:local_to_global_itineraries}
  For every $A, B > 0$, there exists $\lambda > 0$ and $A', B' > 0$
  satisfying the following. Suppose that $W_i, W_j$ are walls in
  $\mathbf{W}$ such that for every pair $n, m$ with
  $\udist(W_n, W_m) < \lambda$ and $W_i \le W_n < W_m \le W_j$, the
  pair $W_n, W_m$ has $(A,B)$-gaps. Then $W_i, W_j$ has
  $(A',B')$-gaps.
\end{lem}

\begin{proof}[Proof of \Cref{prop:halfcone_intervals_regular}]
  Fix walls $W_i < W_j$ in $\mathbf{W}$ so that
  $W \le W_i < W_j \le W'$. Since $W_i$ and $W_j$ separate $W$ and
  $W'$, \Cref{lem:halfcones_nest} and the assumption
  $\overline{\halfcone_+(W')} \not\subset \halfcone_+(W)$ means that
  the strong inclusion
  $\overline{\halfcone_+(W_j)} \subset \halfcone_+(W_i)$ cannot
  hold. We want to find uniform constants $A, B$ so that
  $\mu_{1,2}(W_i, W_j) \ge A|\elt(W_i, W_j)| - B$.

  If $\overline{W}_n \cap \overline{W}_m \ne \emptyset$, then
  \Cref{lem:disjoint_walls} implies that $\elt(W_i, W_j) \in C(T)$ for
  a proper standard subgroup $C(T) < C$, and we are done after
  applying \Cref{lem:inductive_statement_itineraries}. So, assume that
  $\overline{W_i} \cap \overline{W_j} = \emptyset$.

  We consider an efficient itinerary $\mc{V}$ between $W_i$ and
  $W_j$. By \Cref{lem:halfcone_nesting_failure}, $\mc{V}$ is
  equivalent to an itinerary of the form
  $\mc{W}_-, \mc{W}_0, \mc{W}_+$, where $|\mc{W}_0| < 2R$ for a
  uniform constant $R$, the itineraries $\mc{W}_-$ and $\mc{W}_+$ are
  efficient, and the intersections
  $\bigcap_{W \in \mc{W}_-} \overline{W}$ and
  $\bigcap_{W \in \mc{W}_+} \overline{W}$ are both nonempty. By
  \Cref{lem:disjoint_walls}, both $\elt(\mc{W}_-)$ and
  $\elt(\mc{W}_+)$ lie in (possibly different) proper standard
  subgroups of $C$.

  We let $\mc{W}_-'$ be the concatenation $\mc{W}_-, \mc{W}_0$, so
  that $\mc{V} = \mc{W}_-', \mc{W}_+$. By
  \Cref{lem:inductive_statement_itineraries} and
  \Cref{lem:wall_additive_error}, we see that $\mc{W}_-'$ and
  $\mc{W}_+$ both have $(A_0, B_0)$-gaps, for uniform constants
  $A_0, B_0$. We then choose constants $A, B, \lambda$ as in the
  local-to-global \cref{lem:local_to_global_itineraries}, depending on
  $A_0$, $B_0$.

  Suppose that for every pair of walls $W_n, W_m \in \mathbf{W}$
  such that $W_n \in \mc{W}_-'$, $W_m \in \mc{W}_+$, and
  $\udist(W_n, W_m) < \lambda$, we have
  $\overline{W_n} \cap \overline{W_m} \ne
  \emptyset$. \Cref{lem:disjoint_walls} then implies that each such
  pair satisfies $\elt(W_n, W_m) \in C(T)$ for some proper subgroup
  $C(T) < C$, and then \Cref{lem:inductive_statement_itineraries}
  implies that each such pair has $(A_0, B_0)$-gaps. Then
  \Cref{lem:local_to_global_itineraries} implies that the pair
  $W_i, W_j$ has $(A,B)$-gaps, and we are done.

  So, we may now suppose that there is some pair of walls
  $W_n, W_m \in \mathbf{W}$ such that $W_n \in \mc{W}_-'$,
  $W_m \in \mc{W}_+$, and $\udist(W_n, W_m) < \lambda$, but
  $\overline{W_n} \cap \overline{W_m} = \emptyset$. Since
  $\mc{W}_-', \mc{W}_+$ both have $(A_0, B_0)$-gaps, so do the pairs
  $W_i, W_n$ and $W_m, W_j$. We can then apply
  \Cref{lem:wellsep_walls_II} to the walls $W_i < W_n < W_m < W_j$ to
  complete the proof.
\end{proof}

As a consequence of \Cref{prop:halfcone_intervals_regular}, we obtain:
\begin{prop}
  \label{prop:z_wall_gaps}
  There are uniform constants $A, B > 0$ so that every pair of
  consecutive walls $Z_i < Z_{i+1}$ in $\mathbf{Z}$ has $(A,B)$-gaps.
\end{prop}
\begin{proof}
  Consider walls $W_n < W_m$ in $\mathbf{W}$ with
  $Z_i \le W_n < W_m \le Z_{i+1}$. We want to find constants $A, B$ so
  that $\iroot(W_n, W_m) \ge A|\elt(W_n, W_m)| - B$. We may assume
  that $W_m > W_{n+1}$: otherwise, by \Cref{prop:disjoint_walls_I}
  \ref{item:v_intersect} we have
  $|\elt(W_n, W_m)| = |\elt(W_n, W_{n+1})| = 1$, and the desired
  inequality holds as long as we ensure $B > A$.

  Now consider the pair of walls $W_n < W_{m - 1}$. We know that
  $W_{m-1} < Z_{i+1}$, which means that the strong inclusion of
  half-cones
  $\overline{\halfcone_+(W_{m - 1})} \subset \halfcone_+(W_n)$ cannot
  hold: in this case \Cref{lem:halfcones_nest} would imply that also
  $\overline{\halfcone_+(W_{m-1})} \subset \halfcone_+(Z_i)$, and
  $Z_{i+1}$ was chosen to be the minimal wall in $\mathbf{W}$ whose
  positive half-cone strongly nests into $\halfcone_+(Z_i)$. Then
  \Cref{prop:halfcone_intervals_regular} implies that
  $\iroot(W_n, W_{m-1}) \ge A|\elt(W_n, W_{m-1})| - B$ for some
  uniform constants $A, B$. So, we can apply
  \Cref{lem:wall_additive_error} and the fact that
  $\udist(W_{m-1}, W_m) < R$ to complete the proof.
\end{proof}

\subsubsection{Combining sub-itineraries with uniform gaps using
  the local-to-global principle}

\Cref{prop:halfcone_intervals_regular} shows that each pair of walls
$Z_i, Z_{i+1}$ defines a sub-itinerary $\mc{U}(Z_i, Z_{i+1})$ which
has $(A,B)$-gaps, for uniform $A, B$. For the next step of the proof,
we remove walls from the collection $\mathbf{Z}$ to ``combine''
adjacent sub-itineraries; we do so in a way which ensures that each
sub-itinerary bounded by consecutive walls in $\mathbf{Z}$ has
$(A',B')$-gaps, for uniform constants $A', B'$ which do not depend on
the itinerary $\mc{U}$.

More precisely, we modify $\mathbf{Z}$ via the following
procedure. Using \Cref{prop:halfcone_intervals_regular} and
\Cref{prop:z_wall_gaps}, we choose constants $A, B > 0$ so that every
pair of consecutive walls $Z_i, Z_{i+1}$ in $\mathbf{Z}$ has
$(A,B)$-gaps, \emph{and} so that every pair of walls $W < W'$ in
$\mathbf{W}$ satisfying
$\overline{\halfcone_+(W')} \not\subset \halfcone_+(W)$ also has
$(A,B)$-gaps. Then, we fix local-to-global constants $\lambda, A', B'$
from \Cref{lem:local_to_global_itineraries}, depending on our chosen
$A, B$. Now, suppose we have a sequence of consecutive elements
$Z_j, Z_{j+1}, \ldots, Z_{j + n}$ in $\mathbf{Z}$ which satisfies the
following property:
\begin{enumerate}[label=(\textdagger)]
\item \label{item:adjacent_nosep} For each wall $Z_i$ in $\mathbf{Z}$
  with $Z_j < Z_i < Z_{j + n}$, and every pair of walls
  $W, W' \in \mathbf{W}$ such that $W < Z_i < W'$ and
  $\udist(W, W') < \lambda$, we have
  $\overline{\halfcone_+(W')} \not\subset \halfcone_+(W)$.
\end{enumerate}
It immediately follows from \Cref{prop:z_wall_gaps},
\Cref{prop:halfcone_intervals_regular} and
\Cref{lem:local_to_global_itineraries} that the pair of walls
$Z_j, Z_{j + n}$ has $(A',B')$-gaps. So, for each maximal subsequence
$Z_j, \ldots, Z_{j + n}$ in $\mathbf{Z}$ satisfying
\ref{item:adjacent_nosep}, we delete all elements except $Z_j$ and
$Z_{j + n}$ from $\mathbf{Z}$.

After this modification (and after reindexing $\mathbf{Z}$
appropriately), it is still true that each pair of consecutive walls
$Z_i, Z_{i+1}$ has $(A, B)$-gaps for uniform $A, B$. Further, since we
have only deleted walls from $\mathbf{Z}$ (and we do not delete the
walls $Z_1$ or $Z_N$), the quantity $r_L(\mathbf{Z})$ can only
decrease for any fixed $L$. So we can still assume that
$r_L(\mathbf{Z}) < \tau$ if this was true before we deleted walls from
$\mathbf{Z}$.

After this modification, however, we gain the following. We now know
that our $\mathbf{Z}$ satisfies the following additional property:
\begin{prop}[Long adjacent itineraries are well-separated]
  \label{prop:separating_walls}
  Let $\lambda$ be the local-to-global constant defined above. Suppose
  that $Z_{i-1}, Z_i, Z_{i+1}$ are three consecutive walls in
  $\mathbf{Z}$ such that $\udist(Z_{i-1}, Z_i) \ge \lambda$ and
  $\udist(Z_i, Z_{i+1}) \ge \lambda$. Then there is a pair of walls
  $W, W' \in \mathbf{W}$ satisfying:
  \begin{enumerate}
  \item $Z_{i-1} < W < Z_i < W' < Z_{i+1}$;
  \item $\udist(W, W') < \lambda$;
  \item $\overline{\halfcone_+(W')} \subset \halfcone_+(W)$.
  \end{enumerate}
\end{prop}

\subsubsection{Combining well-separated adjacent sub-itineraries}

For the next step, we again combine sequences of adjacent
sub-itineraries by deleting elements from $\mathbf{Z}$.

Given $L_0 > 0$, we consider maximal subsequences
$Z_j, Z_{j+1}, \ldots, Z_{j + n}$ of $\mathbf{Z}$ which
satisfy:
\begin{enumerate}[label=($\star L_0$)]
\item\label{item:long_consecutive_itineraries} for all indices $i$
  with $j \le i < j + n$, we have
  $\udist(Z_i, Z_{i+1}) \ge L_0$.
\end{enumerate}

We claim the following:
\begin{lem}
  \label{lem:wellsep_gluing}
  There are uniform constants $A, L_0 > 0$ so that, if
  $Z_j, Z_{j + 1}, \ldots, Z_{j + n}$ is a subsequence of
  $\mathbf{Z}$ satisfying \ref{item:long_consecutive_itineraries},
  then
  \[
    \iroot(Z_j, Z_{j + n}) \ge A \cdot \udist(Z_j, Z_{j +
      n}).
  \]
\end{lem}
\begin{proof}
  As each pair of consecutive walls $Z_i < Z_{i+1}$ has uniform gaps,
  we know there are constants $A', B'$ so that
  $\iroot(Z_i, Z_{i+1}) \ge A'\udist(Z_i, Z_{i+1}) - B'$. Let $K$ be
  the constant from \Cref{lem:wall_additive_error}, let $\lambda$ be
  the local-to-global constant from \Cref{prop:separating_walls}, and
  let $M, D$ be the constants from \Cref{lem:wellsep_walls_I}
  depending on $\lambda$. Then, choose $L_0$ large enough so that
  \[
    A'L_0 - 3\lambda K - B' - D > M,
  \]
  and define $A = (A'L_0 - 2\lambda K - B' - D)/L_0$, so that
  $A L_0 > M + \lambda K$.

  To prove the proposition, we induct on the length of the sequence
  $Z_j, \ldots, Z_{j + n}$. In the base case $n = 1$, we have
  \begin{align*}
    \iroot(Z_j, Z_{j + 1})
    &\ge A'\udist(Z_j, Z_{j+1}) - B'\\
    &=\left(A + \frac{2\lambda K + B' + D}{L_0}\right)\udist(Z_j, Z_{j+1}) - B'.
  \end{align*}
  Since $\udist(Z_j, Z_{j+1}) \ge L_0$ by assumption, we see that
  \[
    \iroot(Z_j, Z_{j+1}) \ge A \cdot \udist(Z_j, Z_{j+1}) + 2\lambda K
    + D > A \cdot \udist(Z_j, Z_{j+1}).
  \] 
  
  We now assume that $n > 1$, so that inductively we have
  \[
    \iroot(Z_j, Z_{j + n - 1}) \ge A \cdot \udist(Z_j,
    Z_{j + n - 1}).
  \]
  From \Cref{prop:separating_walls}, we can find walls $W, W'$ with
  $W < Z_{j + n - 1} < W'$ satisfying
  $\overline{\halfcone_+(W')} \subset \halfcone_+(W)$ and
  $\udist(W, W') < \lambda$. In particular we know
  $\udist(W, Z_{j + n - 1}) < \lambda$, so by
  \Cref{lem:wall_additive_error} we know that
  \begin{equation}
    \label{eq:inductive_bound}
    \iroot(Z_j, W) \ge A \cdot \udist(Z_j, Z_{j + n - 1}) -
    \lambda K.
  \end{equation}
  In particular, since $\udist(Z_j, Z_{j + n - 1}) \ge L_0$, it
  follows that $\iroot(Z_j, W) > M$.

  We can also combine \Cref{lem:wall_additive_error} with the argument
  from the base case and the fact that
  $\udist(W', Z_{j + n}) < \lambda$ to see that
  \begin{equation}
    \label{eq:length1_bound}
    \iroot(W', Z_{j + n}) \ge A \cdot \udist(Z_{j + n - 1}, Z_{j +
      n}) + \lambda K + D.
  \end{equation}
  In particular, this tells us that
  $\iroot(W', Z_{j + n}) > M + 2\lambda K + D > M$.

  We now apply \Cref{lem:wellsep_walls_I} to the walls
  $Z_j < W < W' < Z_{j + n}$ to see that
  \begin{align*}
    \udist(Z_{j}, Z_{j + n})
    &\ge \iroot(Z_{j}, W) + \iroot(W', Z_{j + n}) - D. 
  \end{align*}
  Putting this inequality together with \eqref{eq:inductive_bound} and
  \eqref{eq:length1_bound} and using the fact that
  $\udist(Z_{j}, Z_{j + n}) = \udist(Z_{j}, Z_{j + n - 1}) +
  \udist(Z_{j + n - 1}, Z_{j + n})$ completes the proof.
\end{proof}

We keep the constant $L_0$ from \Cref{lem:wellsep_gluing} fixed for
the rest of the paper. As in the previous step, we now modify the
collection $\mathbf{Z}$ by replacing all maximal sequences
$Z_j, Z_{j + 1}, \ldots, Z_{j + n}$ satisfying
\ref{item:long_consecutive_itineraries} with $Z_j, Z_{j + n}$. Once
again, we know that for any $L$, the quantity $r_L(\mathbf{Z})$ can
only decrease after this modification.

After reindexing, the previous lemma now gives us the following:
\begin{prop}
  \label{prop:subitinerary_linear_growth}
  For a uniform constant $A_0$, and every $i$ with $1 \le i < N$, if
  $\udist(Z_i, Z_{i+1}) \ge L_0$, then
  $\iroot(Z_i, Z_{i+1}) > A_0 \cdot \udist(Z_i, Z_{i+1})$.
\end{prop}

In addition, $\mathbf{Z}$ now satisfies:
\begin{prop}[No adjacent long sub-itineraries]
  \label{prop:no_adjacent_long_itineraries}
  Let $Z_{i-1}, Z_i, Z_{i+1}$ be consecutive walls in
  $\mathbf{Z}$. Then either $\udist(Z_{i-1}, Z_i) < L_0$ or
  $\udist(Z_i, Z_{i+1}) < L_0$.
\end{prop}

\subsubsection{Choosing the threshold for ``long'' sub-itineraries}

Let $N = |\mathbf{Z}|$. For any given $L > 0$, and for each
$1 < i < j \le N$, we now define
\[
  t^+_L(i,j) := \sum_{n=i}^j \untrunc{L}{\udist(Z_{n-1}, Z_n)}.
\]
We similarly define $t^-_L(i, j)$ via the truncated length:
\[
  t_L^-(i, j) := \sum_{n=i}^j \trunc{L}{\udist(Z_{n-1}, Z_n)}.
\]
That is, $t^+_L(i, j)$ is the amount of time the itinerary
$\mc{U}(Z_i, Z_j)$ spends inside of sub-itineraries
$\mc{U}(Z_n, Z_{n+1})$ whose length is at least $L$, and $t^-_L(i, j)$
is the amount of time it spends inside of itineraries with length
bounded strictly above by $L$.

It follows from \eqref{eq:udist_additive} that for any $L$ and any
$i,j$ we have
\begin{equation}
  \label{eq:short_long_sum}
  t_L^+(i, j) + t_L^-(i, j) = \udist(Z_i, Z_j).
\end{equation}

We also define the quantity
\[
  r_L(i,j) := t_L^-(i,j) / \udist(Z_i, Z_j).
\]
By definition, we have $r_L(\mathbf{Z}) = r_L(1, N)$.

\begin{definition} \label{defn:long threshold prelim}
  We now let $A_0, L_0$ be the constants from
  \Cref{prop:subitinerary_linear_growth}. We then fix constants
  $D_0, M_0$ as in \Cref{lem:wellsep_walls_I}, where the constant $\lambda$
  in the lemma is chosen to be $L_0$. Then, we define
  \[
    L_1 = \max\{M_0/A_0, D_0/A_0, L_0\}
  \]
  and then choose constants $D_1, M_1$ as in
  \Cref{lem:wellsep_walls_I} again, but this time taking $\lambda = L_1$.
\end{definition}

\begin{lem}
  Given $A, B, C > 0$, there exists $L > 0, \tau_1 \in (0,1)$ so that
  for any $i < j$, if $r_L(i,j) < \tau_1$, then
  \[
    A \cdot t_L^+(i,j) - B\cdot t_L^-(i,j) > C.
  \]
\end{lem}
\begin{proof}
  Since $t_L^-(i,j) = \udist(Z_i, Z_j)r_L(i,j)$ and
  $t_L^+(i,j) = \udist(Z_i, Z_j)(1 - r_L(i,j))$, the left-hand side of
  the inequality above is given by
  \[
    \udist(Z_i, Z_j) \cdot (A (1 - r_L(i,j)) - B r_L(i,j)).
  \]
  If $r_L(i,j) < \tau_1$ for some $\tau_1 \in (0,1)$, the above expression is bounded
  below by
  \[
    \udist(Z_i, Z_j)(A(1 - \tau_1) - B\tau_1).
  \]
  As $\tau_1 \to 0$, the quantity $(A(1 - \tau_1) - B\tau_1)$
  approaches $A$. In particular, if $\tau_1 < \frac{A}{2(A + B)}$, then the
  above expression is at least $\udist(Z_i, Z_j) \cdot \frac{A}{2}$. Finally,
  the fact that $r_L(i, j) < \tau_1 < 1$ implies that
  $\udist(Z_i, Z_j) > L$ since there has to be at least one long sub-itinerary, so the lemma follows if we take
  $L \geq 2C / A$.
\end{proof}

\begin{definition}
  \label{defn:long_length_threshold}
  We define
  \[
    B_0 := \max\left\{D_1, D_0, K \right\},
  \]
  where $K$ is the constant from \Cref{lem:wall_additive_error}, and
  then use the previous lemma to fix constants $L, \tau_1$ so that
  \begin{equation}
    \label{eq:uniform_root_ineq}
    A_0 \cdot t_L^+(i,j) - B_0\cdot t_L^-(i,j) > \max\{M_0, M_1\}
  \end{equation}
  whenever $r_L(i,j) < \tau_1$. We also choose $L$ large enough so
  that $L > L_1 \ge L_0$. The length $L$ will be our threshold for
  ``long'' sub-itineraries.
\end{definition}

\subsubsection{Combining all remaining sub-itineraries}

\begin{prop}
  \label{prop:final_gluing}
  For any $1 \le i < j \le N$, if $r_L(i,j) < \tau_1$, then
  \[
    \iroot(Z_i, Z_j) \ge A_0 \cdot t_L^+(i, j) - B_0 \cdot t_L^-(i,
    j).
  \]
\end{prop}
\begin{proof}
  We will induct on $j - i$. In the base case $j = i + 1$, then since
  $r_L(i,j) < \tau_1 < 1$ we must have $\udist(Z_i, Z_{i+1}) \ge
  L$. Since $L \ge L_0$, \Cref{prop:subitinerary_linear_growth} tells
  us that $\iroot(Z_i, Z_j) \ge A_0\udist(Z_i, Z_j)$ and since
  $\udist(Z_i, Z_j) = t_L^+(i,j)$ in this case, we are done.

  So now suppose that $j - i > 1$. First, observe that if
  $\udist(Z_i, Z_{i + 1}) < L$, then
  $r_L(i+1, j) < r_L(i, j) < \tau_1$. So, by induction we have
  \[
    \iroot(Z_{i+1}, Z_j) \ge A_0 \cdot t_L^+(i + 1, j) - B_0 \cdot
    t_L^-(i + 1, j).
  \]
  In this case, we also know that $t_L^+(i, j) = t_L^+(i + 1, j)$. So,
  after applying \Cref{lem:wall_additive_error}, we have
  \begin{align*}
    \iroot(Z_i, Z_j) &\ge A_0 \cdot t_L^+(i, j) - B_0 \cdot t_L^-(i +
                       1, j) - K \cdot \udist(Z_i, Z_{i+1})\\
                     &\ge A_0 \cdot t_L^+(i, j) - B_0 \cdot t_L^-(i +
                       1, j) - B_0 \cdot \udist(Z_i, Z_{i+1}).
  \end{align*}
  But we also know that
  $t_L^-(i,j) = \udist(Z_i, Z_{i+1}) + t_L^-(i + 1, j)$ in this case,
  so this proves the desired inequality. A similar argument also shows
  that the inequality holds in the case where
  $\udist(Z_{j-1}, Z_j) < L$.

  So, we may now assume that we have $\udist(Z_i, Z_{i+1}) \ge L$ and
  $\udist(Z_{j-1}, Z_j) \ge L$. Our strategy in this case is to find a
  suitable pair of walls $Z_{n-1} < Z_n$ to break the itinerary
  $\mc{U}(Z_i, Z_j)$ into three pieces: a ``short'' piece in the
  middle, and two ``mostly long'' pieces on either side. Then we will
  apply induction and one of our additivity lemmas for well-separated
  walls (\Cref{lem:wellsep_walls_I}).

  To find a suitable ``short'' sub-itinerary in the middle of
  $\udist(Z_i, Z_j)$, we prove the following claim:
  \begin{claim}
    For some $n$ with $i + 1 < n < j$, we have
    \begin{equation}
      \label{eq:pivot_itinerary}
      r_L(i, n - 1) < \tau_1, \qquad r_L(n, j) < \tau_1,
    \end{equation}
    and $\udist(Z_{n-1}, Z_n) < L$.
  \end{claim}
  To find such an $n$, first note that since
  $\udist(Z_{j-1}, Z_j) \ge L > L_0$,
  \Cref{prop:no_adjacent_long_itineraries} implies that
  $\udist(Z_{j-2}, Z_{j-1}) < L_0 < L$ and thus $j - 2 > i$. Then, if
  $r_L(i, j-2) < \tau_1$, we can satisfy the claim by taking
  $n = j-1$, since $r_L(j-1, j) = 0$. So, we can assume that
  $r_L(i, j - 2) \ge \tau_1$, and let $n$ be the minimal index such
  that $r_L(i, n) \ge \tau_1$.

  Now, we have
  \begin{align*}
    r_L(i,j)
    &= \frac{1}{\udist(Z_i, Z_j)} (t_L^-(i, n) + t_L^-(n, j))\\
    &= \frac{\udist(Z_i, Z_n)}{\udist(Z_i, Z_j)} r_L(i, n) +
      \frac{\udist(Z_n, Z_j)}{\udist(Z_i, Z_j)}r_L(n, j).
  \end{align*}
  From \eqref{eq:udist_additive}, we also know that
  \[
    \frac{\udist(Z_i, Z_n)}{\udist(Z_i, Z_j)} + \frac{\udist(Z_n,
      Z_j)}{\udist(Z_i, Z_j)} = 1.
  \]
  So, since $r_L(i,n) \ge \tau_1$ we must have $r_L(n, j) < \tau_1$,
  and by minimality of $n$ we also know that $r_L(i, n - 1) <
  \tau_1$. But in particular we also know
  $r_L(i, n-1) < r_L(i, n)$, implying $\udist(Z_{n-1}, Z_n) < L$.

  This proves the claim, so to finish the inductive step, we need to
  consider two cases:
  \begin{description}
  \item[Case 1: $\udist(Z_{n-1}, Z_n) < L_1$] In this case, we
    consider the sub-itineraries $\mc{U}(Z_i, Z_{n-1})$ and
    $\mc{U}(Z_n, Z_j)$. Since $r_L(i, n-1) < \tau_1$ and
    $r_L(n, j) < \tau_1$, we may assume inductively that
    \begin{align}
      \label{eq:inductive_root_ineq_left}
      \iroot(i, n-1) &\ge A_0t_L^+(i, n-1) - B_0t_L^-(i,n-1),\\
    \label{eq:inductive_root_ineq_right}
      \iroot(n, j) &\ge A_0t_L^+(n, j) - B_0t_L^-(n, j).
    \end{align}
    It then follows directly from \Cref{defn:long_length_threshold} and
    \eqref{eq:inductive_root_ineq_left},
    \eqref{eq:inductive_root_ineq_right} above that
    \begin{align*}
      \iroot(i, n-1) > M_1, \qquad \iroot(n, j) > M_1.
    \end{align*}
    So, since $\udist(Z_{n-1}, Z_n) < L_1$, we can apply
    \Cref{lem:wellsep_walls_I} to the walls $Z_i, Z_{n-1}, Z_n, Z_j$ to
    obtain
    \begin{align*}
      \iroot(Z_i, Z_j)
      &> \iroot(Z_i, Z_{n-1}) + \iroot(Z_n, Z_j) - D_1\\
      &\ge A_0(t_L^+(i, n-1) + t_L^+(n, j)) - B_0(t_L^-(i, n-1) +
        t_L^-(n, j)) - D_1\\
      &\ge A_0t_L^+(i,j) - B_0(t_L^-(i, n-1) + t_L^-(n, j) + 1),
    \end{align*}
    where the last inequality holds because we have defined $B_0 \geq
    D_1$. But then since $1 \le \udist(Z_{n-1}, Z_n) < L$, we know that
    \[
      t_L^-(i,j) = t_L^-(i, n-1) + t_L^-(n, j) + \udist(Z_{n-1}, Z_n)
      \ge t_L^-(i,n-1) + t_L^-(n, j) + 1,
    \]
    which proves the desired inequality in this case.
  \item[Case 2: $\udist(Z_{n-1}, Z_n) \ge L_1$] In this case, since
    $L_1 \ge L_0$, \Cref{prop:no_adjacent_long_itineraries} implies
    that the two sub-itineraries $\mc{U}(Z_{n-2}, Z_{n-1})$,
    $\mc{U}(Z_n, Z_{n+1})$ to the left and right of
    $\mc{U}(Z_{n-1}, Z_n)$ both have length less than $L_0$.
    Since we know $r_L(i, n-1) < \tau_1 < 1$ and
    $r_L(n, j) < \tau_1 < 1$, this also tells us that $i < n - 2$ and
    $n + 1 < j$, and that
    \[
      r_L(i, n-2) < \tau_1, \qquad r_L(n + 1, j) < \tau_1.
    \]
    We then see directly from \Cref{defn:long_length_threshold}
    together with the induction hypothesis that
    $\iroot(Z_i, Z_{n-2}) > M_0$ and $\iroot(Z_{n+1}, Z_j) > M_0$. In
    addition, since
    $\udist(Z_{n-1}, Z_n) \ge L_1 = \max\{L_0, M_0/A_0, D_0/A_0\}$,
    \Cref{prop:subitinerary_linear_growth} implies that
    $\iroot(Z_{n-1}, Z_n) \ge \max\{M_0, D_0\}$.
    
    We now apply \Cref{lem:wellsep_walls_I} twice. First, we apply the
    lemma to the walls $Z_i < Z_{n-2} < Z_{n-1} < Z_n$, which gives the
    bound
    \begin{equation}
      \label{eq:case2_bound_1}
      \iroot(Z_i, Z_n) \ge \iroot(Z_i, Z_{n-2}) + \iroot(Z_{n-1}, Z_n) -
      D_0.
    \end{equation}
    Since $\iroot(Z_{n-1}, Z_n) \ge D_0$ we see that
    $\iroot(Z_i, Z_n) > M_0$, which means we can then apply
    \Cref{lem:wellsep_walls_I} to the walls $Z_i < Z_n < Z_{n+1} < Z_j$
    to obtain
    \begin{equation}
      \label{eq:case2_bound_2}
      \iroot(Z_i, Z_j) \ge \iroot(Z_i, Z_n) + \iroot(Z_{n+1}, Z_j) -
      D_0.
    \end{equation}
    Putting \eqref{eq:case2_bound_1} and \eqref{eq:case2_bound_2}
    together we see that
    \[
      \iroot(Z_i, Z_j) \ge \iroot(Z_i, Z_{n-2}) + \iroot(Z_{n-1}, Z_n) +
      \iroot(Z_{n+1}, Z_j) - 2D_0.
    \]
    Since each of $\udist(Z_{n-2}, Z_{n-1})$, $\udist(Z_{n-1}, Z_n)$,
    and $\udist(Z_n, Z_{n+1})$ is less than $L$, we know that
    $t_L^+(i, j) = t_L^+(i, n-2) + t_L^+(n+1, j)$. Thus, after applying
    induction to the terms $\iroot(Z_i, Z_{n-2})$ and
    $\iroot(Z_{n+1}, Z_j)$ in the inequality above, and discarding the
    (nonnegative) $\iroot(Z_{n-1}, Z_n)$ term, we obtain
    \begin{align*}
      \iroot(Z_i, Z_j)
      &\ge A_0t_L^+(i, j) - B_0(t_L^-(i, n - 2) +
        t_L^-(n+1, j)) - 2D_0\\
      &\ge A_0t_L^+(i, j) - B_0(t_L^-(i, n - 2) +
        t_L^-(n+1, j) + 2).
    \end{align*}
    For the last line we apply the fact that $B_0 \ge D_0$. Finally,
    since $t_L^-(n - 2, n - 1) \ge 1$ and $t_L^-(n, n + 1) \ge 1$, we
    get
    \begin{align*}
      t_L^-(i,j) &= t_L^-(i,n-2) + t_L^-(n-2, n - 1) + t_L^-(n - 1, n) +
                   t_L^-(n, n + 1) + t_L^-(n+1, j)\\
                 &> t_L^-(i, n - 2) + t_L^-(n + 1, j) + 2,
    \end{align*}
    and we obtain the desired inequality in this case as well. \qedhere
  \end{description}
\end{proof}

Finally we obtain the estimate we originally wanted.
\begin{proof}[Proof of \Cref{prop:long_intervals_bound}]
  We set
  \[
    \tau = \min\left\{\tau_1, \frac{A_0}{2(A_0 + B_0)}\right\},
  \]
  where $A_0, B_0$ are the constants from \Cref{prop:final_gluing}.

  To simplify notation we write $r = r_L(\mathbf{Z})$. Since all of
  our modifications to $\mathbf{Z}$ have only decreased $r$, if we had
  $r < \tau$ before our modifications, this is still true for our
  current $\mathbf{Z}$.

  By definition, we have $r\cdot \udist(Z_1, Z_N) = t_L^-(1, N)$, and
  \eqref{eq:short_long_sum} implies that
  \[
    t_L^+(1, N) = (1 - r)\cdot \udist(Z_1, Z_N).
  \]
  Then, since
  $r < \tau \le \tau_1$, we can use \Cref{prop:final_gluing} to
  obtain:
  \begin{align*}
    \iroot(Z_1, Z_N)
    &\ge \left(A_0(1 - r) - B_0 r \right) \cdot \udist(Z_1, Z_N)\\
    &\ge \left(A_0(1 - \tau) - B_0 \tau \right) \cdot \udist(Z_1,
      Z_N)\\
    &\ge \frac{A_0}{2} \udist(Z_1, Z_N).
  \end{align*}
  Our construction ensures that $Z_1 = W_1$ and $Z_N$ is always the
  maximal wall in $\mathbf{W}$. From \Cref{cor:disjoint_walls} and the
  definition of $\mc{U}$, we know there is a uniform $R > 0$ so that
  $\elt(\mc{U}) = \eta\, \elt(Z_1, Z_N)\, \eta'$, for some
  $\eta, \eta' \in C$ with $|\eta|, |\eta'| < R$.  Thus the desired
  estimate follows from \Cref{lem:root_triangle_inequality}.
\end{proof}

\appendix

\section{Failure of strong nesting for half-cones}
\label{sec:appendix}

The purpose of this appendix is to prove the following two claims:

\begin{prop}
  \label{prop:vindomain_nesting_fail}
  There exists a right-angled Coxeter group $C$, a simplicial
  representation $\rho$ of $C$, and a pair of walls $W, W'$ in the
  Vinberg domain $\vindomain$ for $\rho$ which satisfies the following
  properties:
  \begin{enumerate}
  \item $\elt(W, W')$ does not lie in a proper standard subgroup of
    $C$ (equivalently, by \Cref{lem:disjoint_walls},
    $\overline{W} \cap \overline{W'} = \emptyset$);
  \item $\halfspace_+(W') \subset \halfspace_+(W)$;
  \item $\overline{\halfcone_+(W')}$ is \emph{not} contained in
    $\halfcone_+(W)$.
  \end{enumerate}
\end{prop}

\begin{prop}
  \label{prop:refdomain_nesting_fail}
  There exists a right-angled Coxeter group $C$ and reflections
  $R, R'$ in $C$ such that, for any simplicial representation
  $\rho\colon C \to \SLpm(|S|, \R)$ with fully nondegenerate Cartan matrix,
  and \emph{any} reflection domain $\Omega$ for $\rho$, if $W, W'$ are
  the walls in $\Omega$ preserved by $R, R'$, then:
  \begin{enumerate}
  \item $\elt(W, W')$ does not lie in a proper standard subgroup of
    $C$ (so in particular
    $\overline{W} \cap \overline{W'} = \emptyset$);
  \item $\halfspace_+(W') \subset \halfspace_+(W)$;
  \item $\overline{\halfcone_+(W')}$ is \emph{not} contained in
    $\halfcone_+(W)$.
  \end{enumerate}
\end{prop}

Although the second proposition above implies the first, we will prove
these results one at a time, since the construction for
\Cref{prop:refdomain_nesting_fail} is a slightly more complicated
variation of the construction for \Cref{prop:vindomain_nesting_fail}.

Note that if \Cref{prop:vindomain_nesting_fail} were false, then the
proofs in \Cref{sec:main_thm} of this paper would considerably
simplify---in particular, the proof of \Cref{lem:wellsep_walls_II}
could be reduced to a direct application of
\Cref{lem:wellsep_walls_I}. \Cref{prop:refdomain_nesting_fail} tells
us that we cannot resolve the problem simply by replacing $\vindomain$
in \Cref{sec:main_thm} with some other carefully chosen reflection
domain.

\begin{remark}
  If $C$ is a hyperbolic Coxeter group, it follows from
  \cite[Corollary 1.11]{DGKLM} that there is a simplicial
  representation $\rho\colon C \to \SLpm(V)$ and a reflection domain
  $\Omega$ for $\rho$ so that half-cones over any two walls $W, W'$ in
  $\Omega$ with disjoint closures will strongly nest. It seems likely
  that this holds even for some examples of simplicial representations
  of non-hyperbolic Coxeter groups, but we do not pursue this here.
\end{remark}

\subsection{Proof of \Cref{prop:vindomain_nesting_fail}}

Consider the right-angled Coxeter group $C$ with generating set
$S = \langle a, b, c, d, e \rangle$, and nerve given in
\Cref{fig:coxeter_nerve_1} below (recall that there is an edge between
two vertices in the nerve precisely when the corresponding generators
commute):

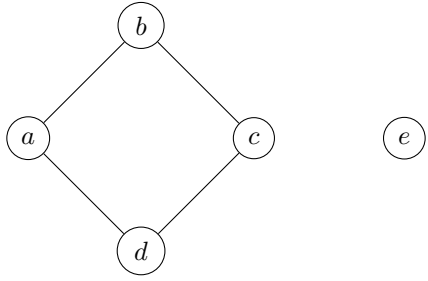
\begin{figure}[H]
  \begin{center}
    \begin{tikzpicture}
      \node[draw,circle] (a) at (-1.5, 0) {$a$};
      \node[draw,circle] (b) at (0, 1.5) {$b$};
      \node[draw,circle] (c) at (1.5, 0) {$c$};
      \node[draw,circle] (d) at (0, -1.5) {$d$};
      \node[draw,circle] (e) at (3.5, 0) {$e$};
      
      \draw (a) -- (b) {};
      \draw (b) -- (c) {};
      \draw (c) -- (d) {};
      \draw (d) -- (a) {};
    \end{tikzpicture}
  \end{center}
  \caption{The nerve of the right-angled Coxeter group $C$ in
    \Cref{prop:vindomain_nesting_fail}.}
  \label{fig:coxeter_nerve_1}
\end{figure}

That is, $C$ is isomorphic to the group $(\Z/2 * \Z/2)^2 * (\Z/2)$,
where the $\Z/2 * \Z/2$ factors are generated by the pairs $a, c$ and
$b, d$, and $e$ generates the remaining $\Z/2$ free factor. Observe
that $C$ is a minimal example of an irreducible non-hyperbolic
right-angled Coxeter group: a theorem of Gromov (see \cite[Section
4.2.C]{gromov} or \cite[Chapter 12]{Davis}) implies that a
right-angled Coxeter group fails to be hyperbolic precisely when the
``empty square'' on the vertices $a, b, c, d$ appears as a full
subgraph of the nerve. (The general version of this theorem for
Coxeter groups which are not necessarily right-angled is due to
Moussong \cite{Moussong}.)

Let $V = \R^5$. We choose a nonsingular Cartan matrix for $C$ and a
simplicial representation determined by this Cartan matrix. We
suppress the representation from the notation, and just let $C$ act
directly on $\P(V)$. For each $s \in \{a, b, c, d, e\}$, we let
$H_s \subset \P(V)$ denote the (projective) reflection hyperplane for
the reflection $s$, and we let $F_s$ denote the closed face of the Tits simplex $\Delta$ fixed by $s$. Further, for any
subset $S' \subset S$, we let $H_{S'}$ denote $\bigcap_{s \in S'} H_s$, and $F_{S'}$ denote $\bigcap_{s \in S'}F_s$.

Consider the 2-dimensional projective subspace
$H_{\{a,c\}} = H_a \cap H_c$. This subspace contains the closed face
$F_{\{a,c\}}$ of the fundamental simplex $\Delta$ fixed pointwise by
the standard subgroup $C(a,c)$. Since $C(a,c)$ is infinite,
$F_{\{a,c\}}$ must be contained in the boundary of $\vindomain$ and
hence so are all of its translates under the action of $C$.

Now, since the subgroup $C(b, d)$ centralizes $C(a, c)$, it preserves
the subspace $H_{\{a,c\}}$, and in fact $b$ and $d$ act on this
subspace by projective reflections fixing the lines
$H_b \cap H_{\{a,c\}}$ and $H_d \cap H_{\{a,c\}}$. The relative
interior of the orbit $C(b,d) \cdot F_{\{a,c\}}$ is an infinite-sided
convex polygon $P$ in $H_{\{a,c\}}$, which must be a subset of
$\dee \vindomain$ (see \Cref{fig:5dim_halfcone}).
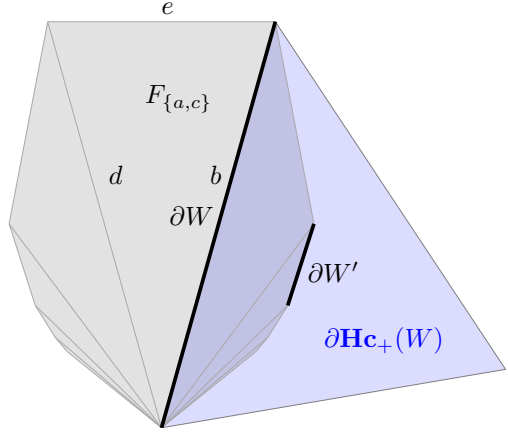
\begin{figure}[h]
  \begin{center}
    \def\svgwidth{.5\textwidth}
\begingroup%
  \makeatletter%
  \providecommand\color[2][]{%
    \errmessage{(Inkscape) Color is used for the text in Inkscape, but the package 'color.sty' is not loaded}%
    \renewcommand\color[2][]{}%
  }%
  \providecommand\transparent[1]{%
    \errmessage{(Inkscape) Transparency is used (non-zero) for the text in Inkscape, but the package 'transparent.sty' is not loaded}%
    \renewcommand\transparent[1]{}%
  }%
  \providecommand\rotatebox[2]{#2}%
  \newcommand*\fsize{\dimexpr\f@size pt\relax}%
  \newcommand*\lineheight[1]{\fontsize{\fsize}{#1\fsize}\selectfont}%
  \ifx\svgwidth\undefined%
    \setlength{\unitlength}{568.79998779bp}%
    \ifx\svgscale\undefined%
      \relax%
    \else%
      \setlength{\unitlength}{\unitlength * \real{\svgscale}}%
    \fi%
  \else%
    \setlength{\unitlength}{\svgwidth}%
  \fi%
  \global\let\svgwidth\undefined%
  \global\let\svgscale\undefined%
  \makeatother%
  \begin{picture}(1,0.82770471)%
    \lineheight{1}%
    \setlength\tabcolsep{0pt}%
    \put(0,0){\includegraphics[width=\unitlength,page=1]{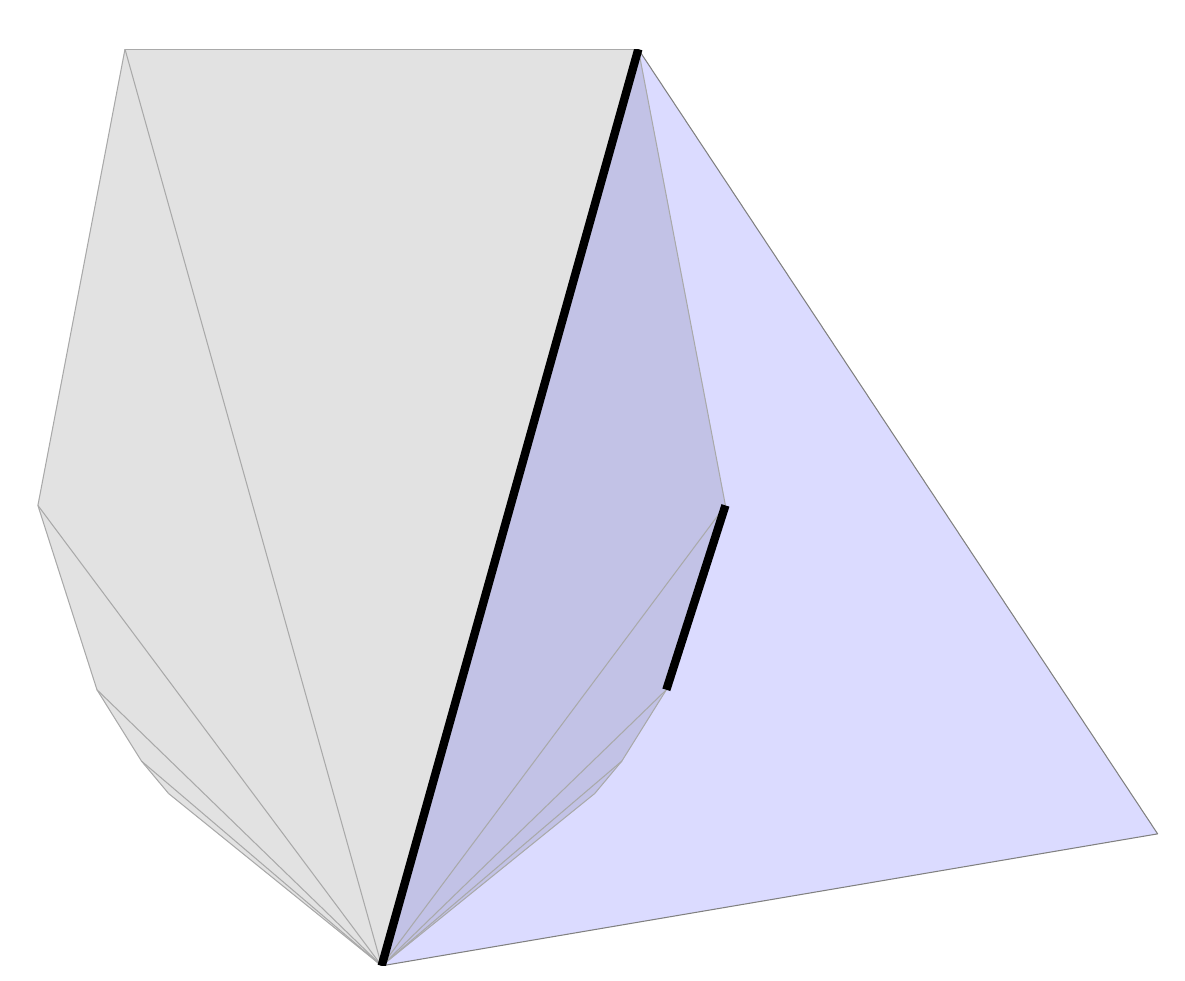}}%
    \put(0.32208158,0.80168777){\makebox(0,0)[lt]{\lineheight{1.25}\smash{\begin{tabular}[t]{l}$e$\end{tabular}}}}%
    \put(0.41490859,0.47679325){\makebox(0,0)[lt]{\lineheight{1.25}\smash{\begin{tabular}[t]{l}$b$\end{tabular}}}}%
    \put(0.22151899,0.47679325){\makebox(0,0)[lt]{\lineheight{1.25}\smash{\begin{tabular}[t]{l}$d$\end{tabular}}}}%
    \put(0.29113926,0.63150493){\makebox(0,0)[lt]{\lineheight{1.25}\smash{\begin{tabular}[t]{l}$F_{\{a,c\}}$\end{tabular}}}}%
    \put(0.63150496,0.16736987){\color[rgb]{0,0,1}\makebox(0,0)[lt]{\lineheight{1.25}\smash{\begin{tabular}[t]{l}$\partial \halfcone_+(W)$\end{tabular}}}}%
    \put(0.33755275,0.39943742){\makebox(0,0)[lt]{\lineheight{1.25}\smash{\begin{tabular}[t]{l}$\partial W$\end{tabular}}}}%
    \put(0.60056261,0.29113926){\makebox(0,0)[lt]{\lineheight{1.25}\smash{\begin{tabular}[t]{l}$\partial W'$\end{tabular}}}}%
  \end{picture}%
\endgroup%

  \end{center}
  \caption{The polygon $P$ giving a face of the Vinberg domain
    $\vindomain$ for a simplicial representation of $C$. The polygon
    is tiled by copies of the triangle $F_{\{a,c\}}$; the walls of
    this tiling are translates of $H_{\{a,c\}} \cap \dee W(b)$ and
    $H_{\{a,c\}} \cap \dee W(d)$.  The shaded triangle has a vertex at
    the polar of $b$, so it lies in the boundary of the half-cone
    $\halfcone_+(W)$.}
  \label{fig:5dim_halfcone}
\end{figure}

On the other hand, since the generator $e$ does \emph{not} commute
with either $a$ or $c$, the intersection
$H_{\{a,c\}} \cap e H_{\{a,c\}}$ is the 1-dimensional subspace
$H_{\{a,c\}} \cap H_e$. This tells us that the closed polygon
$\overline{P}$ is exactly the intersection of $\overline{\vindomain}$
with $H_{\{a,c\}}$, i.e.\ it is a face of $\vindomain$.

Now, consider the geodesic word $w = bdeac$, and let $\mc{W}$ be an
$\vindomain$-itinerary traversing $w$, departing from the
identity. Since no pair of consecutive generators in $w$ commutes,
every pair of distinct walls in $\mc{W}$ is disjoint in
$\vindomain$. This means that $\mc{W}$ must be efficient. So, if
$W, W'$ are respectively the first and last walls in $\mc{W}$, we must
have $\elt(W, W') = bdeac$.

As $W \cap W' = \emptyset$ and $\mc{W}$ departs from the identity, we
know that $\halfspace_+(W)$ contains $\halfspace_+(W')$. By
\Cref{prop:wall_expression} we have $W' = bdea \cdot W(c)$, where
$W(c)$ is the reflection wall in $\vindomain$ for $c$. From this, it
follows that:

\begin{prop}
  The intersection $\overline{W'} \cap \overline{P}$ is given by
  $bd \cdot F_{\{a,c,e\}}$, where $F_{\{a,c,e\}}$ is the edge of
  $F_{\{a,c\}}$ fixed by $e$.
\end{prop}
\begin{proof}
  Since $\overline{W(c)} = H_c \cap \overline{\vindomain}$, and
  $\vindomain$ is $C$-invariant, we have
  $\overline{P} \cap \overline{W'} = \overline{P} \cap bdea \cdot
  H_c$. Since $P$ is invariant under $b,d$, this intersection is the
  same as $bd(\overline{P} \cap ea \cdot H_c)$. Then since $e$ does
  not commute with $c$ or $a$, we have
  $(ea \cdot H_c) \cap H_{\{a,c\}} = H_{\{a,c,e\}}$ and thus
  $\overline{P} \cap ea \cdot H_c = F_{\{a,c,e\}}$.
\end{proof}

We know that the polar of $b$ lies in the projective subspace
$H_{a,c}$ since $b$ commutes with both $a$ and $c$. So, we can use the
argument from \Cref{lem:halfspaces_in_halfcones} to see that the
boundary of the half-cone $\halfcone_+(W)$ contains the connected
component of $\overline{P} \minus \overline{W}$ which does \emph{not}
contain $F_{\{a,c\}}$. Thus, the boundary of $\halfcone_+(W)$ contains
$\overline{W'}\cap H_{a,c}$ and so the half-cones over these walls
cannot strongly nest.

\begin{remark}
  \label{rem:long_word_fail}
  For any given $k \ge 1$, we can also consider the word
  $w = (bd)^ke(ac)^k$, and an itinerary $\mc{W}$ traversing $w$
  departing from the identity. A nearly identical argument to the
  above shows that the initial and terminal walls $W, W'$ of $\mc{W}$
  also satisfy the conclusions of
  \Cref{prop:vindomain_nesting_fail}. This proves that, in
  \Cref{prop:vindomain_nesting_fail}, the group element $\elt(W, W')$
  cannot even be made to lie ``close'' to a proper standard subgroup:
  we cannot find a uniform constant $R$ so that
  $\elt(W, W') = \eta_1 \gamma \eta_2$ for $\eta_1, \eta_2$ satisfying
  $|\eta_i| < R$ and $\gamma$ lying in a proper standard subgroup of
  $C$.
\end{remark}

\subsubsection{Modifying the example}
\label{sec:modifying}

The argument above also shows that the corresponding half-cones in the
Vinberg domain $\dvindomain \subset \P((\R^5)^*)$ for the dual
representation $\rho^*$ (see \Cref{sec:dual_domains}) do not strongly
nest. By \Cref{lem:halfcone_duality}, the corresponding half-cones in
the dual $\dvindomain^*$ cannot strongly nest either. By
\Cref{prop:vinberg_maximal}, the reflection domain $\dvindomain^*$ is
contained inside of every reflection domain for $\rho$, so we denote
it $\mindomain$.

One could still hope to find some reflection domain $\Omega$ lying
between $\mindomain$ and $\vindomain$ where the half-cones over the
walls $W \cap \Omega$ and $W' \cap \Omega$ strongly nest. In fact, for
the example above, it is possible to find such a domain, by taking
$\Omega$ to be a small neighborhood of $\mindomain$ with respect to
the Hilbert metric on $\vindomain$. This strategy works because the
segments in $\partial \vindomain$ joining $\partial W$ to
$\partial W'$ are not contained in $\partial \mindomain$; see
\Cref{fig:vinface_polars}. But this can fail if the intersection
between $\partial \vindomain$ and $\partial \mindomain$ contains
large-dimensional faces, which is what occurs in the next
counterexample.

\begin{figure}[h]
  \begin{center}
    \def\svgwidth{.6\textwidth}
\begingroup%
  \makeatletter%
  \providecommand\color[2][]{%
    \errmessage{(Inkscape) Color is used for the text in Inkscape, but the package 'color.sty' is not loaded}%
    \renewcommand\color[2][]{}%
  }%
  \providecommand\transparent[1]{%
    \errmessage{(Inkscape) Transparency is used (non-zero) for the text in Inkscape, but the package 'transparent.sty' is not loaded}%
    \renewcommand\transparent[1]{}%
  }%
  \providecommand\rotatebox[2]{#2}%
  \newcommand*\fsize{\dimexpr\f@size pt\relax}%
  \newcommand*\lineheight[1]{\fontsize{\fsize}{#1\fsize}\selectfont}%
  \ifx\svgwidth\undefined%
    \setlength{\unitlength}{568.00604248bp}%
    \ifx\svgscale\undefined%
      \relax%
    \else%
      \setlength{\unitlength}{\unitlength * \real{\svgscale}}%
    \fi%
  \else%
    \setlength{\unitlength}{\svgwidth}%
  \fi%
  \global\let\svgwidth\undefined%
  \global\let\svgscale\undefined%
  \makeatother%
  \begin{picture}(1,0.58068053)%
    \lineheight{1}%
    \setlength\tabcolsep{0pt}%
    \put(0,0){\includegraphics[width=\unitlength,page=1]{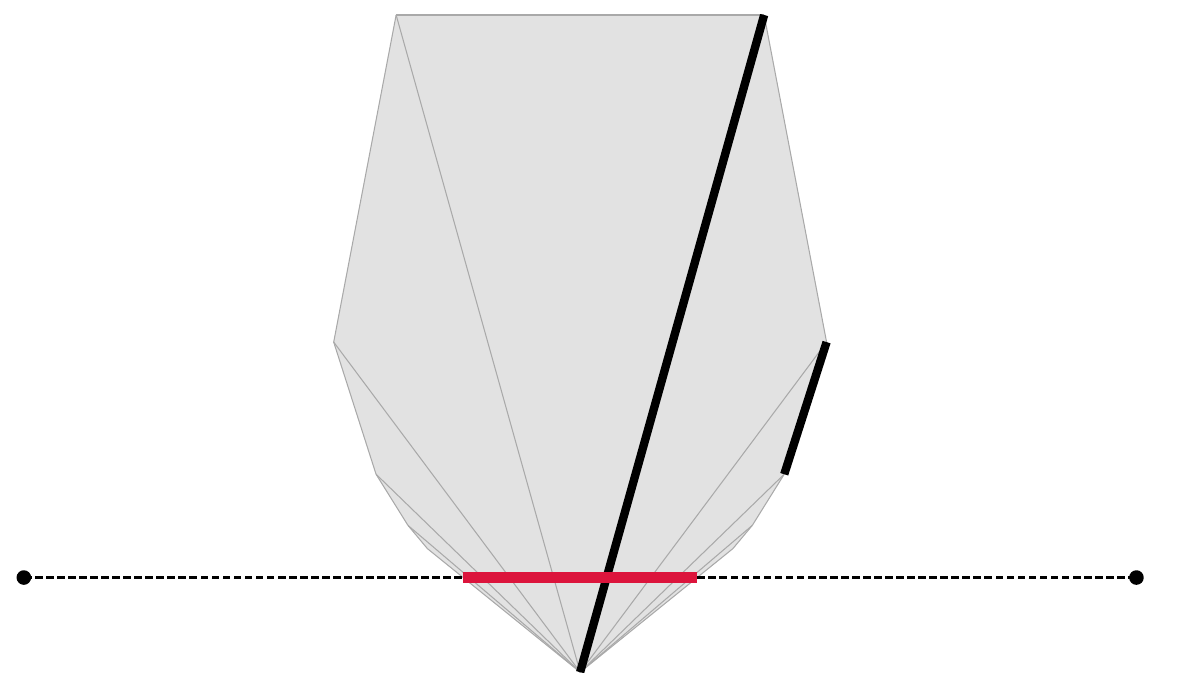}}%
    \put(0.49581188,0.29034026){\makebox(0,0)[lt]{\lineheight{1.25}\smash{\begin{tabular}[t]{l}$\partial W$\end{tabular}}}}%
    \put(0.69017692,0.21259426){\makebox(0,0)[lt]{\lineheight{1.25}\smash{\begin{tabular}[t]{l}$\partial W'$\end{tabular}}}}%
    \put(0.93452158,0.11818838){\makebox(0,0)[lt]{\lineheight{1.25}\smash{\begin{tabular}[t]{l}$v_b$\end{tabular}}}}%
    \put(0.01267592,0.11818838){\makebox(0,0)[lt]{\lineheight{1.25}\smash{\begin{tabular}[t]{l}$v_d$\end{tabular}}}}%
    \put(0.5791112,0.03488906){\color[rgb]{0.8627451,0.07843137,0.23529412}\makebox(0,0)[lt]{\lineheight{1.25}\smash{\begin{tabular}[t]{l}$\partial \mindomain \cap P$\end{tabular}}}}%
  \end{picture}%
\endgroup%

  \end{center}
  \caption{The boundary of the minimal reflection domain $\mindomain$
    does not contain a segment joining $W$ with $W'$. See
    \Cref{thm:minimal_domain} below for a construction of
    $\mindomain$.}
  \label{fig:vinface_polars}
\end{figure}
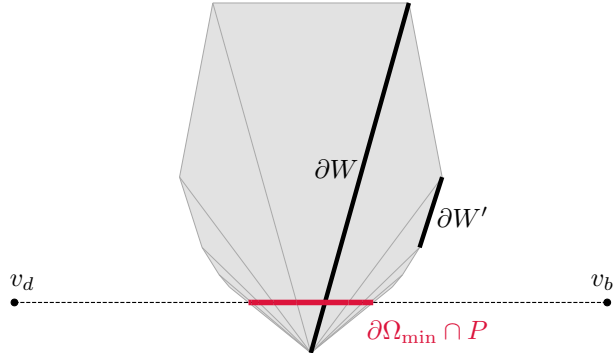

\subsection{Proof of \Cref{prop:refdomain_nesting_fail}}
\label{sec:6dim_counterexample}

We consider a right-angled Coxeter group $C$ whose generating set $S$
splits into three disjoint subsets $D = \{d_1, d_2\}$,
$T = \{t_1, t_2, t_3\}$, and $E = \{e\}$, with the following
relations:
\begin{itemize}
\item Each $t_i \in T$ commutes with each $d_j \in D$;
\item The generator $e$ commutes with $t_1$ and $t_3$.
\end{itemize}

There are no other relations among the generators, meaning the system
$(C, S)$ has the nerve depicted in \Cref{fig:coxeter_nerve_2} below. Observe that the subgroup $C(T)$ is an $(\infty, \infty, \infty)$
triangle group, which commutes with the infinite dihedral subgroup
$C(D)$.

\begin{figure}[H]
  \centering
  \begin{tikzpicture}
    \node[draw,circle] (d1) at (-1.5, 0) {$d_1$};
    \node[draw,circle] (d2) at (1.5, 0) {$d_2$};

    \node[draw,circle] (t1) at (0, 1.5) {$t_1$};
    \node[draw,circle] (t2) at (0,0) {$t_2$};
    \node[draw,circle] (t3) at (0, -1.5) {$t_3$};

    \node[draw,circle] (e) at (3.5, 0) {$e$};

    \draw (d1) -- (t1) {};
    \draw (d1) -- (t2) {};
    \draw (d1) -- (t3) {};

    \draw (d2) -- (t1) {};
    \draw (d2) -- (t2) {};
    \draw (d2) -- (t3) {};

    \draw (t1) -- (e) {};
    \draw (t3) -- (e) {};
  \end{tikzpicture}
  \caption{The nerve of the right-angled Coxeter group $C$ in
    \Cref{prop:refdomain_nesting_fail}.}
  \label{fig:coxeter_nerve_2}
\end{figure}
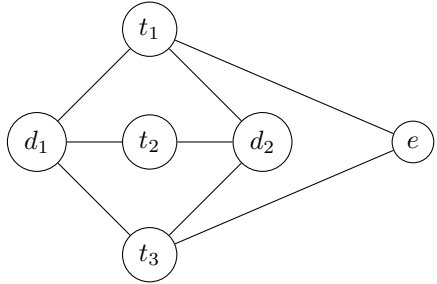

As for the previous example, we let $V = \R^6$, and fix an arbitrary
simplicial representation $C \to \SLpm(6, \R)$ with fully
nondegenerate Cartan matrix. We again omit the representation from the
notation and allow $C$ to act directly on $V$ and $\P(V)$.

Consider the 3-dimensional projective subspace $H_D \subset \P(V)$,
containing the closed face $F_D$ of $\Delta$. This face is a
tetrahedron whose faces span fixed subspaces for the reflections in
$T \cup E$. As every point in $F_D$ has infinite stabilizer, it lies
in the boundary of the Vinberg domain $\vindomain$.

Since the centralizer of the subgroup $C(D)$ is precisely $C(T)$, the
interior of the orbit $C(T) \cdot F_D$ is a subset of $H_D$. In fact,
one may apply the more general form of \Cref{thm:vinberg_action} given
in \cite{Vinberg1971} to see that $C(T) \cdot F_D$ is an
infinite-sided convex polytope $P \subset H_D$, whose closure $\overline{P}$
is a face of the Vinberg domain tiled by copies of $F_D$. See
\Cref{fig:polytope}.

\begin{figure}[H]
  \centering
  \def\svgwidth{.8\textwidth}
  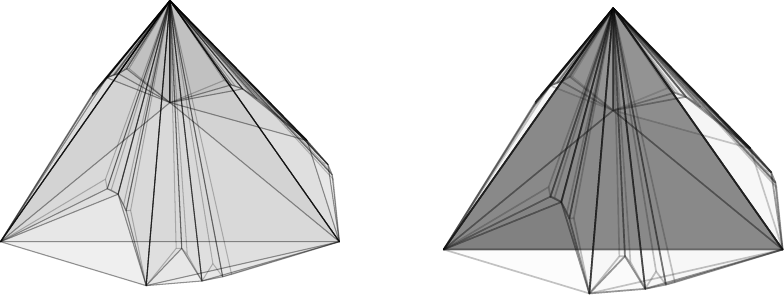
  \caption{The convex polytope $P$ (left and right). On the right, the
    tetrahedral fundamental domain $F_D$ has also been
    highlighted. The reflections $t_1, t_2, t_3$ in the three faces of
    $F_D$ meeting at the topmost vertex preserve the polytope $P$,
    while the reflection $e$ in the bottom face $F_{D \cup E}$ does
    not.}
  \label{fig:polytope}
\end{figure}

It is possible to argue as in the proof of
\Cref{prop:vindomain_nesting_fail} to see that there is a pair of
walls $W, W'$ in $\vindomain$ with disjoint closures, such that the
boundary of $\halfcone_+(W)$ contains a component of
$P \minus \overline{W}$ whose closure intersects $\overline{W'}$. We
want to see that something similar occurs not just for walls in the
Vinberg domain, but in \emph{any} reflection domain for this
simplicial representation.

For this, we consider the unique minimal reflection domain
$\mindomain$ for the simplicial representation $\rho$, whose existence
is guaranteed by the discussion in Section~\ref{sec:modifying}.  We can
get a precise description of $\mindomain$ in our situation using
a theorem of Danciger-Gu\'eritaud-Kassel-Lee-Marquis.

\begin{definition}
  Let $(C, S)$ be an infinite irreducible right-angled Coxeter system,
  and suppose $C$ acts via a simplicial representation on
  $V = \R^{|S|}$ with nonsingular Cartan matrix. For any subset
  $S' \subseteq S$, we let
  $\tilde{\Sigma}_{S'} \subset \overline{\tilde{\Delta}}$ denote the
  set
  \[
    \tilde{\Sigma}_{S'} = \left\{x = \sum_{t \in S'}\lambda_tv_t :
      \alpha_s(x) \le 0 \textrm{ and } \lambda_t \ge 0 \quad \forall s
      \in S, t \in S' \right\}.
  \]
  We let $\Sigma_{S'}$ denote the projectivization of
  $\tilde{\Sigma}_{S'}$ in $\P(V)$, and write
  $\tilde{\Sigma} = \tilde{\Sigma}_S$ and $\Sigma = \Sigma_S$.
\end{definition}

\begin{thm}[{See \cite[Theorem 5.2]{DGKLM}}]
  \label{thm:minimal_domain}
  Let $(C, S)$ be an infinite irreducible right-angled Coxeter group
  with $|S| > 2$, acting on $V = \R^{|S|}$ by a simplicial
  representation with nonsingular Cartan matrix. Then the set
  $\overline{\mindomain} = \overline{\bigcup_{\gamma \in C}
    \gamma\Sigma}$ is the closure of the unique reflection domain
  $\mindomain \subset \vindomain$ which is contained in every
  reflection domain for $\rho$.
\end{thm}

\begin{remark}
  The version of \Cref{thm:minimal_domain} proved in \cite{DGKLM}
  holds under considerably weaker hypotheses than what we have stated
  here. In particular, the result in \cite{DGKLM} can be applied to
  non-right-angled Coxeter groups and non-simplicial representations.
\end{remark}

We now consider the intersection of the sets $\Sigma$ and
$\overline{\mindomain}$ with the projective subspace $H_D$, for the
specific right-angled Coxeter group $C$ we described above. Observe
that the subgroup $C(T)$ acts on the subspace
$V_T = \spn\{v_{t_1}, v_{t_2}, v_{t_3}\}$, via a simplicial
representation whose Cartan matrix is a (fully nondegenerate)
principal submatrix of our original Cartan matrix. Thus
$\Sigma_T \subset \Sigma \cap \P(V_T)$ is a hexagon, with alternating
sides contained in the projective lines $H_{t_i} \cap \P(V_T)$ (see
\Cref{fig:iiitri}).

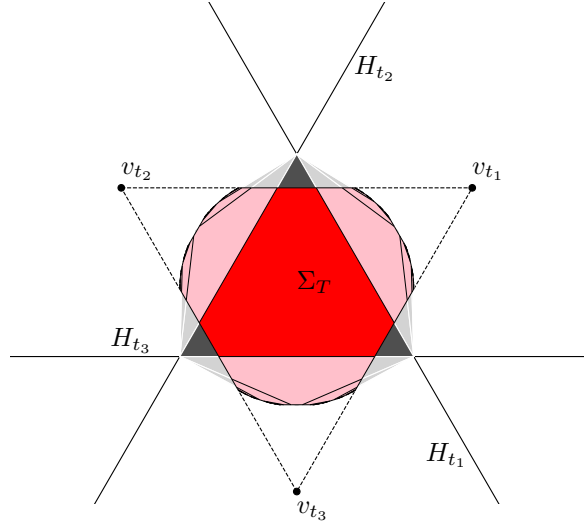
\begin{figure}[h]
  \centering
  \def\svgwidth{.6\textwidth}
\begingroup%
  \makeatletter%
  \providecommand\color[2][]{%
    \errmessage{(Inkscape) Color is used for the text in Inkscape, but the package 'color.sty' is not loaded}%
    \renewcommand\color[2][]{}%
  }%
  \providecommand\transparent[1]{%
    \errmessage{(Inkscape) Transparency is used (non-zero) for the text in Inkscape, but the package 'transparent.sty' is not loaded}%
    \renewcommand\transparent[1]{}%
  }%
  \providecommand\rotatebox[2]{#2}%
  \newcommand*\fsize{\dimexpr\f@size pt\relax}%
  \newcommand*\lineheight[1]{\fontsize{\fsize}{#1\fsize}\selectfont}%
  \ifx\svgwidth\undefined%
    \setlength{\unitlength}{576bp}%
    \ifx\svgscale\undefined%
      \relax%
    \else%
      \setlength{\unitlength}{\unitlength * \real{\svgscale}}%
    \fi%
  \else%
    \setlength{\unitlength}{\svgwidth}%
  \fi%
  \global\let\svgwidth\undefined%
  \global\let\svgscale\undefined%
  \makeatother%
  \begin{picture}(1,1)%
    \lineheight{1}%
    \setlength\tabcolsep{0pt}%
    \put(0,0){\includegraphics[width=\unitlength,page=1]{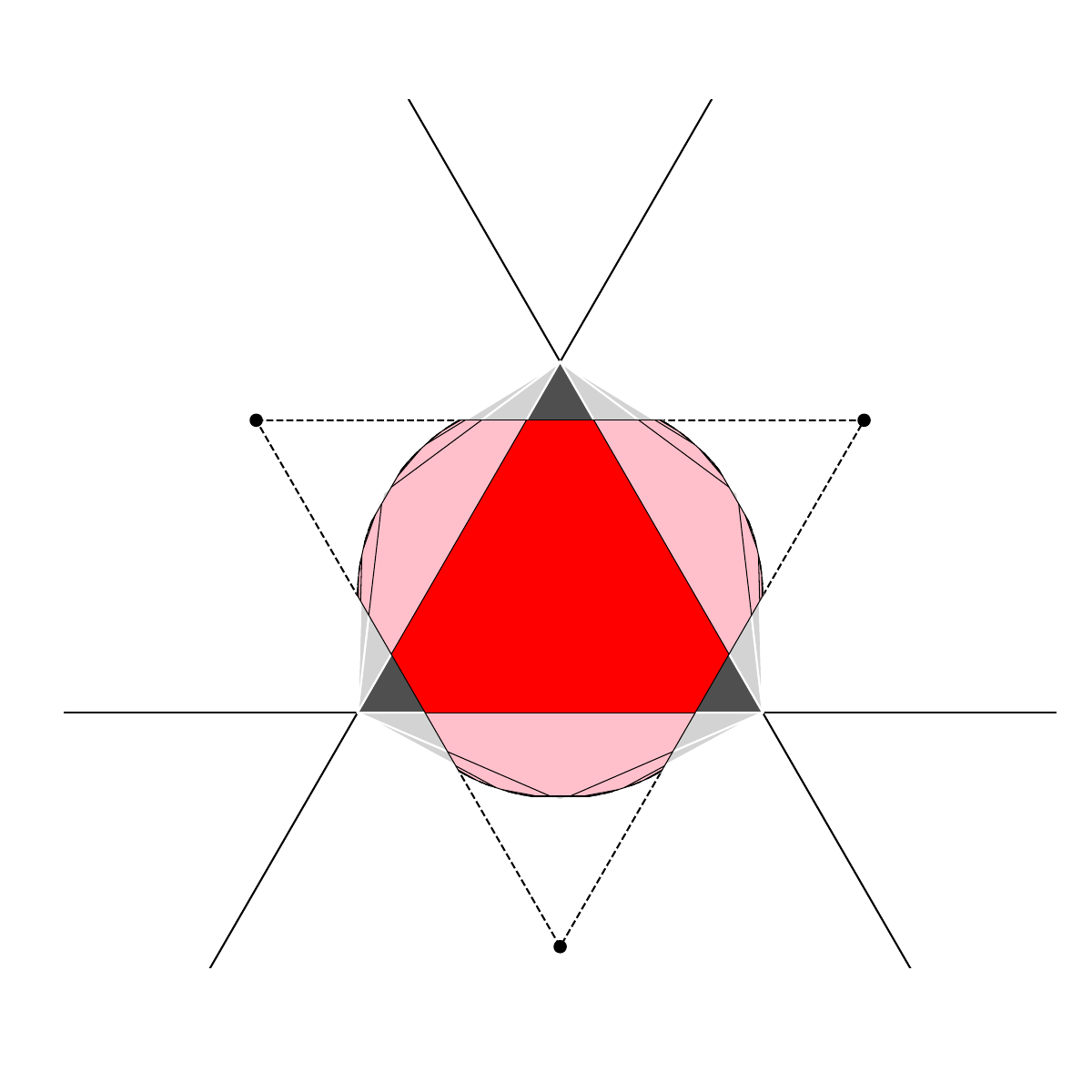}}%
    \put(0.51296658,0.45447185){\makebox(0,0)[lt]{\lineheight{1.25}\smash{\begin{tabular}[t]{l}$\Sigma_T$\end{tabular}}}}%
    \put(0.23464145,0.63788809){\makebox(0,0)[lt]{\lineheight{1.25}\smash{\begin{tabular}[t]{l}$v_{t_2}$\end{tabular}}}}%
    \put(0.79129171,0.63788809){\makebox(0,0)[lt]{\lineheight{1.25}\smash{\begin{tabular}[t]{l}$v_{t_1}$\end{tabular}}}}%
    \put(0.51296658,0.09900188){\makebox(0,0)[lt]{\lineheight{1.25}\smash{\begin{tabular}[t]{l}$v_{t_3}$\end{tabular}}}}%
    \put(0.60386721,0.79534912){\makebox(0,0)[lt]{\lineheight{1.25}\smash{\begin{tabular}[t]{l}$H_{t_2}$\end{tabular}}}}%
    \put(0.71749295,0.18177001){\makebox(0,0)[lt]{\lineheight{1.25}\smash{\begin{tabular}[t]{l}$H_{t_1}$\end{tabular}}}}%
    \put(0.2175396,0.36357122){\makebox(0,0)[lt]{\lineheight{1.25}\smash{\begin{tabular}[t]{l}$H_{t_3}$\end{tabular}}}}%
  \end{picture}%
\endgroup%

  \caption{The action of the standard subgroup $C(T)$ via a simplicial
    representation on $\P(V_T)$. Copies of the red hexagon $\Sigma_T$
    tile the minimal domain $P_{\min}$ (in pink) for this simplicial
    representation.}
  \label{fig:iiitri}
\end{figure}

In particular, since $C(D)$ and $C(T)$ commute, we know that
$\P(V_T) \subset H_D$, and therefore $\Sigma_T \subset F_D$. In
addition, since $t_1$ and $t_3$ commute with all of the generators in
$D \cup E$, the subspace $\P(V_{t_1, t_3})$ is contained in the
2-dimensional projective subspace $H_{D \cup E} = H_D \cap H_E$, which
means that the edge $\Sigma_{t_1, t_3}$ of the hexagon $\Sigma_T$ is
contained in $H_{D \cup E}$ (see \Cref{fig:hexagon}, left).

Since $\Sigma_T \subset F_D$, the relative interior of the orbit
$C(T) \cdot \Sigma_T$ is a $C(T)$-invariant convex subset $P_{\min}$
of the polytope $P$, contained in the $C(T)$-invariant subspace
$\P(V_T)$; it is a copy of the minimal $C(T)$-invariant domain for the
simplicial representation of $C(T)$ on $V_T$ (see \Cref{fig:hexagon},
right).

\begin{figure}[h]
  \centering
\begingroup%
  \makeatletter%
  \providecommand\color[2][]{%
    \errmessage{(Inkscape) Color is used for the text in Inkscape, but the package 'color.sty' is not loaded}%
    \renewcommand\color[2][]{}%
  }%
  \providecommand\transparent[1]{%
    \errmessage{(Inkscape) Transparency is used (non-zero) for the text in Inkscape, but the package 'transparent.sty' is not loaded}%
    \renewcommand\transparent[1]{}%
  }%
  \providecommand\rotatebox[2]{#2}%
  \newcommand*\fsize{\dimexpr\f@size pt\relax}%
  \newcommand*\lineheight[1]{\fontsize{\fsize}{#1\fsize}\selectfont}%
  \ifx\svgwidth\undefined%
    \setlength{\unitlength}{323.54269193bp}%
    \ifx\svgscale\undefined%
      \relax%
    \else%
      \setlength{\unitlength}{\unitlength * \real{\svgscale}}%
    \fi%
  \else%
    \setlength{\unitlength}{\svgwidth}%
  \fi%
  \global\let\svgwidth\undefined%
  \global\let\svgscale\undefined%
  \makeatother%
  \begin{picture}(1,0.41030504)%
    \lineheight{1}%
    \setlength\tabcolsep{0pt}%
    \put(0,0){\includegraphics[width=\unitlength,page=1]{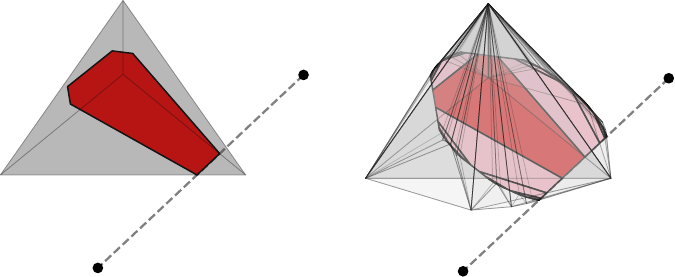}}%
    \put(0.12033233,0.03727127){\color[rgb]{0,0,0}\makebox(0,0)[lt]{\lineheight{1.25}\smash{\begin{tabular}[t]{l}$v_{t_3}$\end{tabular}}}}%
    \put(0.43124327,0.32355684){\color[rgb]{0,0,0}\makebox(0,0)[lt]{\lineheight{1.25}\smash{\begin{tabular}[t]{l}$v_{t_1}$\end{tabular}}}}%
  \end{picture}%
\endgroup%

  \caption{Left: the hexagon $\Sigma_T$, embedded in the tetrahedron
    $F_D$. The subset $\Sigma_{t_1, t_3}$ is an edge of the hexagon
    contained in the bottom face $F_{D \cup E}$ of $F_D$. Right: the
    convex subset $P_{\min}$ embedded in the polytope $P$.}
  \label{fig:hexagon}
\end{figure}

We now consider the geodesic word $w = t_1t_3t_2ed_1d_2$, and let
$\mc{W}$ be an $\vindomain$-itinerary traversing $w$, departing from
the identity. As in the previous example, since no pair of consecutive
generators in $w$ commutes, the itinerary $\mc{W}$ is efficient. Thus, writing $W_{\mathrm{Vin}}$ and $W'_{\mathrm{Vin}}$ for the first and last walls of $\mc{W}$,
$\elt(W_{\mathrm{Vin}}, W_{\mathrm{Vin}}') = t_1t_3t_2ed_1d_2$.

We let $R, R'$ denote the reflections in $C$ fixing
$W_{\mathrm{Vin}}, W_{\mathrm{Vin}}'$, and let $\Omega$ be an
arbitrary reflection domain for $C$. We let $W, W'$ denote the walls
in $\Omega$ fixed by $R, R'$, and let $W_{\min}, W_{\min}'$ denote the
walls in $\mindomain$ fixed by the same pair of reflections. By
\Cref{prop:vinberg_maximal} and \Cref{thm:minimal_domain}, we have
\[
  \mindomain \subseteq \Omega \subseteq \vindomain,
\]
and
\[
  W_{\min} \subseteq W \subseteq W_{\mathrm{Vin}}, \qquad W_{\min}'
  \subseteq W' \subseteq W_{\mathrm{Vin}}'.
\]

By \Cref{prop:wall_expression}, we have
$W_{\min}' = t_1t_3t_2ed_1 \cdot (H_{d_2} \cap \mindomain)$. Using
this, we show:
\begin{prop}
  \label{prop:wprime_intersect}
  The intersection $\overline{W_{\min}'} \cap \overline{P_{\min}}$
  contains $t_1t_3t_2 \cdot \Sigma_{t_1, t_3}$.
\end{prop}
\begin{proof}
  Since $P_{\min}$ is $C(T)$-invariant and $\mindomain$ is
  $C$-invariant, we just need to check that
  $\overline{P_{\min}} \cap ed_1 \cdot H_{d_2}$ contains
  $\Sigma_{t_1, t_3}$. Since $e$ does not commute with $d_1$ or $d_2$,
  we know that $ed_1 \cdot H_{d_2} \cap H_D = H_{D \cup E}$, and as
  $P_{\min} \subset H_D$ we therefore have
  \[
    \overline{P_{\min}} \cap ed_1 \cdot H_{d_2} = \overline{P_{\min}}
    \cap H_{D \cup E}.
  \]
  We have already seen that the edge $\Sigma_{t_1, t_3}$ of the
  hexagon $\Sigma_T$ is contained in $H_{D \cup E}$, which gives us
  the desired intersection since $C(T)$-translates of $\Sigma_T$ tile
  $P_{\min}$.
\end{proof}

Now, since the $\vindomain$-itinerary $\mc{W}$ departs from the
identity, we know that $\halfspace_+(W_{\mathrm{Vin}})$ contains
$\halfspace_+(W'_{\mathrm{Vin}})$ and thus
$\halfspace_+(W') \subset \halfspace_+(W)$. We consider the boundary
of the half-cone $\halfcone_+(W)$. Note that, since the entire
tetrahedron $F_D$ has infinite stabilizer in $C$, any point in
$\overline{W} \cap F_D$ must lie in $\dee W$. In particular, since
$W_{\min} \subset W$, the boundary $\dee W$ must contain
$\overline{W_{\min}} \cap F_D$, which is a reflection wall in the
2-dimensional domain $P_{\min}$.

As $P_{\min}$ is a reflection domain for $C(T)$ in the
$C(T)$-invariant subspace $\P(V_T)$,
\Cref{lem:halfspaces_in_halfcones} tells us that the closure of the
half-cone $\halfcone_+(\overline{W_{\min}} \cap P_{\min})$ on the
domain $P_\text{min}$ contains the component of
$P_{\min} \minus \overline{W_{\min}}$ which does not contain
$\Sigma_T$ (see \Cref{fig:pmin_halfcone}).

\begin{figure}[h]
  \centering
  \def\svgwidth{.7\textwidth}
\begingroup%
  \makeatletter%
  \providecommand\color[2][]{%
    \errmessage{(Inkscape) Color is used for the text in Inkscape, but the package 'color.sty' is not loaded}%
    \renewcommand\color[2][]{}%
  }%
  \providecommand\transparent[1]{%
    \errmessage{(Inkscape) Transparency is used (non-zero) for the text in Inkscape, but the package 'transparent.sty' is not loaded}%
    \renewcommand\transparent[1]{}%
  }%
  \providecommand\rotatebox[2]{#2}%
  \newcommand*\fsize{\dimexpr\f@size pt\relax}%
  \newcommand*\lineheight[1]{\fontsize{\fsize}{#1\fsize}\selectfont}%
  \ifx\svgwidth\undefined%
    \setlength{\unitlength}{568.79998779bp}%
    \ifx\svgscale\undefined%
      \relax%
    \else%
      \setlength{\unitlength}{\unitlength * \real{\svgscale}}%
    \fi%
  \else%
    \setlength{\unitlength}{\svgwidth}%
  \fi%
  \global\let\svgwidth\undefined%
  \global\let\svgscale\undefined%
  \makeatother%
  \begin{picture}(1,0.51265825)%
    \lineheight{1}%
    \setlength\tabcolsep{0pt}%
    \put(0,0){\includegraphics[width=\unitlength,page=1]{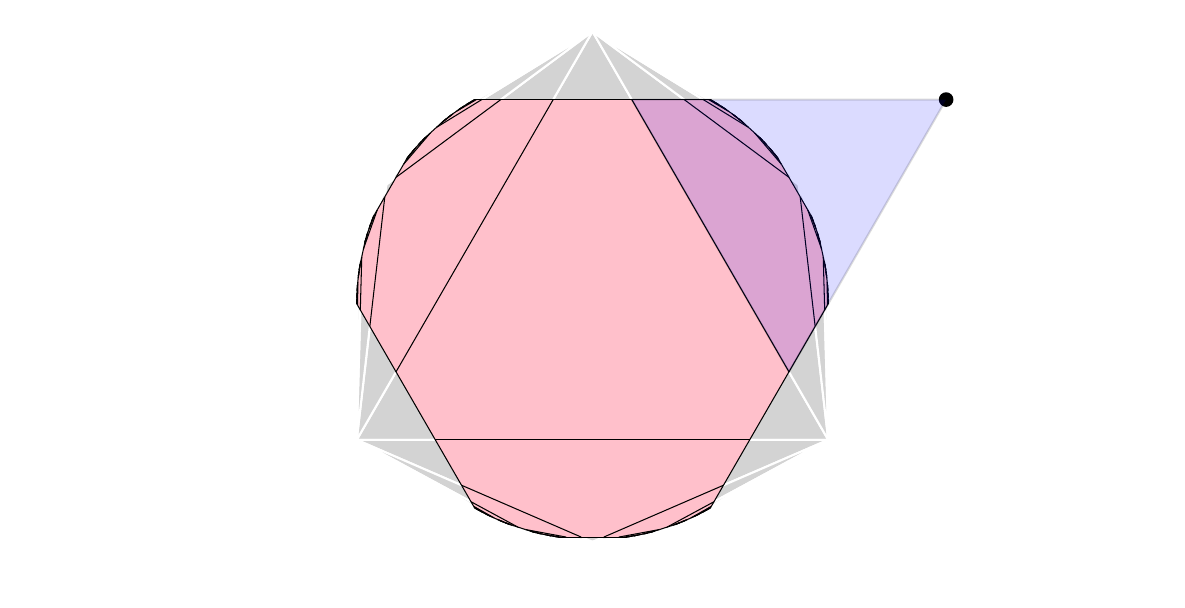}}%
    \put(0.79843466,0.45299756){\makebox(0,0)[lt]{\lineheight{1.25}\smash{\begin{tabular}[t]{l}$v_{t_1}$\end{tabular}}}}%
    \put(0.48781647,0.29287976){\color[rgb]{1,0,0}\makebox(0,0)[lt]{\lineheight{1.25}\smash{\begin{tabular}[t]{l}$\partial W_{\min}$\end{tabular}}}}%
    \put(0.68275319,0.34161394){\color[rgb]{1,0,0}\makebox(0,0)[lt]{\lineheight{1.25}\smash{\begin{tabular}[t]{l}$\partial W_{\min}'$\end{tabular}}}}%
    \put(0.42689876,0.23196203){\makebox(0,0)[lt]{\lineheight{1.25}\smash{\begin{tabular}[t]{l}$P_{\min}$\end{tabular}}}}%
    \put(0,0){\includegraphics[width=\unitlength,page=2]{iiitri_halfcone.pdf}}%
  \end{picture}%
\endgroup%

  \caption{The boundary of the half-cone over $W_{\min}$ intersects
    the boundary of the wall $W_{\min}'$.}
  \label{fig:pmin_halfcone}
\end{figure}

In particular, as $W_{\min}$ is a reflection wall for $t_1$, this
half-cone contains $t_1t_3t_2 \cdot \Sigma_{t_1, t_3}$. As
$\overline{W_{\min}} \cap P_{\min} \subset \dee W$, we have
\[
  \halfcone_+(\overline{W_{\min}} \cap P_{\min}) \subset \partial
  \halfcone_+(W),
\]
meaning that $\partial \halfcone_+(W)$ contains
$t_1t_3t_2 \cdot \Sigma_{t_1, t_3}$. Then, since
$W'_{\min} \subset W'$ we see from \Cref{prop:wprime_intersect} that
$\overline{W'} \cap \partial \halfcone_+(W)$ is nonempty, meaning that
the halfcones $\halfcone_+(W)$ and $\halfcone_+(W')$ cannot strongly
nest.

\begin{remark}
  As in \Cref{rem:long_word_fail}, we can apply a nearly identical
  argument to the word $(t_1t_3t_2)^ke(d_1d_2)^k$ for any given
  $k > 0$ to see that the group element $\elt(W, W')$ can also be made
  to lie arbitrarily far from any standard subgroup of $C$.
\end{remark}

\bibliography{specialbib}{}
\bibliographystyle{siam}

\end{document}